\documentclass[oneside,english]{amsart}
\usepackage[T1]{fontenc}
\usepackage[latin9]{inputenc}
\usepackage{mathrsfs}
\usepackage{amstext}
\usepackage{amsthm}
\usepackage{amssymb}
\makeindex

\makeatletter
\numberwithin{equation}{section}
\numberwithin{figure}{section}
  \theoremstyle{plain}
  \newtheorem*{thm*}{\protect\theoremname}
 \theoremstyle{definition}
 \newtheorem*{defn*}{\protect\definitionname}
\theoremstyle{plain}
\newtheorem{thm}{\protect\theoremname}[section]
  \theoremstyle{definition}
  \newtheorem{defn}[thm]{\protect\definitionname}
  \theoremstyle{definition}
  \newtheorem{example}[thm]{\protect\examplename}
  \theoremstyle{remark}
  \newtheorem{rem}[thm]{\protect\remarkname}
  \theoremstyle{plain}
  \newtheorem{lem}[thm]{\protect\lemmaname}
  \theoremstyle{plain}
  \newtheorem{prop}[thm]{\protect\propositionname}
  \theoremstyle{plain}
  \newtheorem{cor}[thm]{\protect\corollaryname}
  \theoremstyle{plain}
  \newtheorem{fact}[thm]{\protect\factname}

\makeatother

\usepackage{babel}
  \providecommand{\corollaryname}{Corollary}
  \providecommand{\definitionname}{Definition}
  \providecommand{\examplename}{Example}
  \providecommand{\factname}{Fact}
  \providecommand{\lemmaname}{Lemma}
  \providecommand{\propositionname}{Proposition}
  \providecommand{\remarkname}{Remark}
  \providecommand{\theoremname}{Theorem}
\providecommand{\theoremname}{Theorem}

\begin{document}

\title{Simultaneous monomialization.}

\author{Julie Decaup}
\begin{abstract}
We give a proof of the simultaneous monomialization Theorem in zero
characteristic for rings essentially of finite type over a field and
for quasi-excellent rings. The methods develop the key elements theory
that is a more subtle notion than the notion of key polynomials.
\end{abstract}

\address{Instituto de Matemáticas, Unidad Cuernavaca, Universidad Nacional
Autónoma de México.}

\email{julie.decaup@im.unam.mx}

\maketitle
\tableofcontents{}

\newpage{}

\part{Introduction.}

The resolution of singularities can be formulated in the following
way.

Let $V$ be a singular variety. The variety $V$ admits a resolution
of singularities if there exists a smooth variety $W$ and a proper
birational morphism $W\to V$. 

This problem has been solved in many cases but remains an open problem
in others. In characteristic zero Hironaka proved resolution of singularities
in all dimensions (\cite{H}) in 1964. Much work has been done since
1964 to simplify and better understand resolution of singularities
in characteristic zero. We mention \cite{BEV1}, \cite{B}, \cite{BDMV},
\cite{BGMW}, \cite{BM3}, \cite{BM4}, \cite{BM8}, \cite{BMT},
\cite{BV1}, \cite{BV2}, \cite{BrV1}, \cite{BrV2}, \cite{Cu1},
\cite{EV1}, \cite{Tem1}, \cite{Tem2}, \cite{Tem3}, \cite{Vi},
and \cite{Wl}. The problem remains open in positive characteristic.
The first proof for surfaces is due to S. Abhyankar in 1956 \cite{A1}
with subsequent strengthenings by H. Hironaka \cite{H3} and J. Lipman
\cite{L} to the case of more general 2-dimensional schemes, with
Lipman giving necessary and sufficient condition for a 2-dimensional
scheme to admit a resolution of singularities. See also \cite{CSch}.
Still, Abhyankar's proof is extremely technical and difficult and
comprises a total of 508 pages (\cite{A2}, \cite{A3}, \cite{A4},
\cite{A5}, \cite{A6}). For a more recent and more palatable proof
we refer the reader to \cite{Cu2}. It was not until much later that
V. Cossart and O. Piltant settled the problem of resolution of threefolds
in complete generality (their theorem holds for arbitrary quasi-excellent
noetherian schemes of dimension three, including the arithmetic case)
in a series of three long papers spanning the years 2008 to 2019 \cite{CP1},
\cite{CP2} and \cite{CP3}.\\
To try to solve the problem of resolution of singularities numerous
methods were introduced, in particular Zariski and Abhyankar used
the local uniformization. But it does not allow at the moment to solve
completely the problem. 

We are interested in a stronger problem than the local uniformization:
the monomialization problem. In this work we solve the monomialization
problem in characteristic zero. We hope that these methods, applicable
in positive characteristic, may help to attack the global problem
of resolution of singularities on a different point of view.

One of the essential tools to handle the monomialization or the local
uniformization is a valuation. Let us look on an example how valuations
naturally fit into the problem.

Let $V$ be a singular variety and $Z$ be an irreducible closed set
of $V$. 

If we knew how to resolve the singularities of $V$, we would have
a smooth variety $W$ and a proper birational morphism $W\to V$.
In $W$, we can consider a irreducible set $Z'$ whose image is $Z$.
And so the regular local ring $\mathcal{O}_{W,Z'}$ dominates the
non regular local ring $\mathcal{O}_{V,Z}$. It means that we have
an inclusion $\mathcal{O}_{V,Z}\subseteq\mathcal{O}_{W,Z'}$ and the
maximal ideal of $\mathcal{O}_{V,Z}$ is the intersection of those
of $\mathcal{O}_{W,Z'}$ with $\mathcal{O}_{V,Z}$. Up to a blow-up
$Z'$ is a hypersurface and so $\mathcal{O}_{W,Z'}$ is dominated
by a discrete valuation ring. In this case the valuation is the order
of vanishing along the hypersurface.

Before stating the local uniformization Theorem, we need a classical
notion that will be very important: the center of a valuation. For
details, we can read (\cite{Z2}) or (\cite[sections 2 and 3]{V}).

Let $K$ be a field and $\nu$ be a valuation defined over $K$. We
set 
\[
R_{\nu}:=\left\{ x\in K\text{ such that }\nu\left(x\right)\geq0\right\} ,
\]
 the valuation ring of $\nu$, and $\mathfrak{m}_{\nu}$ its maximal
ideal. 

We consider a subring $A$ of $K$ such that $A\subset R_{\nu}$.
Then the center of $\nu$ in $A$ is the ideal $\mathfrak{p}$ of
$A$ such that $\mathfrak{p}=A\cap\mathfrak{m}_{\nu}$.

Now we consider an algebraic variety $V$ over a field $k$ and $K$
its fractions field. Assume $V$ is an affine variety. Then $V=\mathrm{Spec}\left(A\right)$
where $A$ is a finite type integral $k$-algebra with $A\subseteq K$.
If $A\subseteq R_{\nu}$, then the center of $\nu$ over $V$ is the
point $\zeta$ of $V$ which corresponds to the prime ideal $A\cap\mathfrak{m}_{\nu}$
of $A$. 

The irreducible closed sub-scheme $Z$ of $V$ defined by $A\cap\mathfrak{m}_{\nu}$
(it means the image of the morphism $\mathrm{Spec}\left(\frac{A}{A\cap\mathfrak{m_{\nu}}}\right)\to\mathrm{Spec}\left(A\right)$)
has a generic point $\xi$. Equivalently $\zeta$ is the point associated
to the zero ideal. We say that $Z$ is the center of $\nu$ over $V$.
.\\
Now let us state the local uniformization Theorem. It has been proved
in characteristic zero but it is always a conjecture in positive characteristic.
\begin{thm*}[Zariski \cite{Z2}]
 Let $X=\mathrm{Spec}\left(A\right)$ be an affine variety of fractions
field $K$ over a field $k$. We consider $\nu$ a valuation over
$K$ of valuation ring $R_{\nu}$. 

Then $A$ can be embedded in a regular local sub-ring $A'$ essentially
of finite type over $k$ and dominated by $R_{\nu}$.
\end{thm*}
In this work we prove a stronger result: the simultaneous monomialization
Theorem. We are going to explain what is the monomialization and what
are the objects that we handle.

Let $k$ be a field of characteristic zero and $f\in k[u_{1},\dots,u_{n}]$
be a polynomial in $n$ variables, irreducible over $k$. We denote
by $V\left(f\right)$ the hypersurface defined by $f$ and we assume
that it has a singularity at the origin. Then we set $R:=k\left[u_{1},\dots,u_{n}\right]_{\left(u_{1},\dots,u_{n}\right)}$.
This is a regular local ring that is essentially of finite type over
the field $k$. The vector $u=\left(u_{1},\dots,u_{n}\right)$ is
a regular system of parameters of $R$. We use the notation $\left(R,u\right)$
to express the fact that $u$ is a regular system of parameters of
the regular local ring $R$.
\begin{defn*}[\ref{def:monom}]
The element $f$ is monomializable if there exists a map 
\[
\left(R,u\right)\to\left(R',u'=\left(u'_{1},\dots u'_{n}\right)\right)
\]
 that is a sequence of blow-ups such that the total transform of $f$
is a monomial. It means that in $R'$, the total transform of $f$
is $v\prod\limits _{i=1}^{n}\left(u'_{i}\right)^{\alpha_{i}}$, with
$v$ a unit of $R'$.
\end{defn*}
Now we can give a simplified version of one of the main theorems of
this work.
\begin{thm*}[\ref{thm:Monomialisons}]
 Let $\left(R,u\right)$ be a regular local ring that is essentially
of finite type over a field $k$ of characteristic zero.

Then there exists a countable sequence of blow-ups 
\[
\left(R,u\right)\to\cdots\to\left(R_{i},u^{\left(i\right)}\right)\to\cdots
\]

that monomializes simultaneously all the elements of $R$.
\end{thm*}
Equivalently, it means that for each element $f$ in $R$, there exists
an index $i$ such that in $R_{i}$, $f$ is one monomial.

If $f$ is an irreducible polynomial of $k\left[u_{1},\dots,u_{n}\right]$,
then $A:=\frac{R}{\left(f\right)}$ is a local domain. We can find
a valuation $\nu$ over $\mathrm{Frac}\left(A\right)$ centered in
$R$. One consequence of Theorem \ref{thm:Monomialisons} is that
the total transform of $f$ in one of the $R_{i}$ is $v\prod\limits _{j=1}^{n}\left(u_{j}^{\left(i\right)}\right)^{\alpha_{j}}$
. By the irreducibility of $f$ its strict transform is exactly $u_{n}^{\left(i\right)}$. 

Hence there exists an embedding of $A$ into the ring $A'=\frac{R_{i}}{\left(u_{n}^{\left(i\right)}\right)}$
which is dominated by $R_{\nu}$. So a consequence of Theorem \ref{thm:Monomialisons}
is the Local Uniformization Theorem as announced.

And we obtain a stronger result here: the total transform is a normal
crossing divisor. We call this result the embedded local uniformization.
We will give a new proof of this theorem in this work.

Let us explain why simultaneous monomialization is a stronger result
than the embedded local uniformization Theorem. First we monomialize
all the elements of $R$ with the same sequence of blow-ups. Secondly,
this sequence is effective and at each step of the process we can
express the $u^{(i+1)}$ in terms of the $u^{(i)}$. Indeed, we consider
an essentially of finite type regular local ring $R$, and a valuation
centered in $R$. Thanks to this valuation we construct an effective
sequence of blow-ups that monomializes all the elements of $R$. One
more advantage of the proof we give here is that in the essentially
of finite type case, we prove the simultaneous embedded local uniformization
whatever is the valuation. In particular we do not need any hypothesis
on the rank of the valuation. 

One of the most important ingredient in the proof of this theorem
is the notion of key polynomial. We give here a new definition of
key polynomial, introduced by Spivakovsky and appearing for the first
time in (\cite{DMS} and \cite{NS}). Let $K$ be a field, $\nu$
be a valuation over $K$ and we denote by $\partial_{b}:=\frac{1}{b!}\frac{\partial^{b}}{\partial X^{b}}$
the formal derivative of the order $b$ on $K\left[X\right]$. For
every polynomial $P\in K[X]$, we set 
\[
\epsilon_{\nu}\left(P\right):=\max\limits _{b\in\mathbb{N^{\ast}}}\left\{ \frac{\nu\left(P\right)-\nu\left(\partial_{b}P\right)}{b}\right\} .
\]

\begin{defn*}[\ref{def:pc}]
Let $Q\in K[X]$ be a monic polynomial. The polynomial $Q$ is a
\emph{key polynomial for $\nu$ }if for every polynomial $P\in K[X]$:
\[
\epsilon_{\nu}\left(P\right)\geq\epsilon_{\nu}(Q)\Rightarrow\mathrm{deg_{X}}(P)\geq\mathrm{deg_{X}}(Q).
\]
\end{defn*}
One of the interests of this new definition is the following notion:
\begin{defn*}[\ref{def: successeur immediat}]
Let $Q_{1}$ and $Q_{2}$ be two key polynomials. We say that $Q_{2}$
is an \emph{immediate successor} of $Q_{1}$ if $\epsilon(Q_{1})<\epsilon(Q_{2})$
and if $Q_{2}$ is of minimal degree for this property. We denote
this by $Q_{1}<Q_{2}$.
\end{defn*}
We denote by $M_{Q_{1}}$ the set of immediate successors of $Q_{1}$.
We assume that they all have the same degree as $Q_{1}$ and that
$\epsilon\left(M_{Q_{1}}\right)$ does not have any maximal element.
\begin{defn*}[\ref{def:limitpc}]
We assume that there exists a key polynomial $Q'$ such that $\epsilon(Q')>\epsilon(M_{Q_{1}})$.
We call \emph{immediate limit successor }of $Q_{1}$ every polynomial
$Q_{2}$ of minimal degree satisfying $\epsilon(Q_{2})>\epsilon(M_{Q_{1}})$,
and we denote this by $Q_{1}<_{\lim}Q_{2}$.
\end{defn*}
Let $Q_{1}$ and $Q_{2}$ be two key polynomials. Let us write $Q_{2}$
according to the powers of $Q_{1}$, $Q_{2}=\sum\limits _{i=0}^{s}q_{i}Q_{1}^{i}$
where the $q_{i}$ are polynomials of degree strictly less than $Q_{1}$.
We call this expression the\emph{ $Q_{1}$-expansion of $Q_{2}$.}

An important result in this work, and the only one for which we need
the characteristic zero hypothesis, is the following Theorem.
\begin{thm*}[\ref{thm: delta et pc limite}]
 Let $Q_{2}$ be an immediate limit successor of $Q_{1}$. Then the
terms of the $Q_{1}$-expansion of $Q_{2}$ that minimize the valuation
are exactly those of degrees $0$ and $1$.
\end{thm*}
Then the hypothesis of characteristic zero is necessary also for the
results that follow from this theorem.\\

Here we give an idea of our proof of Theorem \ref{thm:Monomialisons}.
Let us consider a regular local ring $R$ essentially of finite type
over a field $k$ of characteristic zero. We fix $u=\left(u_{1},\dots,u_{n}\right)$
a regular system of parameters of $R$.

The first ingredient in the proof is the notion of non degeneration. 
\begin{defn*}[\ref{def:nond=0000E9g=0000E9}]
We say that an element $f$ of $R$ is non degenerated with respect
to $u$ if there exists an ideal $N$ of $R$, generated by monomials
in $u$, such that $\nu\left(f\right)=\min\limits _{x\in N}\left\{ \nu\left(x\right)\right\} $.
\end{defn*}
The first step is to monomialize all the elements that are non degenerated
with respect to a regular system of parameters of $R$. So let $f$
be an element of $R$ that is non degenerated with respect to $u$.
We construct a sequence of blow-ups 
\[
\left(R,u\right)\to\cdots\to\left(R',u'\right)
\]
 such that the strict transform of $f$ in $R'$ is a monomial in
$u'$. 

There exist elements $f$ of $R$ that are not non degenerated with
respect to $u$. So we wonder if we could find a sequence of blow-ups
\[
\left(R,u\right)\to\cdots\to\left(T,t\right)
\]
such that $f$ is non degenerated with respect to $t$. If we can,
after a new sequence of blow-ups, we monomialize $f$. Doing this
for all the elements of $R$ would be too complicated. So we would
want to find a sequence of blow-ups $\left(R,u\right)\to\cdots\to\left(R',u'\right)$
such that all the elements of $R$ are non degenerated with respect
to $u'$. It is a little optimistic and we need to do something more
subtle. We will find an infinite sequence of blow-ups 
\[
\left(R,u\right)\to\left(R_{1},u^{(1)}\right)\to\cdots\to\left(R_{i},u^{(i)}\right)\to\cdots
\]
 such that for each element $f$ of $R$, there exists $i$ such that
$f$ is non degenerated with respect to $u^{(i)}$.

For this, we need the second main ingredient: the key polynomials.

We construct a sequence of key polynomials $\left(Q_{i}\right)_{i}$
such that each element $f$ of $R$ is non degenerated with respect
to some $Q_{i}$. It means that:
\[
\forall f\in R\text{, }\exists i\text{ such that }\nu\left(f\right)=\nu_{Q_{i}}\left(f\right).
\]
We construct the sequence $\left(Q_{i}\right)_{i}$ step by step.
We require the following properties for this sequence: for every index
$i$, the polynomial $Q_{i+1}$ is an (eventually limit) immediate
successor of $Q_{i}$. Furthermore the sequence $\left(\epsilon\left(Q_{i}\right)\right)_{i}$
is cofinal in $\epsilon\left(\Lambda\right)$ where $\Lambda$ is
the set of key polynomials of the extension $k\left(u_{1},\dots,u_{n-1}\right)\left(u_{n}\right)$.\\

Equivalently it means:

\[
\begin{cases}
\forall i\text{, }Q_{i}<Q_{i+1}\text{ or }Q_{i}<_{\lim}Q_{i+1},\\
\forall Q\in\Lambda\text{ }\exists i\text{ such that }\epsilon(Q_{i})\geq\epsilon(Q).
\end{cases}
\]

Assume now that we can construct a sequence of blow-ups 
\[
\left(R,u\right)\to\cdots\to\left(R_{j},u^{(j)}\right)\to\cdots
\]
 such that all the $Q_{i}$ belong to a regular system of parameters.
It means that 
\[
\forall i\text{, }\exists j,k\text{ such that }Q_{i}^{\mathrm{strict},j}=u_{k}^{(j)},
\]
 where $Q_{i}^{\mathrm{strict},j}$ is the strict transform of $Q_{i}$
in $R_{j}$. Then every element $f$ of $R$ which is non degenerated
with respect to $Q_{i}$ is non degenerated with respect to $u^{(j)}$.
Thus it is monomializable. So the next step is to monomialize all
the $Q_{i}$. 

In order to do this once again we have to be subtle. The notion of
key polynomial is not stable by blow-up, so we need a better notion:
the notion of key element. Let $\left(Q_{i},Q_{i+1}\right)$ a couple
of (eventually limit) immediate successors of our sequence. We consider
$Q_{i+1}=\sum\limits _{j=0}^{s}q_{j}Q_{i}^{j}$ the $Q_{i}$-expansion
of $Q_{i+1}$. Then we associate to $Q_{i+1}$ a key element $Q'_{i+1}$
defined as follows.
\begin{defn*}[\ref{def:elmtclef}]
An element $Q'_{i+1}=\sum\limits _{j=0}^{s}a_{j}q_{j}Q_{i}^{j}$
where the $a_{j}$ are units is called \emph{a key element} associated
to $Q_{i+1}$.
\end{defn*}
In fact we also have a notion of (eventually limit) immediate successors
in this case.
\begin{defn*}[\ref{def:elmtsuccimm} and \ref{def:limitelmt}]
Let $P_{1}'$ and $P_{2}'$ be two key elements. We say that $P_{1}'$
and $P_{2}'$ are \emph{(eventually limit) immediate successors key
elements} if their respective associated key polynomials $P_{1}$
and $P_{2}$ are such that $P_{1}<P_{2}$ (eventually $P_{1}<_{\mathrm{lim}}P_{2}$).
\end{defn*}
After some blow ups we prove that (eventually limit) immediate successors
become (eventually limit) immediate successors key elements. So we
monomialize these key elements. For this we construct a sequence of
blow-ups
\[
\left(R,u\right)\overset{}{\to}\cdots\overset{}{\to}\left(R_{s},u^{(s)}\right)\overset{}{\to}\cdots
\]

that monomializes all the key polynomials $Q_{i}$. More precisely,
for every index $i$ there exists an index $s_{i}$ such that in $R_{s_{i}}$,
$Q_{i}$ is a monomial in $u^{\left(s_{i}\right)}$ up to a unit of
$R_{s_{i}}$.

So in the case of essentially of finite type regular local rings,
no matter the rank of the valuation is, we prove the embedded local
uniformization Theorem. And we do this using only a sequence of blow-ups
for all the elements of the ring, and in an effective way. It means
that every blow-up is effective and we know how to express all the
systems of coordinates.

~

Then we want to prove the same kind of result over more general rings,
even if it means adding conditions on the valuation. We work with
quasi excellent rings. Indeed, Grothendieck and Nagata showed that
there is no resolution of singularities for rings that are not quasi
excellent.

The second main result of this paper can be express in the following
simplified form.
\begin{thm*}[\ref{thm:BAMDANSTAFACE}]
Let $R$ be a noetherian quasi excellent complete regular local ring
and $\nu$ be a valuation centered in $R$.

Assume that $\nu$ is of rank $1$, or of rank $2$ but composed with
a discrete valuation, and that $\mathrm{car}\left(k_{\nu}\right)=0$. 

There exists a countable sequence of blow-ups 
\[
\left(R,u\right)\to\dots\to\left(R_{l},u^{\left(l\right)}\right)\to\dots
\]

that monomializes all the element of $R$.
\end{thm*}
So let $R$ be a quasi excellent local domain. This time $R$ is not
assume to be of finite type, so we cannot repeat what we did before.
We need to introduce one more ingredient: the implicit prime ideal.

Let $\nu$ be a valuation of the fractions field of $R$ centered
in $R$. We call implicit prime ideal of $R$ associated to $\nu$
the ideal of the completion $\widehat{R}$ of $R$ defined by:

\[
H:=\bigcap\limits _{\beta\in\nu\left(R\setminus\left\{ 0\right\} \right)}P_{\beta}\widehat{R}
\]
 where $P_{\beta}:=\left\{ f\in R\text{ such that }\nu\left(f\right)\geq\beta\right\} $. 

One can show that in this case desingularizing $R$ means desingularizing
$\widehat{R}$. In the last part of this work we also prove that to
desingularize $\widehat{R}$, we only need to desingularize $\widehat{R}_{H}$
and (up to one more sequence of blow-ups) $\frac{\widehat{R}}{H}$.
We prove that the implicit prime ideal satisfies the property that
$\widehat{R}_{H}$ is regular. So we only have to desingularize $\frac{\widehat{R}}{H}$
and this is done by Theorem \ref{thm: quotient reg}.

~\\

\subsection*{Acknowledgments.}

The author is really grateful to her PHD advisor Mark Spivakovsky
for all the helpful discussions.

\newpage{}

\part{Key polynomials.}

The notion of key polynomials was first introduced by Saunders Mac
Lane in 1936, in the case of discrete valuations of rank 1. The first
motivation to introduce this notion was to describe all the extensions
of a valuation to a field extension. Let $K\to L$ be an extension
of field and $\nu$ a valuation on $K$. We consider a valuation $\mu$
that extends $\nu$ to $L$. In the case where $\nu$ is of rank $1$
and where $L$ is a simple algebraic extension of $K$, Mac Lane created
the notion of key polynomial for $\mu$. He also created the notion
of augmented valuations. Given a valuation $\mu$ and $Q$ a key polynomial
of Mac Lane, we write $f=\sum\limits _{i=0}^{r}f_{j}Q^{j}$ the $Q$-expansion
of an element $f\in K\left[X\right]$. An augmented valuation $\mu'$
of $\mu$ is the one defined by $\mu'\left(f\right)=\min\limits _{0\leq j\leq r}\left\{ \mu\left(f_{j}\right)+j\delta\right\} $
where $\delta>\mu\left(Q\right)$. He proved that $\mu$ is the limit
of a family of augmented valuations over the ring $K[x]$. Michel
Vaquié extended this definition to arbitrary valued field $K$, without
assuming that $\nu$ is discrete. The most important difference between
these notions is the fact that those of Vaquié involves limit key
polynomials while those of Mac Lane not. 

More recently, the notion of key polynomials has been used by Spivakovsky
to study the local uniformization problem, and to do this he created
a new notion of key polynomials. It is the one we use here.

\section{Key polynomials of Spivakovsky et al.}

For some results of this part, we refer the reader to \cite{DMS},
but we recall the definitions and properties used in this work to
have a selfcontained manuscript.

First, recall the definition of a valuation.
\begin{defn}
Let $R$ be a commutative domain with a unit element, $K$ be a commutative
field and $\Gamma$ be a totally ordered abelian group. We set $\Gamma_{\infty}:=\Gamma\cup\left\{ +\infty\right\} $.

A \emph{valuation} \index{Valuation} of $R$ is a map 
\[
\nu\colon R\to\Gamma_{\infty}
\]
 such that:
\begin{enumerate}
\item $\forall x\in R$, $\nu(x)=+\infty\Leftrightarrow x=0$,
\item $\forall\left(x,y\right)\in R^{2}$, $\nu\left(xy\right)=\nu\left(x\right)+\nu\left(y\right)$,
\item $\forall\left(x,y\right)\in R^{2}$, $\nu\left(x+y\right)\geq\min\left\{ \nu\left(x\right),\nu\left(y\right)\right\} $.
\end{enumerate}
\end{defn}

Let us give three examples of valuations.
\begin{example}
\label{exa:valuation1}

The map $\nu_{1}:\mathbb{C}\left[x\right]\to\mathbb{Z}\cup\left\{ +\infty\right\} $
which sends a polynomial $P=\sum\limits _{i=0}^{d}p_{i}x^{i}$ to
$\min\left\{ i\text{ such that }p_{i}\neq0\right\} $ is a valuation.
\end{example}

~
\begin{example}
\label{exa:valuation2}We want to define a valuation $\nu_{2}$ on
$\mathbb{C}\left(x,y,z\right)$. The value of a quotient $\frac{P}{Q}$
is $\nu_{2}\left(P\right)-\nu_{2}\left(Q\right)$. 

And we define the value of a polynomial $P=\sum\limits _{i}p_{i}x^{i_{1}}y^{i_{2}}z^{i_{3}}$
as the minimal of the values of $p_{i}x^{i_{1}}y^{i_{2}}z^{i_{3}}$.

Then we only have to define the values of the generators $x$, $y$
and $z$.

Hence the map $\nu_{2}:\mathbb{C}\left(x,y,z\right)\to\mathbb{R}_{\infty}$
which sends $x$ to $1$, $y$ to $2\pi$ and $z$ to $1+\pi$ is
a valuation.
\end{example}

~
\begin{example}
\label{exa:valuation3}Let us set $Q=z^{2}-x^{2}y$. Every polynomial
$P\in\mathbb{C}\left[x,y,z\right]$ can be written according to the
powers of $Q$. We write $P=\sum\limits _{i}p_{i}Q^{i}$ with the
$p_{i}\in\mathbb{C}\left[x,y\right]\left[z\right]$ of degree in $z$
strictly less than $\deg_{z}\left(Q\right)=2$. Assume that the first
non zero $p_{i}$ is $p_{n}$.\\
Then the map $\nu_{3}:\mathbb{C}\left(x,y,z\right)\to\left(\mathbb{R}^{2},\mathrm{lex}\right)$
which sends $P$ to $\left(n,\nu_{2}\left(p_{n}\right)\right)$ defines
a valuation, with $\nu_{2}$ the valuation defined in Example \ref{exa:valuation2}.
\end{example}

Let $K$ be a field equipped with a valuation $\nu$ and consider
a simple transcendental extension
\[
K\hookrightarrow K(X)
\]
with a valuation $\nu$ that extends $\mu$ to $K(X)$. We still denote
by $\nu$ the restriction of $\nu$ to $K[X]$.

For every non zero integer $b,$ we set $\partial_{b}:=\frac{1}{b!}\frac{\partial^{b}}{\partial X^{b}}$.
This is called the formal derivative of order $b$.

For every polynomial $P\in K[X]$, we set
\[
\epsilon_{\nu}\left(P\right):=\max\limits _{b\in\mathbb{N^{\ast}}}\left\{ \frac{\nu\left(P\right)-\nu\left(\partial_{b}P\right)}{b}\right\} .
\]
\index{epsilon_{nu}left(Pright)@$\epsilon_{\nu}\left(P\right)$}
\begin{rem}
Most of the time we will note $\epsilon\left(P\right):=\epsilon_{\nu}\left(P\right)$.
\end{rem}

\begin{example}
\label{exa:calcul epsilon}We consider $\mathbb{C}\left(x,y\right)\left[z\right]$
and the valuation $\nu:=\nu_{3}$ defined in Example \ref{exa:valuation3}.

We have $\nu\left(z\right)=\left(0,1+\pi\right)$ and $\nu\left(\partial z\right)=\nu\left(1\right)=\left(0,0\right)$.
So 
\[
\epsilon\left(z\right)=\max\limits _{b\in\mathbb{N^{\ast}}}\left\{ \frac{\nu\left(z\right)-\nu\left(\partial_{b}z\right)}{b}\right\} =\frac{\nu\left(z\right)-\nu\left(\partial z\right)}{1}=\nu\left(z\right)=\left(0,1+\pi\right).
\]

Also we have $\nu\left(x\right)=\left(0,1\right)$ and $\nu\left(\partial x\right)=\nu\left(0\right)=\left(+\infty,+\infty\right)$
so $\epsilon\left(x\right)=\left(-\infty,-\infty\right)$. Furthermore
$\epsilon\left(y\right)=\left(-\infty,-\infty\right)$.

Finally, let us compute $\epsilon\left(Q\right)$. Recall that $Q=z^{2}-x^{2}y$.
We have $\nu\left(Q\right)=\left(1,0\right)$, $\nu\left(\partial Q\right)=\nu\left(2z\right)=\left(0,1+\pi\right)$
and $\nu\left(\partial_{2}Q\right)=\nu\left(2\right)=\left(0,0\right)$.

So $\epsilon\left(Q\right)=\max\left\{ \frac{\nu\left(Q\right)-\nu\left(\partial Q\right)}{1},\frac{\nu\left(Q\right)-\nu\left(\partial_{2}Q\right)}{2}\right\} =\max\left\{ \frac{\left(1,0\right)-\left(0,1+\pi\right)}{1},\frac{\left(1,0\right)-\left(0,0\right)}{2}\right\} =\left(1,-1-\pi\right)$.
\end{example}

\begin{defn}
\label{def:pc}Let $Q\in K[X]$ be a monic polynomial. We say that
$Q$ is a \emph{key polynomial\index{Key polynomial} for $\nu$}
if for every polynomial $P\in K[X]$, we have:

\[
\epsilon_{\nu}\left(P\right)\geq\epsilon_{\nu}(Q)\Rightarrow\mathrm{deg_{X}}(P)\geq\mathrm{deg_{X}}(Q).
\]
\end{defn}

\begin{example}
\label{exa:pc}We consider the same example as in example \ref{exa:calcul epsilon}.

Let us show that $z$ is a key polynomial. We do a proof by contrapositive.
Let $P$ be a polynomial of degree in $z$ strictly less than $\deg_{z}\left(z\right)=1$.
So $P$ does not depend on $z$. Then we saw that $\epsilon\left(P\right)=\left(-\infty,-\infty\right)$.
So $\epsilon\left(P\right)<\epsilon\left(z\right)$ and $z$ is a
key polynomial.

Now, let us show that $Q=z^{2}-x^{2}z$ is a key polynomial. So we
consider a polynomial $P$ such that $\epsilon\left(P\right)\geq\epsilon\left(Q\right)=\left(1,-1-\pi\right)$.

Then $\epsilon\left(P\right)=\left(n,\ast\right)$ where $n\geq1$
and $\ast$ is a scalar. So $\nu\left(P\right)=\left(m,\ast\right)$
where $m\geq1$. Hence $Q^{m}\mid P$ and so $\deg_{z}\left(P\right)\geq\deg_{z}\left(Q\right)$.
We proved that $Q$ is a key polynomial.

We have two key polynomials $z$ and $Q$ and we have $\epsilon\left(z\right)<\epsilon\left(Q\right)$.
One can show that $Q$ is of minimal degree for this property. In
this situation we will say that $Q$ is an immediate successor of
$z$.
\end{example}

For every polynomial $P\in K[X]$, we set 
\[
b_{\nu}\left(P\right):=\min I(P)
\]

where 
\[
I(P):=\left\{ b\in\mathbb{N}^{\ast}\text{ such that }\frac{\nu\left(P\right)-\nu\left(\partial_{b}P\right)}{b}=\epsilon_{\nu}\left(P\right)\right\} .
\]
Again, if there is no confusion, we will omit the index $\nu$.

Let $P$ and $Q$ be two polynomials such that $Q$ is monic. Then
$P$ can be written $\sum\limits _{j=1}^{n}p_{j}Q^{j}$ with $p_{j}$
polynomials of degree strictly less than the degree of $Q$. This
expression is unique and is called the $Q$-expansion of $P$.
\begin{defn}
Let $\left(P,Q\right)\in K[X]^{2}$ such that $Q$ is monic, and we
consider $P=\sum\limits _{j=1}^{n}p_{j}Q^{j}$ the $Q$-expansion
of the polynomial $P$. Then we set $\nu_{Q}\left(P\right):=\min\limits _{0\leq j\leq n}\nu\left(p_{j}Q^{j}\right)$.
The map $\nu_{Q}$ is called the \emph{$Q$-truncation\index{Truncated valuation}}
of $\nu$.

Also we set 
\[
S_{Q}\left(P\right):=\left\{ j\in\left\{ 0,\dots,n\right\} \text{ such that }\nu\left(p_{j}Q^{j}\right)=\nu_{Q}\left(P\right)\right\} 
\]
 and 
\[
\delta_{Q}\left(P\right):=\max\left\{ S_{Q}\left(P\right)\right\} .
\]
Now, we set 
\[
\widetilde{P}_{\nu,Q}:=\sum\limits _{j\in S_{Q}\left(P\right)}p_{j}Q^{j}.
\]
\end{defn}

\begin{rem}
\label{rem:valuation tronquee non valuation}In the general case,
$\nu_{Q}$ is not a valuation. But if $Q$ is a key polynomial, we
are going to show that $\nu_{Q}$ is a valuation.

In order to do that, we need the next result, which will also be needed
for a proof of the fundamental theorem \ref{thm: delta et pc limite}.
\end{rem}

\begin{lem}
\label{lem:paslechoixfaut le mettre}Let $t\in\mathbb{N}_{>1}$ and
$Q$ be a key polynomial. We consider $P_{1},\dots,P_{t}$ some polynomials
of $K[X]$ all of degree strictly less than $\deg\left(Q\right)$
and we set $\prod\limits _{i=1}^{t}P_{i}:=qQ+r$ the Euclidean division
of $\prod\limits _{i=1}^{t}P_{i}$ by $Q$ in $K[X]$. Then:
\[
\nu\left(r\right)=\nu\left(\prod\limits _{i=1}^{t}P_{i}\right)<\nu\left(qQ\right).
\]
\end{lem}

\begin{proof}
We use induction on $t$.

Base of the induction: $t=2$. So we want to show that $\nu\left(P_{1}P_{2}\right)<\nu\left(qQ\right)$. 

Indeed, if $\nu\left(P_{1}P_{2}\right)<\nu\left(qQ\right)$, then
\[
\begin{array}{ccc}
\nu(r) & = & \nu\left(P_{1}P_{2}-qQ\right)\\
 & = & \nu\left(P_{1}P_{2}\right)\\
 & < & \nu\left(qQ\right)
\end{array}
\]
and we have the result. 

Assume, aiming for contradiction, that $\nu\left(P_{1}P_{2}\right)\geq\nu\left(qQ\right)$
and so $\nu\left(r\right)\geq\nu\left(qQ\right)$. Since $Q$ is a
key polynomial, every polynomial $P$ of degree strictly less than
$\deg\left(Q\right)$ satisfies $\epsilon\left(P\right)<\epsilon\left(Q\right)$.
In particular, for every non-zero integer $j$, we have $\nu\left(P\right)-\nu\left(\partial_{j}P\right)<j\epsilon\left(Q\right)$.
So it is the case for $P_{1}$, $P_{2}$ and $r$. Since $P_{1}$
and $P_{2}$ are of degree strictly less than $\deg\left(Q\right)$,
we have
\[
\begin{array}{ccc}
\deg_{X}\left(P_{1}P_{2}\right) & = & \deg_{X}\left(P_{1}\right)+\deg_{X}\left(P_{2}\right)\\
 & < & 2\deg_{X}\left(Q\right).
\end{array}
\]
However, $\deg_{X}\left(P_{1}P_{2}\right)=\deg_{X}\left(qQ\right)=\deg_{X}\left(q\right)+\deg_{X}\left(Q\right)$.
So $q$ is of degree strictly less than $\deg\left(Q\right)$ too,
and then $q$ satisfies, for every non-zero integer $j$: $\nu\left(q\right)-\nu\left(\partial_{j}q\right)<j\epsilon\left(Q\right)$.
We are going to compute $\nu\left(\partial_{b(Q)}\left(qQ\right)\right)$
in two different ways to get a contradiction.

First, 
\[
\nu\left(\partial_{b(Q)}\left(qQ\right)\right)=\nu\left(\sum\limits _{j=0}^{b(Q)}\left(\partial_{b(Q)-j}\left(Q\right)\partial_{j}\left(q\right)\right)\right).
\]
 Look at the first term of the sum: $q\partial_{b(Q)}\left(Q\right)$,
and compute its value $\nu\left(q\partial_{b(Q)}\left(Q\right)\right)$.
We are going to show that its value is the smallest of the sum.

We have 
\[
\begin{array}{ccc}
\nu\left(q\partial_{b(Q)}\left(Q\right)\right) & = & \nu\left(q\right)+\nu\left(\partial_{b(Q)}\left(Q\right)\right)\\
 & = & \nu\left(q\right)+\nu\left(Q\right)-b\left(Q\right)\epsilon\left(Q\right)
\end{array}
\]
 by definition of $b\left(Q\right)$. But we know that for every non-zero
integer $j$, we have $\nu\left(q\right)<j\epsilon\left(Q\right)+\nu\left(\partial_{j}q\right)$,
so 
\[
\begin{array}{ccc}
\nu\left(q\partial_{b(Q)}\left(Q\right)\right) & < & \left(j-b\left(Q\right)\right)\epsilon\left(Q\right)+\nu\left(Q\right)+\nu\left(\partial_{j}q\right)\\
 & \leq & \nu\left(\partial_{j}q\right)+\nu\left(\partial_{b\left(Q\right)-j}Q\right).
\end{array}
\]

Then $q\partial_{b(Q)}\left(Q\right)$ is the term of smallest value
in the sum. In particular, 

\begin{equation}
\begin{array}{ccc}
\nu\left(\partial_{b(Q)}\left(qQ\right)\right) & = & \nu\left(q\partial_{b(Q)}\left(Q\right)\right)\\
 & = & \nu\left(q\right)+\nu\left(\partial_{b\left(Q\right)}\left(Q\right)\right)\\
 & = & \nu\left(qQ\right)-b\left(Q\right)\epsilon\left(Q\right).
\end{array}\label{eq:0}
\end{equation}

Now we compute this value in a different way. We have:

\[
\begin{array}{ccc}
\nu\left(\partial_{b(Q)}\left(qQ\right)\right) & = & \nu\left(\partial_{b\left(Q\right)}\left(P_{1}P_{2}-r\right)\right)\\
 & = & \nu\left(\partial_{b\left(Q\right)}\left(P_{1}P_{2}\right)-\partial_{b\left(Q\right)}\left(r\right)\right)\\
 & \geq & \min\left\{ \nu\left(\partial_{b\left(Q\right)}\left(P_{1}P_{2}\right)\right),\nu\left(\partial_{b\left(Q\right)}\left(r\right)\right)\right\} .
\end{array}
\]

But also:
\[
\begin{array}{ccc}
\nu\left(\partial_{b\left(Q\right)}\left(P_{1}P_{2}\right)\right) & = & \nu\left(\sum\limits _{j=0}^{b\left(Q\right)}\partial_{j}\left(P_{1}\right)\partial_{b\left(Q\right)-j}\left(P_{2}\right)\right)\\
 & \geq & \min\limits _{0\leq j\leq b\left(Q\right)}\left\{ \nu\left(\partial_{j}P_{1}\right)+\nu\left(\partial_{b\left(Q\right)-j}\left(P_{2}\right)\right)\right\} .
\end{array}
\]

If $j\neq0$, we have $\nu\left(P_{1}\right)<j\epsilon\left(Q\right)+\nu\left(\partial_{j}\left(P_{1}\right)\right)$
and so 
\[
\nu\left(\partial_{j}\left(P_{1}\right)\right)>\nu\left(P_{1}\right)-j\epsilon\left(Q\right)
\]
 since $\deg_{X}\left(P_{1}\right)<\deg_{X}\left(Q\right)$. If $0\leq j<b\left(Q\right)$,
we also have
\[
\nu\left(\partial_{b\left(Q\right)-j}\left(P_{2}\right)\right)>\nu\left(P_{2}\right)-\left(b\left(Q\right)-j\right)\epsilon\left(Q\right).
\]
So if $0<j<b\left(Q\right)$, we have 
\[
\nu\left(\partial_{j}P_{1}\right)+\nu\left(\partial_{b\left(Q\right)-j}\left(P_{2}\right)\right)>\nu\left(P_{1}P_{2}\right)-b\left(Q\right)\epsilon\left(Q\right).
\]
This inequality stays true if $j=0$ and $j=b\left(Q\right)$, so:

\[
\nu\left(\partial_{b\left(Q\right)}\left(P_{1}P_{2}\right)\right)>\nu\left(P_{1}P_{2}\right)-b\left(Q\right)\epsilon\left(Q\right).
\]

By hypothesis, $\nu\left(P_{1}P_{2}\right)\geq\nu\left(qQ\right)$,
so 
\[
\nu\left(\partial_{b\left(Q\right)}\left(P_{1}P_{2}\right)\right)>\nu\left(qQ\right)-b\left(Q\right)\epsilon\left(Q\right).
\]

But since $r$ is of degree strictly less than $\deg\left(Q\right)$,
we know that $\nu\left(\partial_{b\left(Q\right)}\left(r\right)\right)>\nu\left(r\right)-b\left(Q\right)\epsilon\left(Q\right)$,
and by hypothesis $\nu\left(r\right)\geq\nu\left(qQ\right)$. Then
$\nu\left(\partial_{b\left(Q\right)}\left(r\right)\right)>\nu\left(qQ\right)-b\left(Q\right)\epsilon\left(Q\right)$.

So
\[
\begin{array}{ccc}
\nu\left(\partial_{b(Q)}\left(qQ\right)\right) & \geq & \min\left\{ \nu\left(\partial_{b\left(Q\right)}\left(P_{1}P_{2}\right)\right),\nu\left(\partial_{b\left(Q\right)}\left(r\right)\right)\right\} \\
 & > & \nu\left(qQ\right)-b\left(Q\right)\epsilon\left(Q\right)
\end{array}
\]
 which contradicts (\ref{eq:0}). So we do have $\nu\left(r\right)=\nu\left(P_{1}P_{2}\right)<\nu\left(qQ\right)$,
and this completes the proof of the base of the induction.

We now assume the result true for $t-1\geq2$ and we are going to
show it for $t$. We set $P:=\prod\limits _{i=1}^{t-1}P_{i}$.

Let 
\[
P=q_{1}Q+r_{1}
\]
 be the Euclidean division of $P$ by $Q$ and 
\[
r_{1}P_{t}=q_{2}Q+r_{2}
\]
 be that of $r_{1}P_{t}$ by $Q$. Since $PP_{t}=qQ+r$, we have $r=r_{2}$
and $q=q_{1}P_{t}+q_{2}$.

By the induction hypothesis, $\nu\left(r_{1}\right)=\nu\left(P\right)<\nu\left(q_{1}Q\right)$.
In particular, 
\[
\nu\left(r_{1}P_{t}\right)=\nu\left(\prod\limits _{i=1}^{t}P_{i}\right)<\nu\left(q_{1}P_{t}Q\right).
\]
Since the polynomials $r_{1}$ and $P_{t}$ are both of degree strictly
less than $\deg\left(Q\right)$, we can apply the base of the induction
and so 
\[
\nu\left(r_{1}P_{t}\right)=\nu\left(r_{2}\right)<\nu\left(q_{2}Q\right).
\]

So $\nu\left(r\right)=\nu\left(r_{2}\right)=\nu\left(r_{1}P_{t}\right)=\nu\left(\prod\limits _{i=1}^{t}P_{i}\right)$
and furthermore this value is strictly less than both $\nu\left(q_{1}P_{t}Q\right)$
and than $\nu\left(q_{2}Q\right)$. So it is strictly less than the
minimum, which is less then or equal to $\nu\left(q_{1}P_{t}Q+q_{2}Q\right)$
by definition of a valuation. So 
\[
\begin{array}{ccc}
\nu\left(r\right) & = & \nu\left(\prod\limits _{i=1}^{t}P_{i}\right)\\
 & < & \nu\left(\left(q_{1}P_{t}+q_{2}\right)Q\right)\\
 & = & \nu\left(qQ\right)
\end{array}
\]
which completes the proof.
\end{proof}
\begin{thm}
Let $Q$ be a key polynomial. The map $\nu_{Q}$ is a valuation.
\end{thm}

\begin{proof}
The only thing we have to prove is that for every $\left(P_{1},P_{2}\right)\in K[X]^{2}$,
we have
\[
\nu_{Q}\left(P_{1}P_{2}\right)=\nu_{Q}\left(P_{1}\right)+\nu_{Q}\left(P_{2}\right).
\]
 First case: $P_{1}$ and $P_{2}$ are both of degree strictly less
than $\deg\left(Q\right)$. Then $\nu_{Q}\left(P_{1}\right)=\nu\left(P_{1}\right)$
and $\nu_{Q}\left(P_{2}\right)=\nu\left(P_{2}\right)$. Since $\nu$
is a valuation, we have $\nu\left(P_{1}P_{2}\right)=\nu\left(P_{1}\right)+\nu\left(P_{2}\right)$. 

Then, $\nu\left(P_{1}P_{2}\right)=\nu_{Q}\left(P_{1}\right)+\nu_{Q}\left(P_{2}\right)$.
Since $P_{1}$ and $P_{2}$ are both of degree strictly less than
$\deg\left(Q\right)$, by previous Lemma, we have $\nu_{Q}\left(P_{1}P_{2}\right)=\nu\left(P_{1}P_{2}\right)$
and we are done.

Second case: $P_{1}=p_{i}^{(1)}Q^{i}$ and $P_{2}=p_{j}^{(2)}Q^{j}$,
with $p_{i}^{(1)}$ and $p_{j}^{(2)}$ both of degree strictly less
than $\deg\left(Q\right)$. 

Let $p_{i}^{(1)}p_{j}^{(2)}=qQ+r$ be the Euclidean division of $p_{i}^{(1)}p_{j}^{(2)}$
by $Q$. Since $\deg_{X}\left(p_{i}^{(1)}p_{j}^{(2)}\right)<2\deg_{X}\left(Q\right)$,
we know that $\deg_{X}\left(q\right)<\deg_{X}\left(Q\right)$, and
by definition of the Euclidean division, we have $\deg_{X}\left(r\right)<\deg_{X}\left(Q\right)$.
So $P_{1}P_{2}=qQ^{i+j+1}+rQ^{i+j}$ is the $Q$-expansion of $P_{1}P_{2}$. 

We are going to prove that in this case we still have 
\[
\nu_{Q}\left(P_{1}P_{2}\right)=\nu\left(P_{1}P_{2}\right),
\]
and since $\nu$ is a valuation, we will have the result. We have:

\[
\begin{array}{ccc}
\nu_{Q}\left(P_{1}P_{2}\right) & = & \nu_{Q}\left(qQ^{i+j+1}+rQ^{i+j}\right)\\
 & = & \min\left\{ \nu\left(qQ^{i+j+1}\right),\nu\left(rQ^{i+j}\right)\right\} \\
 & = & \min\left\{ \nu\left(qQ\right)+\nu\left(Q^{i+j}\right),\nu\left(r\right)+\nu\left(Q^{i+j}\right)\right\} .
\end{array}
\]

However, we can apply thee previous Lemma to the product 
\[
p_{i}^{(1)}p_{j}^{(2)}=qQ+r
\]
and conclude that $\nu\left(r\right)=\nu\left(p_{i}^{(1)}p_{j}^{(2)}\right)<\nu\left(qQ\right)$.

Then 
\[
\begin{array}{ccc}
\nu_{Q}\left(P_{1}P_{2}\right) & = & \nu\left(r\right)+\nu\left(Q^{i+j}\right)\\
 & = & \nu\left(p_{i}^{(1)}p_{j}^{(2)}\right)+\nu\left(Q^{i+j}\right)\\
 & = & \nu\left(P_{1}P_{2}\right)
\end{array}
\]
and we have the result.

Last case: general case. Since we only look at the terms of smallest
value, we can replace $P_{1}$ by 
\[
\left(\widetilde{P}_{1}\right)_{\nu,Q}=\sum\limits _{j\in S_{Q}\left(P_{1}\right)}p_{j}^{(1)}Q^{j}
\]
 and $P_{2}$ by 
\[
\left(\widetilde{P}_{2}\right)_{\nu,Q}=\sum\limits _{i\in S_{Q}\left(P_{2}\right)}p_{i}^{(2)}Q^{i}.
\]
We know that 
\[
\nu_{Q}\left(P_{1}+P_{2}\right)\geq\min\left\{ \nu_{Q}\left(P_{1}\right),\nu_{Q}\left(P_{2}\right)\right\} 
\]
 and 
\[
\nu_{Q}\left(p_{j}^{(1)}Q^{j}p_{i}^{(2)}Q^{i}\right)=\nu_{Q}\left(p_{j}^{(1)}Q^{j}\right)+\nu_{Q}\left(p_{i}^{(2)}Q^{i}\right).
\]
So

\[
\begin{array}{ccc}
\nu_{Q}\left(P_{1}P_{2}\right) & = & \nu_{Q}\left(\sum p_{j}^{(1)}p_{i}^{(2)}Q^{j+i}\right)\\
 & \geq & \min\left\{ \nu_{Q}\left(p_{j}^{(1)}Q^{j}\right)+\nu_{Q}\left(p_{i}^{(2)}Q^{i}\right)\right\} .
\end{array}
\]

However
\[
\nu_{Q}\left(p_{j}^{(1)}Q^{j}\right)=\nu\left(p_{j}^{(1)}Q^{j}\right)=\nu_{Q}\left(P_{1}\right)
\]
 and 
\[
\nu_{Q}\left(p_{i}^{(2)}Q^{i}\right)=\nu\left(p_{i}^{(2)}Q^{i}\right)=\nu_{Q}\left(P_{2}\right).
\]
So $\nu_{Q}\left(P_{1}P_{2}\right)\geq\nu_{Q}\left(P_{1}\right)+\nu_{Q}\left(P_{2}\right)$,
and we only have to show that it is an equality. In order to do that,
it is enough to find a term in the $Q$-expansion of $P_{1}P_{2}$
whose value is exactly $\nu_{Q}\left(P_{1}\right)+\nu_{Q}\left(P_{2}\right)$.
Let us consider the term of smallest value in each $Q$-expansion,
so let us consider $p_{n_{1}}^{(1)}Q^{n_{1}}$ and $p_{m_{2}}^{(2)}Q^{m_{2}}$,
where $n_{1}=\min S_{Q}\left(P_{1}\right)$ and $m_{2}=\min S_{Q}\left(P_{2}\right)$.

Let $p_{n_{1}}^{(1)}p_{m_{2}}^{(2)}=qQ+r$ be the Euclidean division
of $p_{n_{1}}^{(1)}p_{m_{2}}^{(2)}$ by $Q$, which is its $Q$-expansion
too. 

By Lemma \ref{lem:paslechoixfaut le mettre}, we have $\nu(r)=\nu\left(p_{n_{1}}^{(1)}p_{m_{2}}^{(2)}\right)$.
In fact, in the $Q$-expansion of $P_{1}P_{2}$, there is the term
$rQ^{n_{1}+m_{2}}$, and we have:

\[
\begin{array}{ccc}
\nu_{Q}\left(rQ^{n_{1}+m_{2}}\right) & = & \nu\left(rQ^{n_{1}+m_{2}}\right)\\
 & = & \nu\left(p_{n_{1}}^{(1)}p_{m_{2}}^{(2)}Q^{n_{1}+m_{2}}\right)\\
 & = & \nu_{Q}\left(P_{1}\right)+\nu_{Q}\left(P_{2}\right).
\end{array}
\]

This completes the proof.
\end{proof}
\begin{rem}
For every polynomial $P\in K[X]$, we have 
\[
\nu_{Q}\left(P\right)\leq\nu\left(P\right).
\]
It will be very important to be able to determine when this inequality
is an equality. 

A key polynomial $P$ which satisfies the strict inequality and which
is of minimal degree for this property will be called an immediate
successor of $Q$ (Definition \ref{def: successeur immediat}). We
will study these polynomials in more details in this work. First,
let us concentrate on the equality case.
\end{rem}

\begin{defn}
Let $Q$ be a key polynomial and $P$ be a polynomial such that $\nu_{Q}\left(P\right)=\nu\left(P\right)$.
We say that $P$ is \emph{non-degenerate with respect to $Q$}.

Another very important thing is to be able to compare the $\epsilon$
of key polynomials. Indeed, if I have two key polynomials $Q_{1}$
and $Q_{2}$, do I have $\epsilon\left(Q_{1}\right)<\epsilon\left(Q_{2}\right)$,
or do I have $\epsilon\left(Q_{1}\right)=\epsilon\left(Q_{2}\right)$
? Being able to answer will be crucial. The next four results can
be found in \cite{DMS} but we recall them for more clarity.
\end{defn}

\begin{lem}
\label{lem: inegalit=0000E9 epsilon_Q(P)et epsilon(Q)}For every polynomial
$P\in K[X]$ and every stricly positive integer $d$, we have :

\[
\nu_{Q}\left(\partial_{d}P\right)\geq\nu_{Q}\left(P\right)-d\epsilon\left(Q\right)
\]
\end{lem}

\begin{proof}
We consider the $Q$-expansion $P=\sum\limits _{i=0}^{n}p_{i}Q^{i}$
of $P$.

Assume we have the result for $p_{i}Q^{i}$. It means that
\[
\nu_{Q}\left(\partial_{d}\left(p_{i}Q^{i}\right)\right)\geq\nu_{Q}\left(p_{i}Q^{i}\right)-d\epsilon\left(Q\right)
\]
 for every index $i$. Then:

\[
\begin{array}{ccc}
\nu_{Q}\left(\partial_{d}P\right) & = & \nu_{Q}\left(\partial_{d}\left(\sum\limits _{i=0}^{n}p_{i}Q^{i}\right)\right)\\
 & = & \nu_{Q}\left(\sum\limits _{i=0}^{n}\partial_{d}\left(p_{i}Q^{i}\right)\right)\\
 & \geq & \min\limits _{0\leq i\leq n}\nu_{Q}\left(\partial_{d}\left(p_{i}Q^{i}\right)\right)\\
 & \geq & \min\limits _{0\leq i\leq n}\left\{ \nu_{Q}\left(p_{i}Q^{i}\right)-d\epsilon\left(Q\right)\right\} \\
 & \geq & \min\limits _{0\leq i\leq n}\left\{ \nu_{Q}\left(p_{i}Q^{i}\right)\right\} -d\epsilon\left(Q\right)\\
 & \geq & \nu_{Q}\left(P\right)-d\epsilon\left(Q\right)
\end{array}
\]

and the proof is finished.

So we just have to prove the result for $P=p_{i}Q^{i}$. 

First, we know that $\nu_{Q}\left(\partial_{d}Q\right)\geq\nu_{Q}\left(Q\right)-d\epsilon\left(Q\right)$.
Now we will prove that we have the result for $P=p_{i}$. Then we
will conclude by showing that if we have the result for two polynomials,
we have the result for the product. 

So let us prove the result for $P=p_{i}$. 

Since $\deg_{X}\left(p_{i}\right)<\deg_{X}\left(Q\right)$ and since
$Q$ is a key polynomial, we have $\epsilon\left(p_{i}\right)<\epsilon\left(Q\right)$.
So, for every strictly positive integer $d$, we have:

\[
\begin{array}{ccc}
\nu_{Q}\left(\partial_{d}p_{i}\right) & = & \nu\left(\partial_{d}p_{i}\right)\\
 & \geq & \nu\left(p_{i}\right)-d\epsilon\left(p_{i}\right)\\
 & = & \nu_{Q}\left(p_{i}\right)-d\epsilon\left(p_{i}\right)\\
 & > & \nu_{Q}\left(p_{i}\right)-d\epsilon\left(Q\right).
\end{array}
\]

Now, it just remains to prove that if we have the result for two polynomials
$P$ and $S$, then we have it for $PS$. Assume the result proven
for $P$ and $S$. Then:

\[
\begin{array}{ccc}
\nu_{Q}\left(\partial_{d}\left(PS\right)\right) & = & \nu_{Q}\left(\sum\limits _{r=0}^{d}\partial_{r}\left(P\right)\partial_{d-r}\left(S\right)\right)\\
 & \geq & \min\limits _{0\leq r\leq d}\left\{ \nu\left(\partial_{r}\left(P\right)\right)+\nu\left(\partial_{d-r}\left(S\right)\right)\right\} \\
 & \geq & \min\limits _{0\leq r\leq d}\left\{ \nu_{Q}\left(P\right)-r\epsilon\left(Q\right)+\nu_{Q}\left(S\right)-\left(d-r\right)\epsilon\left(Q\right)\right\} \\
 & \geq & \nu_{Q}\left(PS\right)-d\epsilon\left(Q\right)
\end{array}
\]

This completes the proof.
\end{proof}
\begin{prop}
\label{prop:S_Q(P)neq0}Let $Q$ be a key polynomial and $P\in K[X]$
a polynomial such that $S_{Q}\left(P\right)\neq\left\{ 0\right\} $.

Then there exists a strictly positive integer $b$ such that 
\[
\frac{\nu_{Q}\left(P\right)-\nu_{Q}\left(\partial_{b}P\right)}{b}=\epsilon\left(Q\right).
\]
\end{prop}

\begin{proof}
First, by Lemma \ref{lem: inegalit=0000E9 epsilon_Q(P)et epsilon(Q)},
we can replace $P$ by $\widetilde{P}_{\nu,Q}=\sum\limits _{i\in S_{Q}\left(P\right)}p_{i}Q^{i}$.

We want to show the existence of a strictly positive integer $b$
such that $\nu_{Q}\left(P\right)-\nu_{Q}\left(\partial_{b}P\right)=b\epsilon\left(Q\right)$.

Since $S_{Q}\left(P\right)\neq\left\{ 0\right\} $, we can consider
the smallest non-zero element $l$ of $S_{Q}\left(P\right)$. We write
$l=p^{e}u$, with $p\nmid u$. 

We are going to prove that we have the desired equality for the integer
$b:=p^{e}b\left(Q\right)>0$. To do this, we need to compute $\partial_{b}\left(P\right)$,
it is the objective of the following technical lemma.
\begin{lem}
We have $\partial_{b}\left(P\right)=urQ^{l-p^{e}}+Q^{l-p^{e}+1}R+S$,
where:
\begin{enumerate}
\item The polynomial $r$ is the remainder of the Euclidean division of
$p_{l}\left(\partial_{b\left(Q\right)}Q\right)^{p^{e}}$ by $Q$,
\item The polynomials $R$ and $S$ satisfy
\[
\nu_{Q}\left(S\right)>\nu_{Q}\left(P\right)-b\epsilon\left(Q\right).
\]
\end{enumerate}
\end{lem}

\begin{proof}
First let us show that the Lemma is true for $P=p_{l}Q^{l}$ and that
for every $j\in S_{Q}\left(P\right)\setminus\left\{ l\right\} $,
we have 
\[
\partial_{b}\left(p_{j}Q^{j}\right)=Q^{l-p^{e}+1}R_{j}+S_{j},
\]
where $R_{j}$ and $S_{j}$ are two polynomials, and where $\nu_{Q}\left(S_{j}\right)>\nu_{Q}\left(P\right)-b\epsilon\left(Q\right)$.

So we consider  $j\in S_{Q}\left(P\right)$. We set 
\[
M_{j}:=\left\{ B_{s}=\left(b_{0},\dots,b_{s}\right)\in\mathbb{N}^{s+1}\text{ such that }\sum\limits _{i=0}^{s}b_{i}=b\text{ and }s\leq j\right\} .
\]

The generalized Leibniz rule tells us that:

\[
\partial_{b}\left(p_{j}Q^{j}\right)=\sum\limits _{B_{s}\in M_{j}}\left(T\left(B_{s}\right)\right)
\]

where 
\[
\begin{array}{ccc}
T\left(B_{s}\right) & = & T\left(\left(b_{0},\dots,b_{s}\right)\right)\\
 & = & C\left(B_{s}\right)\partial_{b_{0}}\left(p_{j}\right)\left(\prod\limits _{i=1}^{s}\partial_{b_{i}}\left(Q\right)\right)Q^{j-s}
\end{array}
\]
 with $C\left(B_{s}\right)$ some elements of $K$ whose exact value
can be found in \cite{HMOS}. We set 
\[
\alpha:=\left(0,b\left(Q\right),\dots,b\left(Q\right)\right)\in\mathbb{N}^{p^{e}+1}.
\]

Recall that $I\left(Q\right)=\left\{ d\in\mathbb{N}^{\ast}\text{ such that }\frac{\nu\left(Q\right)-\nu\left(\partial_{d}Q\right)}{d}=\epsilon\left(Q\right)\right\} $.
We set 

\[
N_{j}:=\left\{ B_{s}=\left(b_{0},\dots,b_{s}\right)\in M_{j}\text{ such that }b_{0}>0\text{ or }\left\{ b_{1},\dots,b_{s}\right\} \nsubseteq I\left(Q\right)\right\} ,
\]

\[
S_{j}:=\sum\limits _{B_{s}\in N_{j}}T\left(B_{s}\right)
\]
and finally we set 
\[
Q^{l-p^{e}+1}R_{j}:=\begin{cases}
\sum\limits _{B_{s}\in M_{j}\setminus N_{j}}T\left(B_{s}\right) & \text{if }j\neq l\\
\sum\limits _{B_{s}\in M_{j}\setminus\left(N_{j}\cup\left\{ \alpha\right\} \right)}T\left(B_{s}\right) & \text{if }j=l
\end{cases}.
\]

If $j=l$, the term $T\left(\alpha\right)$ appears $\binom{l}{p^{e}}=u$
times in $\partial_{b}\left(p_{l}Q^{l}\right)$. Equivalently, $C\left(\alpha\right)=u$
and so 
\[
\begin{array}{ccc}
T\left(\alpha\right) & = & up_{l}\left(\partial_{b\left(Q\right)}Q\right)^{p^{e}}Q^{l-p^{e}}\\
 & = & u\left(qQ+r\right)Q^{l-p^{e}}
\end{array}
\]
 where $qQ+r$ is the Euclidean division of $p_{l}\left(\partial_{b\left(Q\right)}Q\right)^{p^{e}}$
by $Q$.

In other words 
\[
T\left(\alpha\right)=\underset{:=R_{0}}{\underbrace{uq}}Q^{l-p^{e}+1}+urQ^{l-p^{e}}.
\]

So if $j\neq l$, then $\partial_{b}\left(p_{j}Q^{j}\right)=Q^{l-p^{e}+1}R_{j}+S_{j}$.
It remains to prove that $\nu_{Q}\left(S_{j}\right)>\nu_{Q}\left(p_{j}Q^{j}\right)-b\epsilon\left(Q\right)$. 

But:

\[
\begin{array}{ccc}
\nu_{Q}\left(S_{j}\right) & = & \nu_{Q}\left(\sum\limits _{B_{s}\in N_{j}}T\left(B_{s}\right)\right)\\
 & = & \nu_{Q}\left(\sum\limits _{B_{s}\in N_{j}}C\left(B_{s}\right)\partial_{b_{0}}\left(p_{j}\right)\left(\prod\limits _{i=1}^{s}\partial_{b_{i}}\left(Q\right)\right)Q^{j-s}\right)\\
 & \geq & \min\limits _{B_{s}\in N_{j}}\left\{ \nu\left(\partial_{b_{0}}\left(p_{j}\right)\right)+\sum\limits _{i=1}^{s}\nu\left(\partial_{b_{i}}\left(Q\right)\right)+\left(j-s\right)\nu\left(Q\right)\right\} .
\end{array}
\]

Since $B_{s}\in N_{j}$, we have two options. The first is $b_{0}=0$
and $\left\{ b_{1},\dots,b_{s}\right\} \nsubseteq I\left(Q\right)$.
In other words for every $i\in\left\{ 1,\dots,s\right\} $ we have
$\nu\left(\partial_{b_{i}}\left(Q\right)\right)\geq\nu\left(Q\right)-b_{i}\epsilon\left(Q\right)$.
And then the inequality is strict for at least one index $i\in\left\{ 1,\dots,s\right\} $.
The second option is $b_{0}>0$ and then 
\[
\frac{\nu\left(p_{j}\right)-\nu\left(\partial_{b_{0}}\left(p_{j}\right)\right)}{b_{0}}\leq\epsilon\left(p_{j}\right)<\epsilon\left(Q\right)
\]
 because $\deg_{X}\left(p_{j}\right)<\deg_{X}\left(Q\right)$ and
$Q$ is a key polynomial. Equivalently, 
\[
\nu\left(\partial_{b_{0}}\left(p_{j}\right)\right)>\nu\left(p_{j}\right)-b_{0}\epsilon\left(Q\right).
\]

So if $b_{0}=0$ and $\left\{ b_{1},\dots,b_{s}\right\} \nsubseteq I\left(Q\right)$,
we have 

\[
\begin{array}{c}
\nu\left(\partial_{b_{0}}\left(p_{j}\right)\right)+\sum\limits _{i=1}^{s}\nu\left(\partial_{b_{i}}\left(Q\right)\right)+\left(j-s\right)\nu\left(Q\right)\\
>\underset{\begin{array}{c}
\shortparallel\\
\nu\left(p_{j}Q^{j}\right)-b\epsilon\left(Q\right)
\end{array}}{\underbrace{\nu\left(p_{j}\right)+s\nu\left(Q\right)-b\epsilon\left(Q\right)+\left(j-s\right)\nu\left(Q\right).}}
\end{array}
\]

And if $b_{0}>0$, then 

\[
\begin{array}{c}
\nu\left(\partial_{b_{0}}\left(p_{j}\right)\right)+\sum\limits _{i=1}^{s}\nu\left(\partial_{b_{i}}\left(Q\right)\right)+\left(j-s\right)\nu\left(Q\right)\\
>\underset{\begin{array}{c}
\shortparallel\\
\nu\left(p_{j}Q^{j}\right)-b\epsilon\left(Q\right)
\end{array}}{\underbrace{\nu\left(p_{j}\right)-b_{0}\epsilon\left(Q\right)+s\nu\left(Q\right)-\sum\limits _{i=1}^{s}b_{i}\epsilon\left(Q\right)+\left(j-s\right)\nu\left(Q\right).}}
\end{array}
\]

So:

\[
\begin{array}{ccc}
\nu_{Q}\left(S_{j}\right) & > & \min\limits _{B_{s}\in N_{j}}\left\{ \nu\left(p_{j}Q^{j}\right)-b\epsilon\left(Q\right)\right\} \\
 & > & \nu_{Q}\left(P\right)-b\epsilon\left(Q\right)
\end{array}
\]
 as desired.

If $j=l$, then 
\[
\partial_{b}\left(p_{l}Q^{l}\right)=\left(R_{0}+R_{l}\right)Q^{l-p^{e}+1}+S_{l}+urQ^{l-p^{e}}
\]
 hand using the same argument as before, $\nu_{Q}\left(S_{l}\right)>\nu_{Q}\left(P\right)-b\epsilon\left(Q\right)$.

It remains to prove the general case. We have:

\[
\begin{array}{ccc}
\partial_{b}\left(P\right) & = & \partial_{b}\left(\sum\limits _{i\in S_{Q}\left(P\right)}p_{i}Q^{i}\right)\\
 & = & \partial_{b}\left(p_{l}Q^{l}\right)+\sum\limits _{j\in S_{Q}\left(P\right)\setminus\left\{ l\right\} }\partial_{b}\left(p_{j}Q^{j}\right).
\end{array}
\]

Then:

\[
\begin{array}{ccc}
\partial_{b}\left(P\right) & = & \left(R_{0}+R_{l}\right)Q^{l-p^{e}+1}+S_{l}+urQ^{l-p^{e}}+\sum\limits _{j\in S_{Q}\left(P\right)\setminus\left\{ l\right\} }\left(Q^{l-p^{e}+1}R_{j}+S_{j}\right)\\
 & = & urQ^{l-p^{e}}+Q^{l-p^{e}+1}R+S
\end{array}
\]

where 
\[
R:=R_{0}+\sum\limits _{j\in S_{Q}\left(P\right)}R_{j}
\]
 and 
\[
S:=\sum\limits _{j\in S_{Q}\left(P\right)}S_{j}.
\]
We have
\[
\nu_{Q}\left(S\right)\geq\min\limits _{j\in S_{Q}\left(P\right)}\left\{ \nu_{Q}\left(S_{j}\right)\right\} >\nu_{Q}\left(P\right)-b\epsilon\left(Q\right).
\]

This completes the proof of the Lemma.
\end{proof}
Recall that we want to prove that 
\[
\nu_{Q}\left(\partial_{b}P\right)=\nu_{Q}\left(P\right)-b\epsilon\left(Q\right).
\]
 We just saw that the $Q$-expansion of $\partial_{b}P$ contains
the term $urQ^{l-p^{e}}$, some terms divisible by $Q^{l-p^{e}+1}$
and others of value strictly higher than $\nu_{Q}\left(P\right)-b\epsilon\left(Q\right)$.
It is sufficient now to show that 
\[
\nu_{Q}\left(\partial_{b}P\right)\geq\nu_{Q}\left(P\right)-b\epsilon\left(Q\right)
\]
 and that 
\[
\nu_{Q}\left(urQ^{l-p^{e}}\right)=\nu_{Q}\left(P\right)-b\epsilon\left(Q\right).
\]

Let us compute $\nu_{Q}\left(urQ^{l-p^{e}}\right)$. 

Recall that $p_{l}\left(\partial_{b\left(Q\right)}Q\right)^{p^{e}}=qQ+r$.
By Lemma \ref{lem:paslechoixfaut le mettre}, we have $\nu\left(r\right)=\nu\left(p_{l}\left(\partial_{b\left(Q\right)}Q\right)^{p^{e}}\right)$.

So:

\[
\begin{array}{ccc}
\nu_{Q}\left(urQ^{l-p^{e}}\right) & = & \nu_{Q}\left(rQ^{l-p^{e}}\right)\\
 & = & \nu\left(rQ^{l-p^{e}}\right)\\
 & = & \nu\left(p_{l}\left(\partial_{b\left(Q\right)}Q\right)^{p^{e}}\right)+\nu\left(Q^{l-p^{e}}\right)\\
 & = & \nu\left(p_{l}Q^{l}\right)+p^{e}\nu\left(\partial_{b\left(Q\right)}Q\right)-p^{e}\nu\left(Q\right)\\
 & = & \nu_{Q}\left(P\right)+p^{e}\left(\nu\left(\partial_{b\left(Q\right)}Q\right)-\nu\left(Q\right)\right)\\
 & = & \nu_{Q}\left(P\right)+p^{e}\left(-b\left(Q\right)\epsilon\left(Q\right)\right)\\
 & = & \nu_{Q}\left(P\right)-b\epsilon\left(Q\right).
\end{array}
\]

The result now follows from Lemma \ref{lem: inegalit=0000E9 epsilon_Q(P)et epsilon(Q)}.
\end{proof}
\begin{rem}
One can show that the implication of the proposition is, in fact,
an equivalence.
\end{rem}

\begin{prop}
\label{cor:comparaison des epsilon}Let $Q$ be a key polynomial and
$P$ a polynomial such that there exists a strictly positive integer
$b$ such that
\[
\nu_{Q}\left(P\right)-\nu_{Q}\left(\partial_{b}P\right)=b\epsilon\left(Q\right)
\]
 and 
\[
\nu_{Q}\left(\partial_{b}P\right)=\nu\left(\partial_{b}P\right).
\]

Then $\epsilon\left(P\right)\geq\epsilon\left(Q\right)$.

If, in addition, $\nu\left(P\right)>\nu_{Q}\left(P\right)$, then
$\epsilon\left(P\right)>\epsilon\left(Q\right)$.
\end{prop}

\begin{proof}
We have 

\[
\begin{array}{ccc}
\epsilon\left(P\right) & \geq & \frac{\nu\left(P\right)-\nu\left(\partial_{b}P\right)}{b}\\
 & = & \frac{\nu\left(P\right)-\nu_{Q}\left(\partial_{b}P\right)}{b}\\
 & = & \frac{\nu\left(P\right)+b\epsilon\left(Q\right)-\nu_{Q}\left(P\right)}{b}\\
 & = & \epsilon\left(Q\right)+\frac{\nu\left(P\right)-\nu_{Q}\left(P\right)}{b}.
\end{array}
\]

We know that for every polynomial $P$, we have $\nu\left(P\right)\geq\nu_{Q}\left(P\right)$
, so $\epsilon\left(P\right)\geq\epsilon\left(Q\right)$. And if $\nu\left(P\right)>\nu_{Q}\left(P\right)$,
we have the strict inequality $\epsilon\left(P\right)>\epsilon\left(Q\right)$.
\end{proof}
\begin{prop}
\label{prop: inegalite des valuations tronqu=0000E9es}Let $Q_{1}$
and $Q_{2}$ be two key polynomials such that 
\[
\epsilon\left(Q_{1}\right)\leq\epsilon\left(Q_{2}\right)
\]
 and let $P\in K[X]$ be a polynomial.

Then $\nu_{Q_{1}}\left(P\right)\leq\nu_{Q_{2}}\left(P\right)$.

Furthermore, if $\nu_{Q_{1}}\left(P\right)=\nu\left(P\right)$, then
$\nu_{Q_{2}}\left(P\right)=\nu\left(P\right)$.
\end{prop}

\begin{proof}
First, we show that $\nu_{Q_{2}}\left(Q_{1}\right)=\nu\left(Q_{1}\right)$.
If $\deg_{X}\left(Q_{1}\right)<\deg_{X}\left(Q_{2}\right)$, we do
have this equality. Otherwise we have $\deg_{X}\left(Q_{1}\right)=\deg_{X}\left(Q_{2}\right)$
since $\epsilon\left(Q_{1}\right)\leq\epsilon\left(Q_{2}\right)$
and since $Q_{1}$ is a key polynomial.

Assume, aiming for contradiction, that $\nu_{Q_{2}}\left(Q_{1}\right)<\nu\left(Q_{1}\right)$. 

So $S_{Q_{2}}\left(Q_{1}\right)\neq\left\{ 0\right\} $ and by Proposition
\ref{prop:S_Q(P)neq0}, there exists a non-zero integer $b$ such
that $\nu_{Q_{2}}\left(Q_{1}\right)-\nu_{Q_{2}}\left(\partial_{b}Q_{1}\right)=b\epsilon\left(Q_{2}\right)$.
However $\deg_{X}\left(\partial_{b}Q_{1}\right)<\deg_{X}\left(Q_{2}\right)$,
so $\nu_{Q_{2}}\left(\partial_{b}Q_{1}\right)=\nu\left(\partial_{b}Q_{1}\right)$
and by Proposition \ref{cor:comparaison des epsilon}, we have $\epsilon\left(Q_{1}\right)>\epsilon\left(Q_{2}\right)$.
This is a contradiction. So we do have $\nu_{Q_{2}}\left(Q_{1}\right)=\nu\left(Q_{1}\right)$.

Let $P=\sum\limits _{i=0}^{n}p_{i}Q_{1}^{i}$ be the $Q_{1}$-expansion
of $P$. 

For every $i\in\left\{ 0,\dots,n\right\} $, we have:

\[
\nu_{Q_{2}}\left(p_{i}Q_{1}^{i}\right)=\nu_{Q_{2}}\left(p_{i}\right)+i\nu_{Q_{2}}\left(Q_{1}\right)=\nu_{Q_{2}}\left(p_{i}\right)+i\nu\left(Q_{1}\right).
\]

But $\deg_{X}\left(p_{i}\right)<\deg_{X}\left(Q_{1}\right)\leq\deg_{X}\left(Q_{2}\right)$,
so $\nu_{Q_{2}}\left(p_{i}\right)=\nu\left(p_{i}\right)$ and $\nu_{Q_{2}}\left(p_{i}Q_{1}^{i}\right)=\nu\left(p_{i}Q_{1}^{i}\right)$.

Then 
\[
\begin{array}{ccc}
\nu_{Q_{2}}\left(P\right) & \geq & \min\limits _{0\leq i\leq n}\left\{ \nu_{Q_{2}}\left(p_{i}Q_{1}^{i}\right)\right\} \\
 & = & \min\limits _{0\leq i\leq n}\left\{ \nu\left(p_{i}Q_{1}^{i}\right)\right\} \\
 & = & \nu_{Q_{1}}\left(P\right).
\end{array}
\]

Assume that, in addition, $\nu_{Q_{1}}\left(P\right)=\nu\left(P\right)$.
Then $\nu\left(P\right)\leq\nu_{Q_{2}}\left(P\right)$. By definition
of $\nu_{Q_{2}}$, we have $\nu_{Q_{2}}\left(P\right)\leq\nu\left(P\right)$,
and hence the equality.
\end{proof}
\begin{prop}
\label{prop:non degeneresence de polyn=0000F4mes par rapport au m=0000EAme polyn=0000F4me clef}Let
$P_{1},\dots,P_{n}\in K[X]$ be polynomials and set $d:=\max\limits _{1\leq i\leq n}\left\{ \deg_{X}\left(P_{i}\right)\right\} $.

There exists a key polynomial $Q$ of degree less than or equal to
$d$ such that all the $P_{i}$ are non-degenerate with respect to
$Q$. In other words, there exists a key polynomial $Q$ such that
for every $i$, we have $\nu_{Q}\left(P_{i}\right)=\nu\left(P_{i}\right)$.
\end{prop}

\begin{proof}
Assume the result for only one polynomial and let $n>1$. 

So we have $Q_{1},\dots,Q_{n}$ some key polynomials of degrees less
than or equal to $d$ such that for every $i\in\left\{ 1,\dots,n\right\} $,
the polynomial $P_{i}$ is non-degenerate with respect to $Q_{i}$.
In other words $\nu_{Q_{i}}\left(P_{i}\right)=\nu\left(P_{i}\right)$. 

We can assume 
\[
\epsilon\left(Q_{n}\right)=\max\limits _{1\leq i\leq n}\left\{ \epsilon\left(Q_{i}\right)\right\} .
\]
By Proposition \ref{prop: inegalite des valuations tronqu=0000E9es},
for every $i\in\left\{ 1,\dots,n\right\} $, we have $\nu_{Q_{i}}\left(P_{i}\right)=\nu\left(P_{i}\right)=\nu_{Q_{n}}\left(P_{i}\right)$.
So all the $P_{i}$ are non-degenerate with respect to $Q_{n}$. This
completes the proof.

It remains to show the result for $n=1$. We give a proof by contradiction.
Assume the existence of a polynomial $P$ such that for every key
polynomial $Q$ of degree less than or equal to $d$, we have $\nu_{Q}\left(P\right)<\nu\left(P\right)$.
We choose $P$ of minimal degree for this property.

Let us show that there exists a key polynomial $Q$, of degree less
than or equal to $d=\deg_{X}\left(P\right)$, such that for every
$b>0$, we have $\nu_{Q}\left(\partial_{b}P\right)=\nu\left(\partial_{b}P\right)$.

First, for every $b>d$, we have $\partial_{b}P=0$. Then, by minimality
of the degree of $P$, for every $b\in\left\{ 1,\dots,d\right\} $,
there exists a key polynomial $Q_{b}$ such that $\nu_{Q_{b}}\left(\partial_{b}P\right)=\nu\left(\partial_{b}P\right)$. 

Take an element $Q\in\left\{ Q_{1},\dots,Q_{d}\right\} $ such that
$\epsilon\left(Q\right)=\max\limits _{1\leq b\leq d}\left\{ \epsilon\left(Q_{b}\right)\right\} $.
By Proposition \ref{prop: inegalite des valuations tronqu=0000E9es},
we have $\nu_{Q}\left(\partial_{b}P\right)=\nu\left(\partial_{b}P\right)$,
for every $b>0$.

So we have $\nu_{Q}\left(P\right)<\nu\left(P\right)$. In particular,
$S_{Q}\left(P\right)\neq\left\{ 0\right\} $, and $\nu_{Q}\left(\partial_{b}P\right)=\nu\left(\partial_{b}P\right)$
for every $b>0$. By Proposition \ref{prop:S_Q(P)neq0} and Corollary
\ref{cor:comparaison des epsilon}, we conclude that $\epsilon\left(P\right)>\epsilon\left(Q\right)$.

Let us show that this last inequality is true for every key polynomial
of degree less than or equal than $\deg(P)$. Let $Q_{0}$ be such
a key polynomial. 

First case: $\epsilon\left(Q_{0}\right)\leq\epsilon\left(Q\right)$.
Then $\epsilon\left(Q_{0}\right)<\epsilon\left(P\right)$ since $\epsilon\left(Q\right)<\epsilon\left(P\right)$.

Last case: $\epsilon\left(Q_{0}\right)>\epsilon\left(Q\right)$. By
Proposition \ref{prop: inegalite des valuations tronqu=0000E9es},
we have $\nu\left(\partial_{b}P\right)=\nu_{Q}\left(\partial_{b}P\right)=\nu_{Q_{0}}\left(\partial_{b}P\right)$
for every $b>0$. By hypothesis we know that $\nu_{Q_{0}}\left(P\right)<\nu\left(P\right)$.
So by Proposition \ref{prop:S_Q(P)neq0} and Corollary \ref{cor:comparaison des epsilon},
we have $\epsilon\left(P\right)>\epsilon\left(Q_{0}\right)$ as desired.

So we know that for every key polynomial of degree less than or equan
than those of $P$, we have $\epsilon\left(P\right)>\epsilon\left(Q\right)$.
But by definition of key polynomials, there exists a key polynomial
$\widetilde{Q}$ of degree less than or equal than those of $P$ and
such that $\epsilon\left(P\right)\leq\epsilon\left(\widetilde{Q}\right)$.
Contradiction. This completes the proof.
\end{proof}
\newpage{}

\section{Immediate successors.}
\begin{defn}
\label{def: successeur immediat}Let $Q_{1}$ and $Q_{2}$ be two
key polynomials. We say that $Q_{2}$ is an \emph{immediate successor\index{Immediate successor key polynomial}}
of $Q_{1}$ and we write $Q_{1}<Q_{2}$ if $\epsilon(Q_{1})<\epsilon(Q_{2})$
and if $Q_{2}$ is of minimal degree for this property.
\end{defn}

\begin{rem}
We keep the hypotheses of Example \ref{exa:pc}. Then we have $z<z^{2}-x^{2}y$.
\end{rem}

\begin{defn}
It will be useful to have simpler ways to check if a key polynomial
is an immediate successor of another key polynomial. This is why we
give these two results.
\end{defn}

\begin{prop}
\label{prop: successeurs imm=0000E9diats et in=0000E9galit=0000E9 des valuations}Let
$Q_{1}$ and $Q_{2}$ be two key polynomials. The following are equivalent.
\begin{enumerate}
\item The polynomials $Q_{1}$ and $Q_{2}$ satisfy $Q_{1}<Q_{2}$.
\item We have $\nu_{Q_{1}}(Q_{2})<\nu(Q_{2})$ and $Q_{2}$ is of minimal
degree for this property.
\end{enumerate}
\end{prop}

\begin{proof}
First let us show that 
\[
\epsilon\left(Q_{1}\right)<\epsilon\left(Q_{2}\right)\Rightarrow\nu_{Q_{1}}(Q_{2})<\nu(Q_{2}).
\]

We set $b:=b(Q_{2})=\min\left\{ b\in\mathbb{N}^{\ast}\text{ such that }\frac{\nu(Q_{2})-\nu(\partial_{b}Q_{2})}{b}=\epsilon(Q_{2})\right\} $.

We have 
\[
\begin{array}{ccc}
\epsilon(Q_{1})<\epsilon(Q_{2}) & \Leftrightarrow & b\epsilon(Q_{1})<\nu(Q_{2})-\nu(\partial_{b}Q_{2})\\
 & \Rightarrow & b\epsilon(Q_{1})<\nu(Q_{2})-\nu_{Q_{1}}(\partial_{b}Q_{2})
\end{array}
\]
because for every polynomial $g$, we have $\nu_{Q_{1}}(g)\leq\nu(g)$.

But by Lemma \ref{lem: inegalit=0000E9 epsilon_Q(P)et epsilon(Q)},
$\nu_{Q_{1}}(Q_{2})-\nu_{Q_{1}}(\partial_{b}Q_{2})\leq b\epsilon(Q_{1})$,
so 
\[
\nu_{Q_{1}}(Q_{2})-\nu_{Q_{1}}(\partial_{b}Q_{2})<\nu(Q_{2})-\nu_{Q_{1}}(\partial_{b}Q_{2}).
\]

Then $\nu_{Q_{1}}(Q_{2})<\nu(Q_{2})$.

Now let us show that $\nu_{Q_{1}}(Q_{2})<\nu(Q_{2})\Rightarrow\epsilon\left(Q_{1}\right)<\epsilon\left(Q_{2}\right)$.
Assume, aiming for contradiction, that $\epsilon(Q_{1})\geq\epsilon(Q_{2})$.
Then $\deg(Q_{1})\geq\deg(Q_{2}).$

If we have $\deg(Q_{1})>\deg(Q_{2})$, then $\nu_{Q_{1}}(Q_{2})=\nu(Q_{2})$
and this is a contradiction. Hence we assume that $Q_{1}$ and $Q_{2}$
have same degree.

Let $Q_{2}=Q_{1}+(Q_{2}-Q_{1})$ be the $Q_{1}$-expansion of $Q_{2}$. 

If $\nu(Q_{1})\neq\nu(Q_{2}-Q_{1})$, then
\[
\nu(Q_{2})=\min\left\{ \nu(Q_{1}),\nu(Q_{2}-Q_{1})\right\} =\nu_{Q_{1}}(Q_{2})
\]
 and again we have a contradiction.

So $\nu(Q_{1})=\nu(Q_{2}-Q_{1})=\nu_{Q_{1}}(Q_{2})<\nu(Q_{2})$.

But $\nu(Q_{2})=\nu_{Q_{2}}(Q_{2})\leq\nu_{Q_{1}}(Q_{2})$ by Proposition
\ref{prop: inegalite des valuations tronqu=0000E9es}. Again, this
is a contradiction.

So we showed that $\epsilon(Q_{1})<\epsilon(Q_{2})\Leftrightarrow\nu_{Q_{1}}(Q_{2})<\nu(Q_{2})$.

Let $Q_{2}$ be of minimal degree for the first property. 

Assume the existence of $Q_{3}$ of degree strictly less than $Q_{2}$
such that $\nu_{Q_{1}}(Q_{3})<\nu(Q_{3})$. So $\epsilon(Q_{1})<\epsilon(Q_{3})$,
which is in contradiction with the minimality of the degree of $Q_{2}$
for this property.

So we have $Q_{1}<Q_{2}\Rightarrow\nu_{Q_{1}}(Q_{2})<\nu(Q_{2})$
and $Q_{2}$ is of minimal degree for this property.

Take $Q_{2}$ such that $\nu_{Q_{1}}(Q_{2})<\nu(Q_{2})$ and $Q_{2}$
is of minimal degree for this property. Assume the existence of $Q_{3}$
of degree strictly less than $\deg\left(Q_{2}\right)$ and such that
$\epsilon(Q_{1})<\epsilon(Q_{3})$. By this last property, we have
$\nu_{Q_{1}}(Q_{3})<\nu(Q_{3})$, which is in contradiction with the
minimality of the degree of $Q_{2}$ for this property.

This completes the proof.
\end{proof}
\begin{prop}
\label{prop: successeurs immediats et somme des formes initiales}Let
$Q_{1}$ and $Q_{2}$ be two key polynomials, and let
\[
Q_{2}=\sum\limits _{j\in\Theta}q_{j}Q_{1}^{j}
\]
 be the $Q_{1}$-expansion of $Q_{2}$ . 

The following are equivalent: 
\begin{enumerate}
\item The polynomials $Q_{1}$ and $Q_{2}$ satisfy $Q_{1}<Q_{2}$.
\item We have $\sum\limits _{j\in S_{Q_{1}}\left(Q_{2}\right)}\mathrm{in}_{\nu}\left(q_{j}Q_{1}^{j}\right)=0$
with $Q_{2}$ of minimal degree for this property.
\end{enumerate}
\end{prop}

\begin{proof}
First, let us show that
\[
Q_{1}<Q_{2}\Rightarrow\sum\limits _{j\in S_{Q_{1}}\left(Q_{2}\right)}\mathrm{in}_{\nu}\left(q_{j}Q_{1}^{j}\right)=0.
\]

Assume $Q_{1}<Q_{2}$. By Proposition \ref{prop: successeurs imm=0000E9diats et in=0000E9galit=0000E9 des valuations},
we know that $\nu_{Q_{1}}(Q_{2})<\nu(Q_{2})$. So by definition 
\[
\sum\limits _{j\in S_{Q_{1}}\left(Q_{2}\right)}\mathrm{in}_{\nu}\left(q_{j}Q_{1}^{j}\right)=0.
\]

Furthermore, if $Q_{1}<Q_{2}$, we have that $Q_{2}$ is of minimal
degree for this property by definition of immediate successor.

Now let us show that if $\sum\limits _{j\in S_{Q_{1}}\left(Q_{2}\right)}\mathrm{in}_{\nu}\left(q_{j}Q_{1}^{j}\right)=0$
with $Q_{2}$ of minimal degree for this property, then $Q_{1}<Q_{2}$.

Assume $\sum\limits _{j\in S_{Q_{1}}\left(Q_{2}\right)}\mathrm{in}_{\nu}\left(q_{j}Q_{1}^{j}\right)=0$.
Then
\[
\nu(Q_{2})>\min\limits _{j\in\Theta}\nu(q_{j}Q_{1}^{j})=\nu_{Q_{1}}(Q_{2}),
\]
and so $Q_{2}>Q_{1}$ by Proposition \ref{prop: successeurs imm=0000E9diats et in=0000E9galit=0000E9 des valuations}.
\end{proof}
\begin{rem}
Let $Q_{1}$ and $Q_{2}$ be key polynomials such that $Q_{2}$ is
an immediate successor of $Q_{1}$ and let $Q_{2}=\sum\limits _{j\in\Theta}q_{j}Q_{1}^{j}$
be the $Q_{1}$-expansion of $Q_{2}$ . We set 
\[
\widetilde{Q}_{2}=\sum\limits _{j\in S_{Q_{1}}\left(Q_{2}\right)}q_{j}Q_{1}^{j}.
\]
 We will show that $\widetilde{Q}_{2}$ is an immediate successor
of $Q_{1}$. Then we will always consider ``optimal'' immediate
successor key polynomials. This means, by definition, that all the
terms in their expansion in the powers of the previous key polynomial
are of same value.
\end{rem}

\begin{prop}
\label{prop:on peut le prendre optimal}Let $Q_{1}$ and $Q_{2}$
be key polynomials such that $Q_{2}$ is an immediate successor of
$Q_{1}$ and let $Q_{2}=\sum\limits _{j\in\Theta}q_{j}Q_{1}^{j}$
be the $Q_{1}$-expansion of $Q_{2}$ . We set 
\[
\widetilde{Q}_{2}=\sum\limits _{j\in S_{Q_{1}}\left(Q_{2}\right)}q_{j}Q_{1}^{j}
\]
 Then $\widetilde{Q}_{2}$ is an immediate successor of $Q_{1}$.
\end{prop}

\begin{proof}
First, by definition of $\widetilde{Q}_{2}$, we have $\deg\left(\widetilde{Q}_{2}\right)\leq\deg\left(Q_{2}\right)$.
We are going to show that this inequality is, in fact, an equality.

We have $\sum\limits _{j\in S_{Q_{1}}\left(Q_{2}\right)}\mathrm{in}_{\nu}\left(q_{j}Q_{1}^{j}\right)=\sum\limits _{j\in S_{Q_{1}}\left(\tilde{Q}_{2}\right)}\mathrm{in}_{\nu}\left(q_{j}Q_{1}^{j}\right)=0$.
Since $Q_{2}$ is of minimal degree for this property, we know that
its term of greatest degree appears in this sum. So $\deg_{X}\left(\widetilde{Q}_{2}\right)=\deg_{X}\left(Q_{2}\right)$.

Now let us show that $\epsilon\left(\widetilde{Q}_{2}\right)>\epsilon\left(Q_{1}\right)$. 

Since $\sum\limits _{j\in S_{Q_{1}}\left(\widetilde{Q}_{2}\right)}\mathrm{in}_{\nu}\left(q_{j}Q_{1}^{j}\right)=0$,
we have $\nu_{Q_{1}}\left(\widetilde{Q}_{2}\right)<\nu\left(\widetilde{Q}_{2}\right)$,
and $\widetilde{Q}_{2}$ is still of minimal degree for this property.
Then $S_{Q_{1}}\left(\widetilde{Q}_{2}\right)\neq\left\{ 0\right\} $
and for every non-zero integer $b$, we have $\nu_{Q_{1}}\left(\partial_{b}\widetilde{Q}_{2}\right)=\nu\left(\partial_{b}\widetilde{Q}_{2}\right)$.
By Proposition \ref{prop:S_Q(P)neq0}, there exists a strictly positive
integer $b$ such that $\nu_{Q}\left(P\right)-\nu_{Q}\left(\partial_{b}P\right)=b\epsilon\left(Q\right)$.
So we can use Corollary \ref{cor:comparaison des epsilon} to conclude
that 
\[
\epsilon\left(\widetilde{Q}_{2}\right)>\epsilon\left(Q_{1}\right).
\]

Assume that we already know that $\widetilde{Q}_{2}$ is a key polynomial.
Since $\deg\left(\widetilde{Q}_{2}\right)=\deg\left(Q_{2}\right)$,
we have that $\widetilde{Q}_{2}$ is of minimal degree for the property
$\epsilon\left(\widetilde{Q}_{2}\right)>\epsilon\left(Q_{1}\right)$,
and so $Q_{1}<\widetilde{Q}_{2}$.

It remains to prove that $\widetilde{Q}_{2}$ is a key polynomial.

Assume, aiming for contradiction, that $\widetilde{Q}_{2}$ is not
a key polynomial. Then there exists a polynomial $P\in K[X]$ such
that 
\[
\epsilon\left(P\right)\geq\epsilon\left(\widetilde{Q}_{2}\right)
\]
 and 
\[
\deg_{X}\left(P\right)<\deg_{X}\left(\widetilde{Q}_{2}\right).
\]
 We take $P$ of minimal degree for this property. We can also assume
that $P$ is monic. Let us show that $P$ is a key polynomial.

Let $S\in K[X]$ be a polynomial such that $\epsilon\left(S\right)\geq\epsilon\left(P\right)$.
Then $\epsilon\left(S\right)\geq\epsilon\left(\tilde{Q}_{2}\right)$.
If $\deg_{X}\left(S\right)\geq\deg_{X}\left(\widetilde{Q}_{2}\right)$,
then $\deg_{X}\left(S\right)>\deg_{X}\left(P\right)$ and the proof
is finished. So let us assume that $\deg_{X}\left(S\right)<\deg_{X}\left(\widetilde{Q}_{2}\right)$.

We have $\epsilon\left(S\right)\geq\epsilon\left(\widetilde{Q}_{2}\right)$
and $\deg_{X}\left(S\right)<\deg_{X}\left(\widetilde{Q}_{2}\right)$.
By minimality of the degree of $P$ for this property, we have $\deg_{X}\left(S\right)\geq\deg_{X}\left(P\right)$,
and hence $P$ is a key polynomial.

So there exists a key polynomial $P$ such that 
\[
\epsilon\left(P\right)\geq\epsilon\left(\widetilde{Q}_{2}\right)
\]
 and 
\[
\deg_{X}\left(P\right)<\deg_{X}\left(\widetilde{Q}_{2}\right).
\]

Since $\epsilon\left(\widetilde{Q}_{2}\right)>\epsilon\left(Q_{1}\right)$,
we also have $\epsilon\left(P\right)>\epsilon\left(Q_{1}\right)$.
By minimality of the degree of $Q_{2}$ among the key polynomials
satisfying this inequality, we have $\deg_{X}\left(Q_{2}\right)\leq\deg_{X}\left(P\right)<\deg_{X}\left(\widetilde{Q}_{2}\right)$
which is a contradiction by the equality of the degrees of $Q_{2}$
and $\widetilde{Q}_{2}$. Hence the polynomial $\widetilde{Q}_{2}$
is a key polynomial.
\end{proof}
\begin{defn}
\index{Optimal immediate successor key polynomial}Let $Q_{1}$ and
$Q_{2}$ be two key polynomials such that $Q_{1}<Q_{2}$. We say that
$Q_{2}$ is an \emph{optimal immediate successor} of $Q_{1}$ if all
the terms of its $Q_{1}$-expansion have same value.
\end{defn}

\begin{rem}
Proposition \ref{prop:on peut le prendre optimal} shows how to associate
to every immediate successor $Q_{2}$ of $Q_{1}$ an optimal immediate
successor $\widetilde{Q}_{2}$. 

Hence, if $Q_{1}$ is not maximal in the set of the key polynomials
$\Lambda$, it admits an optimal immediate successor.
\end{rem}

Let $Q\in\Lambda$ be a key polynomial. We note 
\[
M_{Q}:=\left\{ P\in\Lambda\text{ such that }Q<P\right\} .
\]

\begin{defn}
\label{def:limitpc}We assume that $M_{Q}$ does not have a maximal
element and that for every element $P\in M_{Q}$ we have $\deg_{X}\left(P\right)=\deg_{X}\left(Q\right)$.

We also assume that there exists a key polynomial $Q'\in\Lambda$
such that $\epsilon(Q')>\epsilon(M_{Q})$. 

We call a \emph{limit immediat successor\index{Limit immediat successor key polynomial}}
of $Q$ every polynomial $Q'$ of minimal degree which has this property,
and we write $Q<_{\lim}Q'$.
\end{defn}

\begin{prop}
\label{prop: l'un ou l'autre}Let $Q$ and $Q'$ be two key polynomials
such that $\epsilon\left(Q\right)<\epsilon\left(Q'\right)$. Then
there exists a sequence $Q_{1}=Q,\dots,Q_{h}=Q'$ where for every
index $i$, the polynomial $Q_{i+1}$ is either an immediate successor
of $Q_{i}$ or a limit immediate successor of $Q_{i}$.
\end{prop}

\begin{proof}
If $Q'$ is an immediate successor of $Q,$ we are done, so we assume
that $Q'$ is not an immediate successor of $Q$, and we write this
$Q\nless Q'$.

Let us first look at $M_{Q}=M_{Q_{1}}$. If this set has a maximum,
we denote this maximum by $Q_{2}$. We have:

\[
\begin{cases}
Q<Q_{2}\\
\epsilon\left(Q\right)<\epsilon\left(Q'\right)\\
Q\nless Q'
\end{cases}
\]
 and by minimality of the degree of $Q_{2}$ we know that $\deg_{X}\left(Q_{2}\right)<\deg_{X}\left(Q'\right)$.
But $Q'$ is a key polynomial, so $\epsilon\left(Q_{2}\right)<\epsilon\left(Q'\right)$.

Then we have 
\[
\begin{cases}
Q=Q_{1}<Q_{2}\\
\epsilon\left(Q\right)<\epsilon\left(Q_{2}\right)<\epsilon\left(Q'\right)
\end{cases}
\]
 and since $Q<Q_{2}$, we know that $\deg_{X}\left(Q\right)\leq\deg_{Q}\left(Q_{2}\right)$.

We iterate the process as long as $M_{Q_{i}}$ has a maximum. 

Assume that there exists an index $i$ such that $M_{Q_{i}}$ does
not have a maximum. 

Assume that $\epsilon\left(M_{Q_{i}}\right)\nless\epsilon\left(Q'\right)$.
So there exists $g_{i}\in M_{Q_{i}}$ such that $\epsilon\left(g_{i}\right)\geq\epsilon\left(Q'\right)$.
Since $Q'$ is a key polynomial, we know that $\deg_{X}\left(g_{i}\right)\geq\deg_{X}\left(Q'\right)$.

We have: 
\[
\begin{cases}
\epsilon\left(Q_{i}\right)<\epsilon\left(Q'\right)\\
Q_{i}<g_{i}\\
\deg_{X}\left(Q'\right)\leq\deg_{X}\left(g_{i}\right)
\end{cases}
\]
 By definition of immediate successors, we have $Q_{i}<Q'$ and we
set $Q_{i+1}=Q'$. This completes the proof.

Now assume that $\epsilon\left(Q'\right)>\epsilon\left(M_{Q_{i}}\right)$.

Since $\deg_{X}\left(Q\right)\leq\deg_{X}\left(Q_{i}\right)<\deg_{Q}\left(Q'\right)$
for every index $i$, there exists a natural number $N$ such that
for every index $j\ge N$ we have 
\[
\deg_{X}\left(Q_{j}\right)=\deg_{X}\left(Q_{j+1}\right)<\deg_{X}\left(Q'\right).
\]

Let $P\in M_{Q_{N}}$. By construction, $\epsilon\left(P\right)\leq\epsilon\left(Q_{N+1}\right)<\epsilon\left(Q'\right)$.
If $Q'$ is not of minimal degree for this property, then there exists
a key polynomial $P'$ limit immediate successor of $Q_{N}$, of degree
strictly less than the degree of $Q'$. So 
\[
\deg_{X}\left(Q_{N+1}\right)<\deg_{X}\left(P'\right)<\deg_{X}\left(Q'\right).
\]
 Then we replace $Q_{N+1}$ by $P'$ and iterate the process, which
ends because the sequence of the degrees increase strictly.

Otherwise, $Q'$ is of minimal degree among all the key polynomials
such that $\epsilon\left(M_{Q_{N}}\right)<\epsilon\left(Q'\right)$,
so $Q'$ is a limit immediate successor of $Q_{N}$ and the process
ends at $Q_{N+1}=Q'$.

In each case, we construct a family of key polynomials which begins
at $Q$, ends at $Q'$ and such that for every index $i$, the polynomial
$Q_{i+1}$ is either an immediate successor of $Q_{i}$, or a limit
immediate successor of $Q_{i}$. This completes the proof.
\end{proof}
\begin{prop}
\label{prop: l'un ou l'autre optimal}Let $Q$ and $Q'$ be two key
polynomials such that $\epsilon\left(Q\right)<\epsilon\left(Q'\right)$.
Then there exists a sequence $Q_{1}=Q,\dots,Q_{h}=Q'$ where for every
index $i$, the polynomial $Q_{i+1}$ is either an optimal immediate
successor of $Q_{i}$ or a limit immediate successor of $Q_{i}$.
\end{prop}

\begin{proof}
Let $Q_{2}$ be an optimal immediate successor of $Q$. We look at
$M_{Q}=M_{Q_{1}}$. If this set has a maximum, we denote this maximum
by $P$. 

If $\epsilon\left(Q_{2}\right)=\epsilon\left(P\right)$, we set $P=Q_{2}$.
Otherwise, $\epsilon\left(Q_{2}\right)<\epsilon\left(P\right)$. Since
$P$ and $Q_{2}$ are both immediate successors of $Q$, they have
same degree.

Hence $P$ is an immediate successor of $Q_{2}$, of the same degree
as $Q_{2}$. The polynomial $P$ is then an optimal immediate successor
of $Q_{2}$.

So we set $Q_{3}=P$.

In fact, we have a finite sequence of optimal immediate successors
which begins at $Q$ and ends at $P=\max\left\{ M_{Q}\right\} $.

We iterate the process as long as $M_{Q_{i}}$ has a maximum. Assume
that there exists an index $i$ such that $M_{Q_{i}}$ does not have
a maximum. 

Then we do exactly the same thing that we did in the proof ofcProposition
\ref{prop: l'un ou l'autre} and this completes the proof.
\end{proof}
\begin{lem}
\label{lem:coeff dominant et succ imm}Let $Q$ and $Q'$ be two key
polynomials such that $Q<Q'$ and we denote by $Q'=\sum\limits _{j=0}^{m}q_{j}Q^{j}$
the $Q$-expansion of $Q'$. Then $q_{m}=1$.
\end{lem}

\begin{proof}
Since $\epsilon\left(Q\right)<\epsilon\left(Q'\right)$, we know by
Proposition \ref{prop: successeurs immediats et somme des formes initiales}
that $\sum\limits _{j=0}^{m}\mathrm{in}_{\nu}\left(q_{j}Q^{j}\right)=0$.
In fact we have 
\[
\mathrm{in}_{\nu}\left(q_{m}\right)\mathrm{in}_{\nu}\left(Q\right)^{m}+\dots+\mathrm{in}_{\nu}\left(q_{1}\right)\mathrm{in}_{\nu}\left(Q\right)+\mathrm{in}_{\nu}\left(q_{0}\right)=0.
\]
Then, since $\mathrm{in}_{\nu}\left(q_{m}\right)\neq0$, we have
\begin{equation}
\mathrm{in}_{\nu}\left(Q\right)^{m}+\dots+\frac{\mathrm{in}_{\nu}\left(q_{1}\right)}{\mathrm{in}_{\nu}\left(q_{m}\right)}\mathrm{in}_{\nu}\left(Q\right)+\frac{\mathrm{in}_{\nu}\left(q_{0}\right)}{\mathrm{in}_{\nu}\left(q_{m}\right)}=0.\label{eq:premier coefficient egal a 1}
\end{equation}

We set $a:=\deg_{X}\left(Q\right)$ and we consider $G_{<a}$ the
subalgebra of $\mathrm{gr}_{\nu}\left(K\left[X\right]\right)$ generated
by the initial forms of all the polynomials of degree strictly less
than $a$.

Hence $G_{<a}$ is a saturated algebra, and all the coefficients of
the form $\frac{\mathrm{in}_{\nu}\left(q_{i}\right)}{\mathrm{in}_{\nu}\left(q_{m}\right)}$
of the equation $\left(\ref{eq:premier coefficient egal a 1}\right)$
can be represented by polynomials. We denote by $h_{i}$ some liftings,
of degrees strictly less than $a$.

The element $\mathrm{in}_{\nu}\left(Q\right)$ is hence a solution
of a homogeneous monic equation with coefficients in $G_{<a}$ and
whose leading coefficient is $1$.

We consider the polynomial $\widetilde{Q}=Q^{m}+\sum\limits _{j=0}^{m-1}h_{j}Q^{j}$,
with, by hypothesis, $\deg_{X}\left(\widetilde{Q}\right)\leq\deg_{X}\left(Q'\right)$.
By construction we have 
\[
\mathrm{in}_{\nu}\left(Q\right)^{m}+\sum\limits _{j=0}^{m-1}\mathrm{in}_{\nu}\left(h_{j}\right)\mathrm{in}_{\nu}\left(Q\right)^{j}=0
\]
 and by the proof of the proposition \ref{prop: successeurs immediats et somme des formes initiales},
we have $\epsilon\left(\widetilde{Q}\right)>\epsilon\left(Q\right)$.

By minimality of the degree of $Q'$ for this property, if we can
show that $\widetilde{Q}$ is a key polynomial, then we would have
$\deg_{X}\left(Q'\right)=\deg_{X}\left(\widetilde{Q}\right)$ and
so $q_{m}=1$.

Let us show that $\widetilde{Q}$ is a key polynomial.

Assume, aiming for contradiction, that it is not. Then there exists
a polynomial $P$ such that $\epsilon\left(P\right)\geq\epsilon\left(\widetilde{Q}\right)$
and $\deg_{X}\left(P\right)<\deg_{X}\left(\widetilde{Q}\right)$.
We choose $P$ monic and of minimal degree for this property. Let
us show that $P$ is a key polynomial.

Let $S$ be a polynomial such that $\epsilon\left(S\right)\geq\epsilon\left(P\right)$.
Then $\epsilon\left(S\right)\geq\epsilon\left(\widetilde{Q}\right)$. 

If $\deg_{X}\left(S\right)\geq\deg_{X}\left(\widetilde{Q}\right)$,
then, since $\deg_{X}\left(P\right)<\deg_{X}\left(\widetilde{Q}\right)$,
the proof is finished.

So let us assume that $\deg_{X}\left(S\right)<\deg_{X}\left(\widetilde{Q}\right)$.
Then $\epsilon\left(S\right)\geq\epsilon\left(\widetilde{Q}\right)$
and $\deg_{X}\left(S\right)<\deg_{X}\left(\widetilde{Q}\right)$.
By minimality of the degree of $P$ for that property, $\deg_{X}\left(S\right)\geq\deg_{X}\left(P\right)$
and the proof is finished.

So there exists a key polynomial $P$ such that $\epsilon\left(P\right)\geq\epsilon\left(\widetilde{Q}\right)$
and $\deg_{X}\left(P\right)<\deg_{X}\left(\widetilde{Q}\right)$. 

Since $\epsilon\left(\widetilde{Q}\right)>\epsilon\left(Q\right)$,
we have $\epsilon\left(P\right)>\epsilon\left(Q\right)$.

So we have a key polynomial $P$ such that $\epsilon\left(P\right)>\epsilon\left(Q\right)$.
By minimality of the degree of $Q'$ for this property, we know that
$\deg_{X}\left(Q'\right)\leq\deg_{X}\left(P\right)$. But $\deg_{X}\left(P\right)<\deg_{X}\left(\widetilde{Q}\right)$,
and this implies that $\deg_{X}\left(Q'\right)<\deg_{X}\left(\widetilde{Q}\right)$,
which is a contradiction.

Thus $\widetilde{Q}$ is a key polynomial.
\end{proof}
\begin{prop}
\label{prop: inegalite des indices} Let $Q$ and $Q'$ be two key
polynomials such that 
\[
\epsilon\left(Q\right)<\epsilon\left(Q'\right).
\]
Let $c$ and $c'$ be two polynomials of degrees strictly less than
$\deg_{X}Q'$ and let $j$ and $j'$ be two integers such that :

\[
\begin{cases}
\nu_{Q}\left(c\right)=\nu\left(c\right)\\
\nu_{Q}\left(c'\right)=\nu\left(c'\right)\\
j\leq j'\\
\nu_{Q}\left(c\left(Q'\right)^{j}\right)\leq\nu_{Q}\left(c'\left(Q'\right)^{j'}\right).
\end{cases}
\]
 Then: 
\[
\nu\left(c\left(Q'\right)^{j}\right)\leq\nu\left(c'\left(Q'\right)^{j'}\right).
\]

If, in addition, either $j<j'$ or $\nu_{Q}\left(c\left(Q'\right)^{j}\right)<\nu_{Q}\left(c'\left(Q'\right)^{j'}\right)$,
then 
\[
\nu\left(c\left(Q'\right)^{j}\right)<\nu\left(c'\left(Q'\right)^{j'}\right).
\]
\end{prop}

\begin{proof}
We know that $\nu_{Q}\left(Q'\right)\leq\nu\left(Q'\right)$, hence
\[
\nu\left(Q'\right)-\nu_{Q}\left(Q'\right)\geq0.
\]
 Since we assumed that $j\leq j'$, we have
\[
j\left(\nu\left(Q'\right)-\nu_{Q}\left(Q'\right)\right)\leq j'\left(\nu\left(Q'\right)-\nu_{Q}\left(Q'\right)\right).
\]

Furthermore, we know that $\nu_{Q}\left(c\left(Q'\right)^{j}\right)\leq\nu_{Q}\left(c'\left(Q'\right)^{j'}\right)$,
hence 
\[
\nu_{Q}\left(c\left(Q'\right)^{j}\right)+j\left(\nu\left(Q'\right)-\nu_{Q}\left(Q'\right)\right)\leq\nu_{Q}\left(c'\left(Q'\right)^{j'}\right)+j'\left(\nu\left(Q'\right)-\nu_{Q}\left(Q'\right)\right).
\]

So we have the inequality 
\[
\nu_{Q}\left(c\right)+j\nu_{Q}\left(Q'\right)+j\nu\left(Q'\right)-j\nu_{Q}\left(Q'\right)\leq\nu_{Q}\left(c'\right)+j'\nu_{Q}\left(Q'\right)+j'\nu\left(Q'\right)-j'\nu_{Q}\left(Q'\right).
\]
Equivalently, $\nu_{Q}\left(c\right)+j\nu\left(Q'\right)\leq\nu_{Q}\left(c'\right)+j'\nu\left(Q'\right)$.

But $\nu_{Q}\left(c\right)=\nu\left(c\right)$ and $\nu_{Q}\left(c'\right)=\nu\left(c'\right)$,
so $\nu\left(c\left(Q'\right)^{j}\right)\leq\nu\left(c'\left(Q'\right)^{j'}\right)$.

If, in addition, either $j<j'$ or $\nu_{Q}\left(c\left(Q'\right)^{j}\right)<\nu_{Q}\left(c'\left(Q'\right)^{j'}\right)$,
then we have $\nu\left(c\left(Q'\right)^{j}\right)<\nu\left(c'\left(Q'\right)^{j'}\right)$.
\end{proof}
\begin{lem}
\label{lem: relations entre les Q et Q' dev}Let $Q$ and $Q'$ be
two polynomials such that 
\[
\epsilon\left(Q\right)<\epsilon\left(Q'\right)
\]
 and let $f\in K[X]$ be a polynomial whose $Q'$-expansion is $f=\sum\limits _{j=0}^{r}f_{j}\left(Q'\right)^{j}$.
Then
\[
\nu_{Q}\left(f\right)=\min\limits _{0\leq j\leq r}\left\{ \nu_{Q}\left(f_{j}\left(Q'\right)^{j}\right)\right\} .
\]

If we set 
\[
T_{Q,Q'}\left(f\right):=\left\{ j\in\left\{ 0,\dots,r\right\} \text{ such that }\nu_{Q}\left(f_{j}\left(Q'\right)^{j}\right)=\nu_{Q}\left(f\right)\right\} ,
\]
 then we have
\[
\mathrm{in}_{\nu_{Q}}\left(f\right)=\sum\limits _{j\in T_{Q,Q'}\left(f\right)}\mathrm{in}_{\nu_{Q}}\left(f_{j}\left(Q'\right)^{j}\right).
\]
\end{lem}

\begin{proof}
Only for the purposes of this proof, we will write
\[
\nu'\left(f\right):=\min\limits _{0\leq j\leq r}\left\{ \nu_{Q}\left(f_{j}\left(Q'\right)^{j}\right)\right\} 
\]
 and 
\[
T'\left(f\right):=\left\{ j\in\left\{ 0,\dots,r\right\} \text{ such that }\nu_{Q}\left(f_{j}\left(Q'\right)^{j}\right)=\nu'\left(f\right)\right\} .
\]
Let us show that $\nu_{Q}\left(f\right)=\nu'\left(f\right)$.

First, we have 

\[
\begin{array}{ccc}
\nu_{Q}\left(\sum\limits _{j\in T'\left(f\right)}f_{j}\left(Q'\right)^{j}\right) & \geq & \min\limits _{j\in T'\left(f\right)}\nu_{Q}\left(f_{j}\left(Q'\right)^{j}\right)\\
 & = & \min\limits _{j\in T'\left(f\right)}\nu'\left(f\right)\\
 & = & \nu'\left(f\right).
\end{array}
\]

Set $b'=\max T'\left(f\right)$ and $b=\delta_{Q}\left(f_{b'}\right)$.
In other words 
\[
b=\max\left\{ j\in\left\{ 0,\dots,n\right\} \text{ such that }\nu\left(a_{j}Q^{j}\right)=\nu_{Q}\left(f_{b'}\right)\right\} 
\]
 where $f_{b'}=\sum\limits _{j=0}^{n}a_{j}Q^{j}$. Hence, the expression
$\sum\limits _{j\in T'\left(f\right)}f_{j}\left(Q'\right)^{j}$ contains
the term 
\[
a_{b}c_{\deg_{Q}Q'}Q^{b+b'\deg_{Q}Q'}.
\]

Then for every $j\in\left\{ 0,\dots,r\right\} $ such that $f_{j}\neq0$,
we have:

\[
\begin{array}{ccc}
\nu_{Q}\left(f_{j}\left(Q'\right)^{j}\right) & \geq & \min\limits _{0\leq j\leq r}\left\{ \nu_{Q}\left(f_{j}\left(Q'\right)^{j}\right)\right\} \\
 & = & \nu'\left(f\right)\\
 & = & \nu_{Q}\left(f_{i}\left(Q'\right)^{i}\right)
\end{array}
\]

for every index $i\in T'\left(f\right)$. So in particular, 

\[
\begin{array}{ccc}
\nu_{Q}\left(f_{j}\left(Q'\right)^{j}\right) & \geq & \nu_{Q}\left(f_{b'}\left(Q'\right)^{b'}\right)\\
 & = & \nu_{Q}\left(f_{b'}\right)+\nu_{Q}\left(\left(Q'\right)^{b'}\right)\\
 & = & \nu\left(a_{b}Q^{b}\right)+\nu_{Q}\left(\left(Q'\right)^{b'}\right)\\
 & = & \nu\left(a_{b}Q^{b}\right)+\nu\left(c_{\deg_{Q}Q'}Q^{b'\deg_{Q}Q'}\right)\\
 & = & \nu\left(a_{b}c_{\deg_{Q}Q'}Q^{b+b'\deg_{Q}Q'}\right)
\end{array}
\]

with strict inequality if $j\notin T'\left(f\right)$. 

So 
\[
\nu\left(a_{b}c_{\deg_{Q}Q'}Q^{b+b'\deg_{Q}Q'}\right)=\nu'\left(f\right)
\]
 and 
\[
\nu_{Q}\left(\sum\limits _{j\notin T'\left(f\right)}f_{j}\left(Q'\right)^{j}\right)>\nu'\left(f\right).
\]

By maximality of $b$ and $b'$, the term $a_{b}c_{\deg_{Q}Q'}Q^{b+b'\deg_{Q}Q'}$
cannot be cancelled and so $\nu_{Q}\left(f\right)=\nu\left(a_{b}c_{\deg_{Q}Q'}Q^{b+b'\deg_{Q}Q'}\right)=\nu'\left(f\right)$.
In other words $\nu_{Q}\left(f\right)=\min\limits _{0\leq j\leq r}\left\{ \nu_{Q}\left(f_{j}\left(Q'\right)^{j}\right)\right\} $.
So we also have 
\[
T'\left(f\right)=T_{Q,Q'}\left(f\right).
\]

Then $\sum\limits _{j\in T'\left(f\right)}\mathrm{in}_{\nu_{Q}}\left(f_{j}\left(Q'\right)^{j}\right)$
is a non-zero element of $G_{\nu_{Q}}$, equal to $\mathrm{in}_{\nu_{Q}}\left(f\right)$.
This completes the proof.
\end{proof}
\begin{cor}
\label{cor:forme initiale de tout poly}Let $Q$ and $Q'$ be two
key polynomials such that 
\[
\epsilon\left(Q\right)<\epsilon\left(Q'\right)
\]
 and let
\[
f=\sum\limits _{j=0}^{r}f_{j}\left(Q'\right)^{j}=\sum\limits _{j=0}^{n}a_{j}Q^{j}
\]
 be the $Q'$ and $Q$-expansions of an element $f\in K\left[X\right]$.
We set
\[
\theta:=\min T_{Q,Q'}\left(f\right)=\min\left\{ j\in\left\{ 0,\dots,r\right\} \text{ such that }\nu_{Q}\left(f_{j}\left(Q'\right)^{j}\right)=\nu_{Q}\left(f\right)\right\} 
\]
and we assume that $\nu_{Q}\left(f_{\delta_{Q'}\left(f\right)}\right)=\nu\left(f_{\delta_{Q'}\left(f\right)}\right)$
and that $\nu_{Q}\left(f_{\theta}\right)=\nu\left(f_{\theta}\right)$.

Then:
\begin{enumerate}
\item $\delta_{Q'}\left(f\right)\deg_{Q}Q'\leq\delta_{Q}\left(f\right)$,
and so $\delta_{Q'}\left(f\right)\leq\delta_{Q}\left(f\right)$.
\item If $\delta_{Q}\left(f\right)=\delta_{Q'}\left(f\right)$, we set $\delta:=\delta_{Q}\left(f\right)$
and then 
\[
\deg_{Q}Q'=1,
\]
 
\[
T_{Q,Q'}\left(f\right)=\left\{ \delta\right\} 
\]
 and 
\[
\mathrm{in}_{\nu_{Q}}\left(f\right)=\left(\mathrm{in}_{\nu_{Q}}a_{\delta}\right)\left(\mathrm{in}_{\nu_{Q}}Q'\right)^{\delta}.
\]
\end{enumerate}
\end{cor}

\begin{proof}
First let us show the point $1$.

By the proof of the previous Lemma, we know that 
\[
\theta\deg_{Q}Q'\leq\delta_{Q}\left(f\right).
\]

Furthermore, 

\[
\nu_{Q}\left(f_{\delta_{Q'}\left(f\right)}\right)=\nu\left(f_{\delta_{Q'}\left(f\right)}\right),
\]
 
\[
\nu_{Q}\left(f_{\theta}\right)=\nu\left(f_{\theta}\right).
\]
 By definition of $\delta=\delta_{Q'}\left(f\right)$, we have $\nu\left(f_{\delta}\left(Q'\right)^{\delta}\right)\leq\nu\left(f_{\theta}\left(Q'\right)^{\theta}\right)$.
We know by Lemma \ref{lem: relations entre les Q et Q' dev} that
$\nu_{Q}\left(f\right)=\min\limits _{0\leq j\leq r}\left\{ \nu_{Q}\left(f_{j}\left(Q'\right)^{j}\right)\right\} $.
Since $\theta=\min T_{Q,Q'}\left(f\right)$, we have 
\[
\nu_{Q}\left(f_{\theta}\left(Q'\right)^{\theta}\right)=\nu_{Q}\left(f\right)=\min\limits _{0\leq j\leq r}\left\{ \nu_{Q}\left(f_{j}\left(Q'\right)^{j}\right)\right\} .
\]
 Hence $\nu_{Q}\left(f_{\theta}\left(Q'\right)^{\theta}\right)\leq\nu_{Q}\left(f_{\delta}\left(Q'\right)^{\delta}\right)$. 

Then, since $\nu_{Q}\left(f_{\theta}\right)=\nu\left(f_{\theta}\right)$
and $\nu_{Q}\left(f_{\delta}\right)=\nu\left(f_{\delta}\right)$:

\[
\begin{array}{c}
\nu_{Q}\left(f_{\theta}\left(Q'\right)^{\theta}\right)\leq\nu_{Q}\left(f_{\delta}\left(Q'\right)^{\delta}\right)\\
\Leftrightarrow\nu_{Q}\left(f_{\theta}\right)+\theta\nu_{Q}\left(Q'\right)\leq\nu_{Q}\left(f_{\delta}\right)+\delta\nu_{Q}\left(Q'\right)\\
\Leftrightarrow\nu\left(f_{\theta}\right)+\theta\nu_{Q}\left(Q'\right)\leq\nu\left(f_{\delta}\right)+\delta\nu_{Q}\left(Q'\right).
\end{array}
\]

Assume we have equality on $\nu$, it means that $\nu\left(f_{\theta}\left(Q'\right)^{\theta}\right)=\nu\left(f_{\delta}\left(Q'\right)^{\delta}\right)$.
So $\nu\left(f_{\theta}\right)=\nu\left(f_{\delta}\right)+\delta\nu\left(Q'\right)-\theta\nu\left(Q'\right)$
and

\[
\begin{array}{c}
\nu_{Q}\left(f_{\theta}\left(Q'\right)^{\theta}\right)\leq\nu_{Q}\left(f_{\delta}\left(Q'\right)^{\delta}\right)\\
\Leftrightarrow\nu\left(f_{\delta}\right)+\delta\nu\left(Q'\right)-\theta\nu\left(Q'\right)+\theta\nu_{Q}\left(Q'\right)\leq\nu\left(f_{\delta}\right)+\delta\nu_{Q}\left(Q'\right)\\
\Leftrightarrow\left(\delta-\theta\right)\nu\left(Q'\right)\leq\left(\delta-\theta\right)\nu_{Q}\left(Q'\right).
\end{array}
\]

Since we know that $\epsilon\left(Q\right)<\epsilon\left(Q'\right)$,
by the proof of Proposition \ref{prop: successeurs imm=0000E9diats et in=0000E9galit=0000E9 des valuations},
we know that $\nu_{Q}\left(Q'\right)<\nu\left(Q'\right)$ and then
$\delta-\theta\leq0$, that is $\delta\leq\theta$.

Otherwise we have $\nu\left(f_{\delta}\left(Q'\right)^{\delta}\right)<\nu\left(f_{\theta}\left(Q'\right)^{\theta}\right)$. 

Then the following four conditions hold:

\[
\begin{cases}
\nu_{Q}\left(f_{\theta}\right)=\nu\left(f_{\theta}\right)\\
\nu_{Q}\left(f_{\delta}\right)=\nu\left(f_{\delta}\right)\\
\nu_{Q}\left(f_{\theta}\left(Q'\right)^{\theta}\right)\leq\nu_{Q}\left(f_{\delta}\left(Q'\right)^{\delta}\right)\\
\nu\left(f_{\delta}\left(Q'\right)^{\delta}\right)<\nu\left(f_{\theta}\left(Q'\right)^{\theta}\right).
\end{cases}
\]
 By the contrapositive of Proposition \ref{prop: inegalite des indices},
we deduce that $\delta<\theta$.

In each case, we have $\delta\leq\theta$. Then since $\theta\deg_{Q}Q'\leq\delta_{Q}\left(f\right)$,
we know that $\delta\deg_{Q}Q'\leq\delta_{Q}\left(f\right)$. So in
particular $\delta_{Q'}\left(f\right)\leq\delta_{Q}\left(f\right)$.

Now let us show the point $2$. 

Assume $\delta_{Q'}\left(f\right)=\delta_{Q}\left(f\right)=\delta$.
We just saw that $\delta_{Q'}\left(f\right)\deg_{Q}Q'\leq\delta_{Q}\left(f\right)$,
so we have $\deg_{Q}Q'=1$. Then $Q'=Q+b$ with $b$ a polynomial
of degree strictly less than the degree of $Q$.

We know by the proof of point $1$ that $\delta\leq\theta$. Furthermore,
we know that $\theta\deg_{Q}Q'\leq\delta_{Q}\left(f\right)=\delta$,
in other words $\theta\leq\delta$ since $\deg_{Q}Q'=1$.

Hence $\delta\leq\theta\leq\delta$, hence $\theta=\delta=\min T_{Q,Q'}\left(f\right)$.
We now have to prove that for every index $j>\delta$, we have $j\notin T_{Q,Q'}\left(f\right)$.
Equivalently, that:
\[
\nu_{Q}\left(f_{j}\left(Q'\right)^{j}\right)>\nu_{Q}\left(f\right)=\min\limits _{0\leq i\leq r}\left\{ \nu_{Q}\left(f_{i}\left(Q'\right)^{i}\right)\right\} .
\]
 And then we will have $T_{Q,Q'}\left(f\right)=\left\{ \delta\right\} $.

So let $j>\delta$. By definition of $\delta_{Q}\left(f\right)$ and
$\delta_{Q'}\left(f\right)$, we know that $\nu\left(f_{j}\left(Q'\right)^{j}\right)>\nu_{Q'}\left(f\right)$
and $\nu\left(a_{j}Q^{j}\right)>\nu_{Q}\left(f\right)$.

Furthermore, since $\delta\in T_{Q,Q'}\left(f\right)$, we have $\nu_{Q}\left(f_{\delta}\left(Q'\right)^{\delta}\right)=\nu_{Q}\left(f\right)$.
We want to prove that $\nu_{Q}\left(f_{j}\left(Q'\right)^{j}\right)>\nu_{Q}\left(f_{\delta}\left(Q'\right)^{\delta}\right)$
for every index $j\in\left\{ \delta+1,\dots,r\right\} $.

We know that: 

\[
\begin{cases}
\nu\left(f_{\delta}\left(Q'\right)^{\delta}\right)=\nu_{Q'}\left(f\right)<\nu\left(f_{j}\left(Q'\right)^{j}\right)\\
\nu_{Q}\left(f_{\delta}\right)=\nu\left(f_{\delta}\right) & \text{because }\deg_{X}\left(f_{\delta}\right)<\deg_{X}\left(Q'\right)=\deg_{X}\left(Q\right)\\
\nu_{Q}\left(f_{j}\right)=\nu\left(f_{j}\right) & \text{because }\deg_{X}\left(f_{j}\right)<\deg_{X}\left(Q'\right)=\deg_{X}\left(Q\right)\\
\delta<j
\end{cases}
\]

By the contrapositive of Proposition \ref{prop: inegalite des indices},
we have 
\[
\nu_{Q}\left(f_{\delta}\left(Q'\right)^{\delta}\right)<\nu_{Q}\left(f_{j}\left(Q'\right)^{j}\right).
\]

So we have $T_{Q,Q'}\left(f\right)=\left\{ \delta\right\} $.

By Lemma \ref{lem: relations entre les Q et Q' dev}, we have 

\[
\begin{array}{ccc}
\mathrm{in}_{\nu_{Q}}\left(f\right) & = & \sum\limits _{j\in T_{Q,Q'}\left(f\right)}\mathrm{in}_{\nu_{Q}}\left(f_{j}\left(Q'\right)^{j}\right)\\
 & = & \mathrm{in}_{\nu_{Q}}\left(f_{\delta}\left(Q'\right)^{\delta}\right)\\
 & = & \mathrm{in}_{\nu_{Q}}\left(f_{\delta}\right)\left(\mathrm{in}_{\nu_{Q}}\left(Q'\right)\right)^{\delta}.
\end{array}
\]
\end{proof}
\begin{thm}
\label{thm: delta et pc limite}Let $Q$ and $Q'$ be two key polynomials
such that 
\[
\epsilon\left(Q\right)<\epsilon\left(Q'\right).
\]
We recall that $\mathrm{char}\left(k_{\nu}\right)=0$. If $Q'$ is
a limit immediate successor of $Q$, then $\delta_{Q}\left(Q'\right)=1$.
\end{thm}

\begin{proof}
We give a proof by contradiction. Assume that $\delta_{Q}\left(Q'\right)>1$.
Among all the couples $\left(Q,Q'\right)$ such that $Q'$ is a limit
immediat successor of $Q$ and such that $\delta_{Q}\left(Q'\right)>1$,
we choose $Q$ and $Q'$ such that $\deg\left(Q'\right)-\deg\left(Q\right)$
is minimal.

By definition of a limit immediate successor, for every sequence of
immediate successors $\left(Q_{i}\right)_{i\in\mathbb{N}^{\ast}}$
with $Q_{1}=Q$, we have $Q_{i}\neq Q'$ for every non-zero index
$i$. By definition of limit key polynomials and by hypothesis, we
know that $\deg\left(Q'\right)-\deg\left(Q\right)$ is minimal for
this property.

If we find a polynomial $\widetilde{Q}$ such that 
\[
\epsilon\left(Q\right)<\epsilon\left(\widetilde{Q}\right)<\epsilon\left(Q'\right)
\]
 and $\deg\left(Q\right)<\deg\left(\widetilde{Q}\right)<\deg\left(Q'\right)$,
then by minimality of $\deg\left(Q'\right)-\deg\left(Q\right)$, we
know that there exists a finite sequence of immediate successors between
$Q$ and $\widetilde{Q}$ and that there exists a finite sequence
of immediate successors between $\widetilde{Q}$ and $Q'$. Then we
have a finite sequence of immediate successors between $Q$ and $Q'$,
which is a contradiction.

Hence there exists a key polynomial $\widetilde{Q}$ such that
\[
\epsilon\left(Q\right)<\epsilon\left(\widetilde{Q}\right)<\epsilon\left(Q'\right)
\]
 and $\deg\left(\widetilde{Q}\right)<\deg\left(Q'\right)$, and so
$\deg\left(Q\right)=\deg\left(\widetilde{Q}\right)$.

Let $\widetilde{Q}$ be a such key polynomial. We have $\widetilde{Q}:=Q-a$
where $a$ is a polynomial of degree strictly less than the degree
of $Q$. 

Since $\epsilon\left(Q\right)<\epsilon\left(\widetilde{Q}\right)$,
by Proposition \ref{prop: successeurs immediats et somme des formes initiales},
we know that $\mathrm{in}_{\nu}\left(Q\right)=\mathrm{in}_{\nu}\left(a\right)$.

Consider the $Q$-expansion $\sum\limits _{j=0}^{n}a_{j}Q^{j}$ of
$Q'$. We may assume that $\delta_{Q}\left(Q'\right)=\delta_{\widetilde{Q}}\left(Q'\right)$
and we set $\delta:=\delta_{Q}\left(Q'\right)$.

By Corollary \ref{cor:forme initiale de tout poly}, we know that
$\mathrm{in}_{\nu_{Q}}\left(Q'\right)=\mathrm{in}_{\nu_{Q}}\left(a_{\delta}\right)\mathrm{in}_{\nu_{Q}}\left(\widetilde{Q}\right)^{\delta}$.
In other words $\mathrm{in}_{\nu_{Q}}\left(Q'\right)=\mathrm{in}_{\nu_{Q}}\left(a_{\delta}\right)\mathrm{in}_{\nu_{Q}}\left(Q-a\right)^{\delta}$.

Furthermore, $\partial Q'=\sum\limits _{j=0}^{n}\left[\partial\left(a_{j}\right)Q^{j}+a_{j}jQ^{j-1}\partial Q\right]$.

We first show that the terms $\partial\left(a_{j}\right)Q^{j}$ do
not appear in $\mathrm{in}_{\nu}\left(\partial Q'\right)$. So let
$j\in\left\{ 0,\dots,n\right\} $.

We have 
\[
\begin{array}{ccc}
\nu_{Q}\left(\partial a_{j}\right) & = & \nu\left(\partial a_{j}\right)\\
 & \geq & \nu\left(a_{j}\right)-\epsilon\left(a_{j}\right).
\end{array}
\]
But $Q$ is a key polynomial and $a_{j}$ is of degree strictly less
than the degree of $Q$ since it is a coefficient of a $Q$-expansion.
Then $\epsilon\left(a_{j}\right)<\epsilon\left(Q\right)$.

So 
\[
\nu_{Q}\left(\partial a_{j}\right)>\nu\left(a_{j}\right)-\epsilon\left(Q\right)=\nu_{Q}\left(a_{j}\right)-\epsilon\left(Q\right).
\]

By the proof of Proposition \ref{prop:S_Q(P)neq0}, we know that,
since we are in characteristic zero,
\[
\nu_{Q}\left(Q\right)-\nu_{Q}\left(\partial Q\right)=\epsilon\left(Q\right).
\]

Then $\nu_{Q}\left(\partial a_{j}\right)>\nu_{Q}\left(a_{j}\right)-\nu_{Q}\left(Q\right)+\nu_{Q}\left(\partial Q\right)$.
In fact, 
\[
\nu_{Q}\left(\partial a_{j}\right)+\nu_{Q}\left(Q\right)>\nu_{Q}\left(a_{j}\right)+\nu_{Q}\left(\partial Q\right).
\]

It means that $\nu_{Q}\left(Q\partial a_{j}\right)>\nu_{Q}\left(a_{j}\partial Q\right)$,
and adding $\nu_{Q}\left(Q^{j-1}\right)$ to each side, we obtain:

\[
\nu_{Q}\left(Q^{j}\partial a_{j}\right)>\nu_{Q}\left(a_{j}Q^{j-1}\partial Q\right)=\nu_{Q}\left(ja_{j}Q^{j-1}\partial Q\right).
\]

So 
\[
\mathrm{in}_{\nu_{Q}}\left(\partial Q'\right)=\mathrm{in}_{\nu_{Q}}\left(\sum\limits _{j=1}^{n}\left[ja_{j}Q^{j-1}\partial Q\right]\right).
\]

Even though the expression $\sum\limits _{j=1}^{n}\left[ja_{j}Q^{j-1}\partial Q\right]$
need not be a $Q$-expansion, since $a_{j}$ and $\partial Q$ are
of degrees strictly less than the degree of $Q$ in characteristic
zero, by Lemma \ref{lem:paslechoixfaut le mettre}, the $\nu_{Q}$-initial
form of $a_{j}\partial Q$ is equal to the initial form of its remainder
after the Euclidean division by $Q$. So we conserve this expression
and consider it a substitute of a $Q$-expansion.

Now let us prove that $\delta_{Q}\left(\partial Q'\right)=\delta-1$. 

Replacing $Q$ by $\widetilde{Q}$ in the computation of the initial
form of $Q'$ with respect to $Q$ (respectively $\widetilde{Q}$)
does not change the problem, and we assume that $\delta$ stabilizes
starting with $Q$. Then, if $\delta_{Q}\left(\partial Q'\right)=\delta-1$,
we would also have $\delta_{\widetilde{Q}}\left(\partial Q'\right)=\delta-1$.

Let $j>\delta$. Let us first show that
\[
\nu_{Q}\left(ja_{j}Q^{j-1}\partial Q\right)>\nu_{Q}\left(\delta a_{\delta}Q^{\delta-1}\partial Q\right).
\]
 It is enough to show that 
\[
\nu_{Q}\left(ja_{j}Q^{j-1}\right)>\nu_{Q}\left(\delta a_{\delta}Q^{\delta-1}\right).
\]

But by definition of $\delta$, we have $\nu_{Q}\left(a_{j}Q^{j}\right)>\nu_{Q}\left(a_{\delta}Q^{\delta}\right)$.
So
\[
\nu_{Q}\left(a_{j}Q^{j-1}\right)>\nu_{Q}\left(a_{\delta}Q^{\delta-1}\right)
\]
 hence $\nu_{Q}\left(ja_{j}Q^{j-1}\right)>\nu_{Q}\left(\delta a_{\delta}Q^{\delta-1}\right)$.

We now have to prove that the value of the term $\delta-1$ is minimal.

Let $j<\delta$. We know that $\nu_{Q}\left(a_{j}Q^{j}\right)=\nu_{Q}\left(a_{\delta}Q^{\delta}\right)$,
and hence
\[
\nu_{Q}\left(a_{j}Q^{j-1}\partial Q\right)=\nu_{Q}\left(a_{\delta}Q^{\delta-1}\partial Q\right).
\]
 So $\nu_{Q}\left(ja_{j}Q^{j-1}\partial Q\right)=\nu_{Q}\left(\delta a_{\delta}Q^{\delta-1}\partial Q\right)$
since we are in characteristic zero.

So we do have $\delta_{Q}\left(\partial Q'\right)=\delta_{\widetilde{Q}}\left(\partial Q'\right)=\delta-1$.
By Corollary \ref{cor:forme initiale de tout poly}, we have:
\[
\mathrm{in}_{\nu_{Q}}\left(\partial Q'\right)=\mathrm{in}_{\nu_{Q}}\left(\delta a_{\delta}\partial Q\right)\mathrm{in}_{\nu_{Q}}\left(\widetilde{Q}\right)^{\delta-1}.
\]

In other words
\[
\mathrm{in}_{\nu_{Q}}\left(\partial Q'\right)=\delta\mathrm{in}_{\nu_{Q}}\left(a_{\delta}\partial Q\right)\mathrm{in}_{\nu_{Q}}\left(Q-a\right)^{\delta-1}.
\]

We know that $\nu_{Q}\left(Q-a\right)<\nu\left(Q-a\right)$. Then,
since $\delta>1$,
\[
\nu_{Q}\left(\delta a_{\delta}\partial Q\left(Q-a\right)^{\delta-1}\right)<\nu\left(\delta a_{\delta}\partial Q\left(Q-a\right)^{\delta-1}\right).
\]

It means that the image by $\varphi\colon\mathrm{gr}_{\nu_{Q}}K\left[x\right]\to\mathrm{gr}_{\nu}K\left[x\right]$
of 
\[
\mathrm{in}_{\nu_{Q}}\left(\delta a_{\delta}\partial Q\left(Q-a\right)^{\delta-1}\right)
\]
 is zero. Then, the image by $\varphi$ of $\mathrm{in}_{\nu_{Q}}\left(\partial Q'\right)$
is zero, and so 
\[
\nu_{Q}\left(\partial Q'\right)<\nu\left(\partial Q'\right).
\]

By the proof of Proposition \ref{prop: successeurs imm=0000E9diats et in=0000E9galit=0000E9 des valuations},
we have $\epsilon\left(Q\right)<\epsilon\left(\partial Q'\right)$.
But we know that $\deg\left(\partial Q'\right)<\deg\left(Q'\right)$,
and since $Q'$ is a key polynomial, we have $\epsilon\left(\partial Q'\right)<\epsilon\left(Q'\right)$.

More generally, the above argument holds if we replace $Q$ by any
key polynomial $\widetilde{Q}$ of the same degree as $Q$.

So for every key polynomial $\widetilde{Q}$ of the same degree as
$\deg\left(Q\right)$, we have $\epsilon\left(\widetilde{Q}\right)<\epsilon\left(\partial Q'\right)$.

In fact, $\epsilon\left(Q\right)<\epsilon\left(\partial Q'\right)<\epsilon\left(Q'\right)$
and $\deg\left(\partial Q'\right)<\deg\left(Q'\right)$. So if we
show that $\partial Q'$ is a key polynomial, we will have 
\[
\deg\left(Q\right)=\deg\left(\partial Q'\right).
\]

Let us show that $\partial Q'$ is a key polynomial. We assume, aiming
for contradiction, that it is not. There exists a polynomial $P$
such that $\epsilon\left(P\right)\geq\epsilon\left(\partial Q'\right)$
and $\deg\left(P\right)<\deg\left(\partial Q'\right)$. We choose
$P$ of minimal degree for this property. Using the same idea as before,
we can show that $P$ is a key polynomial.

We have $\deg\left(P\right)<\deg\left(\partial Q'\right)$, hence
$\deg\left(P\right)<\deg\left(Q'\right)$ and since $Q'$ is a key
polynomial, we have $\epsilon\left(P\right)<\epsilon\left(Q'\right)$.

Since $\epsilon\left(P\right)\geq\epsilon\left(\partial Q'\right)$,
we have $\epsilon\left(P\right)>\epsilon\left(Q\right)$.

Thus we have another key polynomial $P$ such that $\epsilon\left(Q\right)<\epsilon\left(P\right)<\epsilon\left(Q'\right)$
and $\deg\left(P\right)<\deg\left(Q'\right)$. Then $\deg\left(P\right)=\deg\left(Q\right)$.
Hence the polynomial $P$ is a key polynomial of same degree as $Q$,
and so $\epsilon\left(P\right)<\epsilon\left(\partial Q'\right)$,
which is a contradiction.

We have proved that $\partial Q'$ is a key polynomial. Then $\deg\left(Q\right)=\deg\left(\partial Q'\right)$.
But then $\epsilon\left(\partial Q'\right)<\epsilon\left(\partial Q'\right)$
and this in a contradiction. This completes the proof.
\end{proof}
\newpage{}

\part{Simultaneous local uniformization in the case of rings essentially
of finite type over a field.}

The objective of this part is to give a proof of the local uniformization
in the case of rings essentially of finite type over a field of zero
characteristic without any restriction on the rank of the valuation.
The proof of the local uniformization is well known in characteristic
zero. It has been proved for the first time by Zariski in 1940 (\cite{Z2})
in every dimension. The benefit of our proof is to present a universal
construction which works for all the elements of the regular ring
we start with, and in which the strict transforms of key polynomials
become coordinates after blowing up. Thus we will have an infinite
sequence of blow-ups given explicitly, together with regular systems
of parameters of the local rings appearing in the sequence, and which
eventually monomializes every element of our algebra essentially of
finite type.

To do this, we will proceed in several steps. Let us give the idea.

Let $k$ be a field of characteristic zero, $R$ a regular local $k$-algebra
essentially of finite type, with residual field $k$. Let $u=(u_{1},\ldots,u_{n})$
be a regular system of parameters of $R$, $\nu$ a valuation centered
in $R$, $\Gamma$ the value group of $\nu$ and $K=k(u_{1},\dots,u_{n-1})$.
We assume that $k=k_{\nu}$. This property is preserved under blowings-up.
Thus every ring that will appear in our local blowing-up sequence
along the valuation $\nu$ will have the same residue field: $k$.

We will construct a single sequence of blowings-up which monomializes
every element of $R$ provided we look far enough in the sequence.
To do this, we will construct a particular sequence of (possibly limit)
immediate successors. We will show that every element $f$ of $R$
will be non-degenerate with respect to a key polynomial $Q$ of this
sequence, in other words, that we will have $\nu_{Q}\left(f\right)=\nu\left(f\right)$.
Furthermore, all the polynomials of this sequence will be monomializable.
At this point we will have proved that every element of $R$ is non-degenerate
with respect to a regular system of parameters of a suitable regular
local ring $R_{i}$. Then we will just have to see that every element
non-degenerate with respect to a regular system of parameters is monomializable
by our sequence of blow-ups.

We will begin this part by some preliminaries, where we define non-degeneracy
and framed and monomial blowing-up.

Then, we will see that every element non-degenerate with respect to
a regular system of paramaters is monomializable. And then it will
be sufficient to prove that it is the case of all the elements of
$R$.

So, after that, we construct sequence of (possibly limit) immediate
successors such that every element $f$ of $R$ is non-degenerate
with respect to one of these key polynomials. 

In sections $6$ and $7$ we prove that all the key polynomials of
this sequence are monomializable, and that we have proven the simultaneous
local uniformization. To do this we will need a new notion: the one
of key element. Indeed, modified by the blow-ups, the key polynomials
of the above mentioned sequence have no reason to still be polynomials.
So we will give a new definition, this one of key element. This notion
has the benefit to be conserved by blow-ups. We will monomialize the
key elements and not the key polynomials, and the proof will be complete
by induction.

\section{Preliminaries.}

Let $k$ be a field of characteristic zero and $R$ a regular local
$k$-algebra which is essentially of finite type over $k$. We consider
$u=(u_{1},\ldots,u_{n})$ a regular system of parameters of $R$ and
$\nu$ a valuation centered on $R$ whose group of values is denoted
by $\Gamma$. We write $\beta_{i}=\nu(u_{i})$ for every integer $i\in\left\{ 1,\dots,n\right\} $,
and $K=k(u_{1},\dots,u_{n-1})$.

\subsection{Non-degenerate elements.}
\begin{defn}
\label{def:nond=0000E9g=0000E9}Let $f\in R$. We say that $f$ is
non-degenerate \index{Non-degenerate elements}with respect to $\nu$
and $u$ if we have $\nu_{u}(f)=\nu(f)$, where $\nu_{u}$ is the
monomial valuation with respect to $u$.
\end{defn}

We need a more convenient way of knowing whether an element is non-degenerate
with respect to a regular system of parameters. It is the objective
of the following Proposition.
\begin{prop}
\label{prop:non d=0000E9gen=0000E9rescence et id=0000E9aux monomiaux}Let
$f\in R$. The element $f$ is non-degenerate with respect to $\nu$
and $u$ if and only if there exists an ideal $N$ of $R$ which contains
$f$, monomial with respect to $u$ and such that
\[
\nu(f)=\nu(N)=\min\limits _{x\in N}\left\{ \nu(x)\right\} .
\]
\end{prop}

\begin{proof}
Let us show that if there exists an ideal $N$ of $R$ which contains
$f$, monomial with respect to $u$ and such that
\[
\nu(f)=\nu(N)=\min\limits _{x\in N}\left\{ \nu(x)\right\} ,
\]
then $\nu_{u}\left(f\right)=\nu\left(f\right)$. Let $N$ be such
an ideal. As $N$ is monomial with respect to $u$, we have $\nu_{u}(N)=\nu(N)$
and $\nu_{u}(N)\leq\nu_{u}(f)$ since $f\in N$.

So $\nu(f)=\nu(N)\leq\nu_{u}(f)$, which give us the equality.

Now let us show that if $\nu_{u}\left(f\right)=\nu\left(f\right)$,
then there exists an ideal $N$ of $R$ which contains $f$, monomial
with respect to $u$ and such that $\nu(f)=\nu(N)=\min\limits _{x\in N}\left\{ \nu(x)\right\} $. 

Let us assume that $\nu_{u}\left(f\right)=\nu\left(f\right)$. Let
$N$ be the smallest ideal of $R$ generated by monomials in $u$
containing $f$. So $\nu(N)=\nu_{u}(N)=\nu_{u}(f)$ and since $\nu_{u}(f)=\nu(f)$,
we have $\nu(N)=\nu(f)$.
\end{proof}

\subsection{Framed and monomial blow-up.}

Let $J_{1}\subset\left\{ 1,\dots,n\right\} $, $A_{1}=\left\{ 1,\dots,n\right\} \setminus J_{1}$
and $j_{1}\in J_{1}$.

We write
\[
u'_{q}=\begin{cases}
\frac{u_{q}}{u_{j_{1}}} & \text{ if }q\in J_{1}\setminus\left\{ j_{1}\right\} \\
u_{q} & \text{otherwise}
\end{cases}
\]
 and we let $R_{1}$ be a localisation of $R'=R\left[u'_{J_{1}\setminus\left\{ j_{1}\right\} }\right]$
by a prime ideal, say $R_{1}=R'_{\mathfrak{m}'}$ of maximal ideal
$\mathfrak{m}_{1}=\mathfrak{\mathfrak{m}^{'}}R_{1}$. Since $R$ is
regular, $R'$ and $R_{1}$ are regular. Let $u^{(1)}=\left(u_{1}^{(1)},\dots,u_{n_{1}}^{(1)}\right)$
be a regular system of parameters of $\mathfrak{m}_{1}$.

We write 
\[
B_{1}:=\left\{ q\in J_{1}\setminus\left\{ j_{1}\right\} \text{ such that }u'_{q}\notin R_{1}^{\times}\right\} 
\]
 and 
\[
C_{1}:=J_{1}\setminus\left(B_{1}\cup\left\{ j_{1}\right\} \right).
\]
Since $u$ is a regular system of parameters of $R$ , we have the
disjoint union 
\[
u'=u'_{A_{1}}\sqcup u'_{B_{1}}\sqcup u'_{C_{1}}\sqcup\left\{ u'_{j_{1}}\right\} .
\]

Let $\pi\colon R\to R_{1}$ be the natural map. Without loss of generality,
we may assume that 
\[
J_{1}=\left\{ 1,\dots,h\right\} .
\]

\begin{defn}
\label{def:=0000E9clatement encadr=0000E9}We say that $\pi\colon\left(R,u\right)\to\left(R_{1},u^{(1)}\right)$
is a framed blow-up \index{Framed blow-up} of $\left(R,u\right)$
along $\left(u_{J_{1}}\right)$ with respect to $\nu$ if there exists
$D_{1}\subset\left\{ 1,\dots,n_{1}\right\} $ such that 
\[
u'_{A_{1}\cup B_{1}\cup\left\{ j_{1}\right\} }=u_{D_{1}}^{(1)}
\]
 and if $\mathfrak{m}'=\left\{ x\in R'\text{ such that }\nu(x)>0\right\} $.
\end{defn}

\begin{rem}
A blow-up $\pi$ is framed if among the given generators of the maximal
ideal $\mathfrak{m}_{1}$ of $R_{1}$, we have all the elements of
$u'$, except, possibly, those that are in $u'_{C_{1}}$. In other
words, except, possibly, those that are invertibles in $R_{1}$.

It is framed with respect to $\nu$ if we localized in the center
of $\nu$.
\end{rem}

Let $\pi$ be such a blow-up.
\begin{defn}
We say that $\pi$ is monomial \index{Monomial framed blow-up}if
$B_{1}=J_{1}\setminus\{j_{1}\}$.
\end{defn}

\begin{rem}
\label{rem: =0000E9clatement monomial nombre de g=0000E9n=0000E9rateurs}Let
$\pi$ be a monomial blow-up.

Then $n_{1}=n$ and $D_{1}=\left\{ 1,\dots,n\right\} $.
\end{rem}

\begin{defn}
Let $\pi\colon\left(R,u\right)\to\left(R_{1},u^{(1)}\right)$ be a
framed blow-up and $T\subset\left\{ 1,\dots,n\right\} $.

We say that $\pi$ is independent \index{Independent framed blow-up}of
$u_{T}$ if $T\cap J_{1}=\emptyset$, in other words if $T\subset A_{1}$.
\end{defn}

\begin{rem}
Since we look at blow-ups with respect to a valuation $\nu$, we have
blow-ups such that $\nu(R_{1})\geq0$. Since $u'_{q}\in R_{1}$ for
every $q\in J_{1}$, we want $\nu\left(\frac{u_{q}}{u_{j_{1}}}\right)\geq0$,
so $\nu\left(u_{q}\right)\geq\nu\left(u_{j_{1}}\right)$ for every
$q\in J_{1}\setminus\left\{ j_{1}\right\} $ . So we can set $j_{1}$
to be an element of $J_{1}$ such that $\beta_{j_{1}}=\min\limits _{q\in J_{1}}\{\beta_{q}\}$.

We have :

\[
\begin{array}{ccc}
B_{1} & := & \left\{ q\in J_{1}\setminus\left\{ j_{1}\right\} \text{ such that }u'_{q}\notin R_{1}^{\times}\right\} \\
 & = & \left\{ q\in J_{1}\setminus\left\{ j_{1}\right\} \text{ such that \ensuremath{\nu}}\left(u'_{q}=\frac{u_{q}}{u_{j_{1}}}\right)>0\right\} \\
 & = & \left\{ q\in J_{1}\setminus\left\{ j_{1}\right\} \text{ such that }\beta_{q}>\beta_{j_{1}}\right\} .
\end{array}
\]
And $C_{1}=\left\{ q\in J_{1}\setminus\left\{ j_{1}\right\} \text{ such that }\beta_{q}=\beta_{j_{1}}\right\} $.
\end{rem}

Let $k_{1}$ be the residue field of $R_{1}$ and $t_{k_{1}}$ the
transcendence degree of $k\hookrightarrow k_{1}$. Let us show that
$t_{k_{1}}\leq\sharp C$.

We write $\bar{R}=\frac{R'}{\mathfrak{m}R'}$. We denote by $\bar{u}_{q}$
the image of $u'_{q}$ in $\bar{R}$ for every $q\in J_{1}\setminus\left\{ j_{1}\right\} $.
So $\bar{R}=k\left[\bar{u}_{B_{1}},\bar{u}_{C_{1}}^{\pm1}\right]$.
We have $R\to R'\to R_{1}\to k_{1}$, which induces homomorphisms
$k\to\bar{R}\to\frac{R_{1}}{\mathfrak{m}R_{1}}\to k_{1}$.

We have $\mathfrak{m}=\mathfrak{m}_{1}\cap R=\mathfrak{m}'R_{1}\cap R=\mathfrak{m}'\cap R$
. Let $\bar{\mathfrak{m}}=\frac{\mathfrak{m}'}{\mathfrak{m}R'}$.
We have

\[
\begin{array}{ccc}
\frac{R_{1}}{\mathfrak{m}R_{1}} & = & \frac{R'_{\mathfrak{m}'}}{\mathfrak{m}R_{\mathfrak{m}'}^{'}}\\
 & = & \left(\frac{R'}{\mathfrak{m}R'}\right)_{\frac{\mathfrak{m}'}{\mathfrak{m}R'}}\\
 & = & \bar{R}_{\bar{\mathfrak{m}}}
\end{array}
\]
 in other words
\begin{equation}
k\to\bar{R}\to\bar{R}_{\bar{\mathfrak{m}}}\to k_{1}.\label{eq:suitemodulom}
\end{equation}

Since $u'_{A_{1}\cup B_{1}\cup\left\{ j_{1}\right\} }\subset\mathfrak{m}'$,
for every $q\in A_{1}\cup B_{1}\cup\left\{ j_{1}\right\} $, the image
of $u'_{q}$ in $k_{1}$ is zero. So $k_{1}$ is generated over $k$
by the images of the $u'_{q}$ with $q\in C_{1}$. Hence $t_{k_{1}}\leq\sharp C_{1}$.

But we have $C_{1}:=J_{1}\setminus\left(B_{1}\cup\left\{ j_{1}\right\} \right)$.
So $\sharp C_{1}+\sharp B_{1}+1=\sharp J_{1}=h$, and:

\begin{equation}
\sharp B_{1}+1\leq t_{k_{1}}+\sharp B_{1}+1\leq\sharp C_{1}+\sharp B_{1}+1=h\leq n.\label{eq:troiscas}
\end{equation}

We will often set $J_{1}\subset\left\{ 1,\dots,r,n\right\} $ where
$r$ is the dimension of $\sum\limits _{i=1}^{n}\mathbb{Q}\nu(u_{i})$
in $\Gamma\otimes_{\mathbb{Z}}\mathbb{Q}$. If $J_{1}\subset\left\{ 1,\dots,r\right\} $,
the family $\beta_{J_{1}}$ is a family of $\mathbb{Q}$-linearly
independent elements, and so $B_{1}=J_{1}\setminus\left\{ j_{1}\right\} $. 

Otherwise $n\in J_{1}$. Then we have $B_{1}=J_{1}\setminus\left\{ j_{1}\right\} $
or $B_{1}=J_{1}\setminus\left\{ j_{1},q_{1}\right\} $ where $q_{1}\in J_{1}\setminus\left\{ j_{1}\right\} $.
The interesting cases are those where $h-2\leq\sharp B_{1}$, in other
words, those where $h-1\leq\sharp B_{1}+1$.

Since (\ref{eq:troiscas}), we have $h-1+t_{k_{1}}\leq\sharp B_{1}+1+t_{k_{1}}\leq h$.

Then we have three cases\label{troiscas}.

The first one, $\sharp B_{1}+1=h$ and $t_{k_{1}}=0$, it occurs when
the blow-up is monomial.

The second one, $\sharp B_{1}+1=h-1$ and $t_{k_{1}}=1$.

The last one, $\sharp B_{1}+1=h-1$ and $t_{k_{1}}=0$.
\begin{fact}
\label{fact:nombredegene}In the cases $1$ and $3$, we have $n_{1}=n$
and in the case $2$ we have $n_{1}=n-1$.
\end{fact}

\begin{rem}
\label{rem: le poly min est de degre 1}In the rest of the chapter,
we will assume that the valuation ring has $k$ as residue field.
So $k_{1}=k$ and $t_{k_{1}}=0$. Hence we will have $n_{1}=n$.

Since $k_{1}\simeq\frac{k\left[Z\right]}{\left(\lambda\left(Z\right)\right)}$,
we know that $\lambda\left(Z\right)$ is a polynomial of degree $1$
over $k$.
\end{rem}

\subsection{Key elements.}

We need a more general notion than the one of key polynomials. Indeed,
after several blow-ups, a key polynomial might not be a polynomial
anymore.

For example, we can have $\frac{1}{u_{n}+1}u_{n-1}$, which is not
a polynomial.
\begin{defn}
\label{def:elmtclef}Let $P_{1},P_{2}$ be two key polynomials for
the field extension $k\left(u_{1}^{(l)},\dots,u_{n-1}^{(l)}\right)\left(u_{n}^{(l)}\right)$
with $P_{2}$ and immediate successor of $P_{1}$. Let $P_{2}=\sum\limits _{j\in S_{P_{1}}\left(P_{2}\right)}a_{j}P_{1}^{j}$
be the $P_{1}$-expansion of $P_{2}$.

We call \emph{key element} \index{Key element}every element $P_{2}'$
of the form 
\[
P_{2}'=\sum\limits _{j\in S_{P_{1}}\left(P_{2}\right)}a_{j}b_{j}P_{1}^{j}
\]
 where $b_{j}$ are units of $R_{l}=k\left(u_{1}^{(l)},\dots,u_{n}^{(l)}\right)_{\left(u_{1}^{(l)},\dots,u_{n}^{(l)}\right)}$.
The polynomial $P_{2}$ is the key polynomial associated to the key
element $P_{2}'$.
\end{defn}

\begin{rem}
A key element is not necessarily a polynomial. Indeed, for example,
$\frac{1}{1+u_{n_{l}}^{(l)}}$ is a unit of $R_{l}$.
\end{rem}

\begin{defn}
\label{def:elmtsuccimm}Let $P_{1}'$ and $P_{2}'$ be two key elements.
We say that $P'_{2}$ is an \emph{immediate successor} of $P'_{1}$,
and we write \index{Immediate successors key elements}, $P_{1}'\ll P_{2}'$,
if their associated key polynomials are immediate successors of each
other.
\end{defn}

Now we define limit immediate successors key elements.
\begin{defn}
\label{def:limitelmt}Let $P_{1}'$ and $P_{2}'$ be two key elements.
We say that $P_{2}'$ is a \emph{limit immediate successor }\index{Limit immediate successor key elements}
of $P'_{1}$, and we write $P_{1}'\ll_{\lim}P_{2}'$, if their associated
key polynomials $P_{1}$ and $P_{2}$ are such that $P_{2}$ is a
limit immediate successor of $P_{1}$.
\end{defn}

\newpage{}

\section{Monomialization in the non-degenerate case.}

In this section, we will monomialize all the elements which are non-degenerate
with respect to a system of parameters.

Let $\alpha$ and $\gamma$ be two elements of $\mathbb{Z}^{n}$,
and let $\delta=\left(\min\left\{ \alpha_{j},\gamma_{j}\right\} \right)_{1\leq j\leq n}$.
We say that $u^{\alpha}\mid u^{\gamma}$ if for every integer $i$,
$\alpha_{i}$ is less than or equal to $\gamma_{i}$ , in other words
if $\alpha$ is componentwise less than or equal to $\beta$.

Let us set
\[
\widetilde{\alpha}=\alpha-\delta=\left(\widetilde{\alpha}_{1},\dots,\widetilde{\alpha}_{a},0,\dots,0\right)\in\mathbb{N}^{n}.
\]
The objective is to build a sequence of blow-ups $(R,u)\to\cdots\to\left(R',u'\right)$
such that in $R'$, we have $u^{\alpha}\mid u^{\gamma}$. 
\begin{defn}
We say that $\alpha\preceq\gamma$ if for every index $i$, we have
$\alpha_{i}\leq\gamma_{i}$.
\end{defn}

We assume that $\gamma\npreceq\alpha$ and that $\alpha\npreceq\gamma$
. So we may assume that $|\widetilde{\alpha}|\neq0$, and $\widetilde{\alpha}_{i}>0$
for every integer $i\in\left\{ 1,\dots,a\right\} $. 

Similarly, we set 
\[
\widetilde{\gamma}=\gamma-\delta=\left(0,\dots,0,\widetilde{\gamma}_{a+1},\dots,\widetilde{\gamma}_{n}\right)\in\mathbb{N}^{n}.
\]
 Interchanging $\alpha$ and $\gamma$, if necessary, we may assume
that $0<|\widetilde{\alpha}|\leq|\widetilde{\gamma}|$.

\subsection{Construction of a stricly decreasing numerical character.}
\begin{defn}
Let $\tau\colon\mathbb{Z}^{n}\times\mathbb{Z}^{n}\to\mathbb{N}^{2}$
be the map such that 
\[
\tau(\alpha,\gamma)=(|\widetilde{\alpha}|,|\widetilde{\gamma}|).
\]
\end{defn}

Let $J$ be a minimal subset of $\left\{ 1,\dots n\right\} $ such
that $\{1,\dots,a\}\subset J$ and $\sum\limits _{q\in J}\widetilde{\gamma}_{q}\geq|\widetilde{\alpha}|$. 

Let $\pi\colon\left(R,u\right)\to\left(R_{1},u^{(1)}\right)$ be a
framed blow-up along $(u_{J})$. Let $j\in J$ be such that $R_{1}$
is a localization of $R\left[\frac{u_{J}}{u_{j}}\right]$.

If $q\in J\setminus\{j\}$, we recall that $u'_{q}=\frac{u_{q}}{u_{j}}$,
and $u'_{q}=u_{q}$ otherwise.

We now define $\widetilde{\alpha}'_{q}=\widetilde{\alpha}_{q}$ for
$q\neq j$, and $\widetilde{\alpha}'_{q}=0$ otherwise. We set $\widetilde{\gamma}'_{q}=\widetilde{\gamma}_{q}$
if $q\neq j$, $\widetilde{\gamma}'_{q}=\sum\limits _{q\in J}\widetilde{\gamma}_{q}-|\widetilde{\alpha}|$
otherwise.

And finally we define
\[
\delta'=(\delta_{1},\dots,\delta_{j-1},\sum\limits _{q\in J}\delta_{q}+|\widetilde{\alpha}|,\delta_{j+1},\dots,\delta_{n}).
\]
So we have:

\[
\begin{array}{ccc}
u^{\alpha} & = & \prod\limits _{l=1}^{n}u_{l}^{\alpha_{l}}\\
 & = & \prod\limits _{\begin{array}{c}
l=1\\
l\in J\setminus\left\{ j\right\} 
\end{array}}^{n}u_{l}^{\alpha_{l}}\times\prod\limits _{\begin{array}{c}
l=1\\
l\notin J\setminus\left\{ j\right\} 
\end{array}}^{n}u_{l}^{\alpha_{l}}.
\end{array}
\]

But for every $l\in J\setminus\left\{ j\right\} $, we have $u_{l}=u_{l}'\times u_{j}$
and for $l\notin J\setminus\left\{ j\right\} $, we have $u_{l}=u_{l}'$.
Hence 
\[
u^{\alpha}=\prod\limits _{\begin{array}{c}
l=1\\
l\in J\setminus\left\{ j\right\} 
\end{array}}^{n}\left(u_{l}'\times u_{j}\right)^{^{\alpha_{l}}}\times\prod\limits _{\begin{array}{c}
l=1\\
l\notin J\setminus\left\{ j\right\} 
\end{array}}^{n}\left(u_{l}'\right)^{^{\alpha_{l}}}.
\]

Let us isolate the term $u_{j}$. We obtain:

\[
u^{\alpha}=u_{j}^{\sum\limits _{l\in J\setminus\left\{ j\right\} }\alpha_{l}}\times\prod\limits _{\begin{array}{c}
l=1\end{array}}^{n}\left(u'_{l}\right)^{^{\alpha_{l}}}
\]
 and since $\tilde{\alpha}=\alpha-\delta$, we have $\alpha=\tilde{\alpha}+\delta$
and then

\[
\begin{array}{ccc}
u^{\alpha} & = & u_{j}^{\sum\limits _{l\in J\setminus\left\{ j\right\} }\alpha_{l}}\times\prod\limits _{\begin{array}{c}
l=1\end{array}}^{n}\left(u'_{l}\right)^{^{\widetilde{\alpha}_{l}+\delta_{l}}}\\
 & = & u_{j}^{\sum\limits _{l\in J\setminus\left\{ j\right\} }\alpha_{l}}\times\prod\limits _{\begin{array}{c}
l=1\\
l\neq j
\end{array}}^{n}\left(u'_{l}\right)^{^{\widetilde{\alpha}_{l}+\delta_{l}}}\times\left(u'_{j}\right)^{^{\widetilde{\alpha}_{j}+\delta_{j}}}.
\end{array}
\]
But $\widetilde{\alpha}'_{q}=\widetilde{\alpha}_{q}$ for $q\neq j$
and $\delta'=(\delta_{1},\dots,\delta_{j-1},\sum\limits _{q\in J}\delta_{q}+|\widetilde{\alpha}|,\delta_{j+1},\dots,\delta_{n})$,
so

\[
\begin{array}{ccc}
u^{\alpha} & = & u_{j}^{\sum\limits _{l\in J\setminus\left\{ j\right\} }\alpha_{l}}\times\prod\limits _{\begin{array}{c}
l=1\\
l\neq j
\end{array}}^{n}\left(u'_{l}\right)^{^{\widetilde{\alpha}'_{l}+\delta'_{l}}}\times\left(u'_{j}\right)^{\widetilde{\alpha}_{j}+\delta_{j}}\\
 & = & u_{j}^{\sum\limits _{l\in J\setminus\left\{ j\right\} }\alpha_{l}+\widetilde{\alpha}_{j}+\delta_{j}}\times\prod\limits _{\begin{array}{c}
l=1\\
l\neq j
\end{array}}^{n}\left(u'_{l}\right)^{^{\widetilde{\alpha}'_{l}+\delta'_{l}}}.
\end{array}
\]

We include another time the term $l=j$ in the product, and then:

\[
\begin{array}{ccc}
u^{\alpha} & = & u_{j}^{\sum\limits _{l\in J\setminus\left\{ j\right\} }\alpha_{l}+\widetilde{\alpha}_{j}+\delta_{j}-\widetilde{\alpha}'_{j}-\delta'_{j}}\times\prod\limits _{\begin{array}{c}
l=1\end{array}}^{n}\left(u'_{l}\right)^{^{\widetilde{\alpha}'_{l}+\delta'_{l}}}\\
 & = & u_{j}^{\sum\limits _{l\in J\setminus\left\{ j\right\} }\alpha_{l}+\widetilde{\alpha}_{j}+\delta_{j}-\widetilde{\alpha}'_{j}-\delta'_{j}}\times\left(u'\right)^{^{\widetilde{\alpha}'+\delta'}}.
\end{array}
\]

But we have

\[
\begin{array}{ccc}
\sum\limits _{l\in J\setminus\left\{ j\right\} }\alpha_{l}+\widetilde{\alpha}_{j}+\delta_{j}-\widetilde{\alpha}'_{j}-\delta'_{j} & = & \sum\limits _{l\in J\setminus\left\{ j\right\} }\alpha_{l}+\widetilde{\alpha}_{j}+\delta_{j}-\delta'_{j}\\
 & = & \sum\limits _{l\in J\setminus\left\{ j\right\} }\alpha_{l}+\widetilde{\alpha}_{j}+\delta_{j}-\sum\limits _{q\in J}\delta_{q}-|\widetilde{\alpha}|\\
 & = & \sum\limits _{l\in J\setminus\left\{ j\right\} }\left(\widetilde{\alpha}_{l}+\delta_{l}\right)+\widetilde{\alpha}_{j}-\sum\limits _{q\in J\setminus\left\{ j\right\} }\delta_{q}-|\widetilde{\alpha}|\\
 & = & \sum\limits _{\begin{array}{c}
l\in J\end{array}}\widetilde{\alpha}_{l}-|\widetilde{\alpha}|\\
 & = & 0.
\end{array}
\]

So $u^{\alpha}=\left(u'\right)^{^{\widetilde{\alpha}'+\delta'}}$
, and similarly $u^{\gamma}=\left(u'\right)^{^{\widetilde{\gamma}'+\delta'}}$
.

We set $\alpha'=\delta'+\widetilde{\alpha}'$ and $\gamma'=\delta'+\widetilde{\gamma}'$.
\begin{prop}
\label{prop:taudecroit}We have $\tau(\alpha',\gamma')<\tau(\alpha,\gamma)$.
\end{prop}

\begin{proof}
First case: $j\in\left\{ 1,\dots,a\right\} $. Then 
\[
|\widetilde{\alpha}'|=|\widetilde{\alpha}|-\widetilde{\alpha}_{j}<|\widetilde{\alpha}|.
\]

Second case: $j\in\left\{ a+1,\dots,n\right\} $. Then $|\widetilde{\alpha}'|=|\widetilde{\alpha}|$.
Let us show that $|\widetilde{\gamma}'|<|\widetilde{\gamma}|$.

We have

\[
\begin{array}{ccc}
|\widetilde{\gamma}'| & = & \sum\limits _{\begin{array}{c}
q=a+1\\
q\neq j
\end{array}}^{n}\widetilde{\gamma}_{q}+\sum\limits _{q\in J}\widetilde{\gamma}_{q}-|\widetilde{\alpha}|\\
 & = & \sum\limits _{\begin{array}{c}
q=a+1\end{array}}^{n}\widetilde{\gamma}_{q}+\sum\limits _{q\in J\setminus\left\{ j\right\} }\widetilde{\gamma}_{q}-|\widetilde{\alpha}|.
\end{array}
\]

By the minimality of $J$, we have $\sum\limits _{q\in J\setminus\left\{ j\right\} }\widetilde{\gamma}_{q}-|\widetilde{\alpha}|<0$,
and so 
\[
|\widetilde{\gamma}'|<\sum\limits _{\begin{array}{c}
q=a+1\end{array}}^{n}\widetilde{\gamma}_{q}=|\widetilde{\gamma}|.
\]

In every case, we have $\left(|\widetilde{\alpha}'|,|\widetilde{\gamma}'|\right)<\left(|\widetilde{\alpha}|,|\widetilde{\gamma}|\right)=\tau(\alpha,\gamma)$.

If $|\widetilde{\alpha}'|\leq|\widetilde{\gamma}'|$, then $\tau(\alpha',\gamma')=\left(|\widetilde{\alpha}'|,|\widetilde{\gamma}'|\right)$
and this completes the proof.

Otherwise, $|\widetilde{\alpha}'|>|\widetilde{\gamma}'|$, so 
\[
\tau(\alpha',\gamma')=\left(|\widetilde{\gamma}'|,|\widetilde{\alpha}'|\right)<\left(|\widetilde{\alpha}'|,|\widetilde{\gamma}'|\right),
\]
 and the proof is complete.
\end{proof}
Renumbering the $u'_{q}$, if necessary, we may assume that$u'_{q}\notin R_{1}^{\times}$
for every $q\in\left\{ 1,\dots,s\right\} $ and $u'_{q}\in R_{1}^{\times}$
otherwise. Since $\pi$ is a framed blow-up, we have $\left\{ u'_{1},\dots,u'_{s}\right\} \subset u^{(1)}$,
so renumbering again, if necessary, we may assume that $u'_{q}=u_{q}^{(1)}$
for every $q\in\left\{ 1,\dots,s\right\} $. We set 
\[
\alpha^{(1)}=\left(\alpha'_{1},\dots,\alpha'_{s},0,\dots0\right)\in\mathbb{Z}^{n_{1}}
\]
 and 
\[
\gamma^{(1)}=\left(\gamma'_{1},\dots,\gamma'_{s},0,\dots0\right)\in\mathbb{Z}^{n_{1}}.
\]
 We have $\tau\left(\alpha^{(1)},\gamma^{(1)}\right)\leq\tau(\alpha',\gamma')$.
By Proposition \ref{prop:taudecroit}, we have
\[
\tau\left(\alpha^{(1)},\gamma^{(1)}\right)<\tau(\alpha,\gamma).
\]

\subsection{Divisibility and change of variables.}

Let $s\in\left\{ 1,\dots,n\right\} $. We write $u=\left(w,v\right)$
where 
\[
w=\left(w_{1},\dots,w_{s}\right)=\left(u_{1},\dots,u_{s}\right)
\]
 and 
\[
v=\left(v_{1},\dots,v_{n-s}\right).
\]
Let $\alpha$ and $\gamma$ be two elements of $\mathbb{Z}^{s}$.
\begin{prop}
\label{prop:existencesuitedivise}There exists a framed local sequence
\[
\left(R,u\right)\to\left(R_{l},u^{(l)}\right),
\]
 with respect to $\nu$, independent of $v$, such that in $R_{l}$,
we have $w^{\alpha}\mid w^{\gamma}$ or $w^{\gamma}\mid w^{\alpha}$.
\end{prop}

\begin{proof}
Unless $\gamma\preceq\alpha$, or $\alpha\preceq\gamma$, we can iterate
the above construction, choosing blow-up with respect to $\nu$ and
independent of $v$. Since $\tau$ is a vector in $\mathbb{N}^{2}$
and is stricly decreasing, after a finite number of steps, the process
stops. After these steps, we have $w^{\alpha}=U\times\left(u^{(l)}\right)^{\alpha^{(l)}}$,
$w^{\gamma}=U\times\left(u^{(l)}\right)^{\gamma^{(l)}}$, with $U\in R_{l}^{\times}$
and with $\gamma^{(l)}\preceq\alpha^{(l)}$, or $\alpha^{(l)}\preceq\gamma^{(l)}$.
So we do have $w^{\alpha}\mid w^{\gamma}$ or $w^{\gamma}\mid w^{\alpha}$
in $R_{l}$.
\end{proof}
Let us now study the change of variables we do at each blow-up. We
consider $i$ and $i'$ some indexes of the framed local sequence
\begin{equation}
\left(R,u\right)\to\dots\to\left(R_{i},u^{(i)}\right)\to\dots\to\left(R_{i'},u^{(i')}\right)\to\dots\to\left(R_{l},u^{(l)}\right).\label{eqjepense que celle l=0000E0 c'est mieux}
\end{equation}

\begin{prop}
\label{prop:chgmtvari}Let us consider $0\leq i<i'\leq l$. We let
$m$ be an element of $\{1,\dots,n_{i}\}$ and $m'$ one of $\{1,\dots,n_{i'}\}$.
Then:
\end{prop}

\begin{enumerate}
\item \emph{There exists a vector $\delta_{m}^{(i',i)}\in\mathbb{N}^{\sharp D_{i}}$
such that 
\[
u_{m}^{(i)}\in\left(u_{D_{i'}}^{(i')}\right)^{\delta_{m}^{(i',i)}}R_{i'}^{\times}.
\]
}
\item \emph{If, in addition, the local sequence (\ref{eqjepense que celle l=0000E0 c'est mieux})
is independent of $u_{T}$, with $T\subset\{1,\dots,n\}$ ; and if
we assume that $u_{m}^{(i)}\notin u_{T}$, then $\left(u_{D_{i'}}^{(i')}\right)^{\delta_{m}^{(i',i)}}$
is monomial in $u_{D_{i'}}^{(i')}\setminus u_{T}$.}
\item \emph{We assume that $i''>0$ such that $i\leq i''<i'$. We have $D_{i"}=\{1,\dots,n_{i''}\}$,
and we assume that $m'\in D_{i'}$. Then there exists a vector $\gamma_{m'}^{(i,i')}$
of $\mathbb{Z}^{n_{i}}$ such that
\[
u_{m'}^{(i')}=\left(u^{(i)}\right)^{\gamma_{m'}^{(i,i')}}.
\]
}
\item \emph{If, in addition, the local sequence (\ref{eqjepense que celle l=0000E0 c'est mieux})
is independent of $u_{T}$ and if we assume that $u_{m'}^{(i')}\notin u_{T}$,
then $u_{m'}^{(i')}$ is monomial in $u^{(i)}\setminus u_{T}$.}
\end{enumerate}
\begin{proof}
We only consider the case $i'=i+1$, the general case can be proved
by induction on $i-i'$. We can also assume that $i=0$.

Let us show $(1)$. By Definition \ref{def:=0000E9clatement encadr=0000E9},
we have $u'_{A_{1}\cup B_{1}\cup\left\{ j_{1}\right\} }=u_{D_{1}}^{(1)}$.

We denote by $D_{1}=D_{1}^{A_{1}}\cup D_{1}^{B_{1}}$ where 
\[
u'_{A_{1}}=u_{D_{1}^{A_{1}}}^{(1)}
\]
 and 
\[
u'_{B_{1}\cup\{j_{1}\}}=u_{D_{1}^{B_{1}}}^{(1)}.
\]

If $m\in A_{1}\cup\{j_{1}\}$, so $u_{m}=u'_{m}$ and the proof is
finished. If $m\in B_{1}$ then $u_{m}=u_{j_{1}}u'_{m}=u'_{j_{1}}u'_{m}$
and the proof is finished.

If $m\in C_{1}$, so $u_{m}=u'_{j_{1}}u'_{m}$ and by definition,
$u'_{m}\in R_{1}^{\times}$, which gives us the result.

Let us show $(3)$. We have $m'\in D_{1}=D_{1}^{A_{1}}\cup D_{1}^{B_{1}}$
and $u'_{A_{1}\cup B_{1}\cup\left\{ j_{1}\right\} }=u_{D_{1}}^{(1)}$.
If $m'\in D_{1}^{A_{1}}$ then by definition $u_{m'}^{(1)}\in u'_{A_{1}}=u_{A_{1}}$
and we have the result. Otherwise $m'\in D_{1}^{B_{1}}$. So 
\[
u_{m'}^{(1)}\in u'_{B_{1}\cup\{j_{1}\}}=\left\{ u_{j_{1}},\frac{u_{q}}{u_{j_{1}}}\text{ }q\in B_{1}\right\} .
\]
 This completes the proof of $(3)$.

Now let us assume that the sequence is independent of $u_{T}$. By
definition we have $u_{J_{1}}\cap u_{T}=\emptyset$ and also 
\[
u_{D_{1}^{B_{1}}}^{(1)}\cap u_{T}=\emptyset.
\]

Let us show $(2)$. Assume that $u_{m}\notin u_{T}$.

If $m\in A_{1}$, then $u_{m}=u'_{m}\in u_{D_{1}^{A_{1}}}^{(1)}$
and $u_{m}\notin u_{T}$ and the proof is finished. Otherwise $m\in J_{1}$.
We saw in the proof of $(1)$ that $m$ was monomial in $u_{D_{1}^{B_{1}}}^{(1)}$,
and since $u_{D_{1}^{B_{1}}}^{(1)}\cap u_{T}=\emptyset$, this completes
the proof of $(2)$.

It remains to prove $(4)$. We assume that $u_{m'}^{(1)}\notin u_{T}$,
with $m'\in D_{1}=D_{1}^{A_{1}}\cup D_{1}^{B_{1}}$.

If $m'\in D_{1}^{A_{1}}$, then $u_{m'}^{(1)}\in u'_{A_{1}}=u_{A_{1}}$.
Since $u_{m'}^{(1)}\notin u_{T}$, we have $u_{m'}^{(1)}\in u\setminus u_{T}$.

Otherwise $m'\in D_{1}^{B_{1}}$ and we saw that $u_{m'}^{(1)}$ is
monomial in $u_{B_{1}\cup[j_{1}\}}\subset u_{J}$. Since $u_{J}\cap u_{T}=\emptyset$,
we are done.
\end{proof}
\begin{rem}
\label{rem:wenfonctiondew'}Let $T\subset A$, be a set of cardinality
$t$, and $s:=n-t$. We set 
\[
v=(v_{1},\dots,v_{t})=u_{T}
\]
 and
\[
w=(w_{1},\dots,w_{s})=u_{\{1,\dots,n\}\setminus T}.
\]

In this Remark, we only consider monomial blow-ups.

We have $u'=(v,w')$ where $w'=(w'_{1},\dots,w'_{s})=(w^{\gamma(1)},\dots,w^{\gamma(s)})$
with $\gamma(i)\in\mathbb{Z}^{s}$, by Proposition \ref{prop:chgmtvari}.
By the proof of this Proposition, the matrix $F_{s}=\left[\gamma(1)\dots\gamma(s)\right]$
is a unimodular matrix. For every $\delta\in\mathbb{Z}^{s}$, we have
$w'^{\delta}=w^{\delta F_{s}}$. In the same vein $w_{i}=w'^{\delta(i)}$
and the $s$-vectors $\delta(1),\dots,\delta(s)$ form a unimodular
matrix equal to the inverse of $F_{s}$. Then we have $w'^{\gamma}=w^{\gamma F_{s}^{-1}}$,
for every $\gamma\in\mathbb{Z}^{s}$.
\end{rem}

\begin{prop}
\label{prop:divaleur} We have: 

\[
w^{\alpha}|w^{\gamma}\text{ in }R_{l}\Leftrightarrow\nu(w^{\alpha})\leq\nu(w^{\gamma}).
\]
\end{prop}

\begin{proof}
We have $u^{(l)}=\left(w_{1}^{(l)},\dots,w_{r_{l}}^{(l)},v\right)$.

By Proposition \ref{prop:chgmtvari}, there exists $\alpha^{(l)},\gamma^{(l)}\in\mathbb{N}^{r_{l}}$
and $y,z\in R_{l}^{\times}$ such that $w^{\alpha}=y\left(w^{(l)}\right)^{\alpha^{(l)}}$
and $w^{\gamma}=z\left(w^{(l)}\right)^{\gamma^{(l)}}$.

For every $i\in\left\{ 1,\dots,r_{l}\right\} $, we have $\nu(w_{i}^{(l)})\geq0$
since the blow-up is with respect to $\nu$, so centered in $R_{l}$.
By construction of $R_{l}$, we have that $\gamma^{(l)}\preceq\alpha^{(l)}$
or $\alpha^{(l)}\preceq\gamma^{(l)}$.

So 
\[
\left(w^{(l)}\right)^{\alpha^{(l)}}\mid\left(w^{(l)}\right)^{\gamma^{(l)}}\Leftrightarrow\nu\left(\left(w^{(l)}\right)^{\alpha^{(l)}}\right)\leq\nu\left(\left(w^{(l)}\right)^{\gamma^{(l)}}\right),
\]
hence 
\[
w^{\alpha}\mid w^{\gamma}\Leftrightarrow\nu(w^{\alpha})\leq\nu(w^{\gamma}).
\]
\end{proof}

\subsection{Monomialization of non-degenerate elements.}

Let $N$ be an ideal of $R$ generated by monomials in $w$. We choose
$w^{\epsilon_{0}},\dots,w^{\epsilon_{b}}$ to be a minimal set of
generators of $N$, with $\nu(w^{\epsilon_{0}})\leq\nu(w^{\epsilon_{i}})$
for every $i$.
\begin{prop}
\label{prop:idealprincipal}There exists a local framed sequence 
\[
\phi\colon\left(R,u\right)\to\left(R_{l},u^{(l)}\right)
\]
 with respect to $\nu$, independent of $v$ and such that $NR_{l}=(w^{\epsilon_{0}})R_{l}$.
\end{prop}

\begin{proof}
Let 
\[
\tau(N,w):=\begin{cases}
\left(b,\min\limits _{0\leq i<j\leq b}\tau\left(w^{\epsilon_{i}},w^{\epsilon_{j}}\right)\right) & \text{if }b\neq0\\
(0,1) & \text{otherwise}.
\end{cases}
\]

Assume $b\neq0$.

We let $\left(w^{\epsilon_{i_{0}}},w^{\epsilon_{j_{0}}}\right)$ be
a pair for which the minimum 
\[
\min\limits _{0\leq i<j\leq b}\tau\left(w^{\epsilon_{i}},w^{\epsilon_{j}}\right)
\]
is attained. By Proposition \ref{prop:taudecroit}, $\tau(N,w)$ is
strictly decreasing at each blow-up.

Since the process stops, $NR_{l}$ is generated by a unique element
as an ideal of $R_{l}$. By Proposition \ref{prop:divaleur}, this
element is $w^{\epsilon_{0}}$ (which has the minimal value), which
divides the others. Then $NR_{l}=(w^{\epsilon_{0}})R_{l}$.
\end{proof}
\begin{defn}
\label{def:monom}An element $f$ of $R$ is monomializable if there
exists a sequence of blow-ups
\[
\left(R,u\right)\to\left(R',u'\right)
\]
 such that the total transformed of $f$ is a monomial. It means that
in $R'$, the total transform of $f$ is $v\prod\limits _{i=1}^{n}\left(u'_{i}\right)^{\alpha_{i}}$,
with $v$ a unit of $R'$.
\end{defn}

\begin{thm}
\label{thm:on monomialise les elements non degeneres}Let $f$ be
a non-degenerate element with respect to $u=(w,v)$, and let $N$
be the ideal which satisfies the conclusion of the Proposition \ref{prop:non d=0000E9gen=0000E9rescence et id=0000E9aux monomiaux},
generated by monomials in $w$.

Then there exists a local framed sequence, independent of $v$, 
\[
\left(R,u\right)\to\left(R',u'\right)
\]
 such that $f$ is a monomial in $u'$ multiplied by a unit of $R'$.
Equivalently, $f$ is monomializable.
\end{thm}

\begin{proof}
Let $\left(R,u\right)\to\left(R',u'\right)$ be the local framed sequence
of the Proposition \ref{prop:idealprincipal}. We have $NR'=w^{\epsilon_{o}}R'$.
Since $f\in N$, by the proof of the Proposition \ref{prop:non d=0000E9gen=0000E9rescence et id=0000E9aux monomiaux},
there exists an element $z\in R'$ such that $f=w^{\epsilon_{0}}z$.
Since $\nu$ is centered in $R'$, to show that $z$ is a unit of
$R'$, we will show that $\nu(z)=0$.

But $\nu(z)=\nu(f)-\nu(w^{\epsilon_{0}})=\nu(N)-\nu(w^{\epsilon_{0}})$
by Proposition \ref{prop:non d=0000E9gen=0000E9rescence et id=0000E9aux monomiaux}.

Since $NR'=w^{\epsilon_{o}}R'$, we have $\nu(N)=\nu(w^{\epsilon_{0}})$,
and so $\nu(z)=0$, and this completes the proof.
\end{proof}
\newpage{}

\section{\label{sectionsuitecofinale}Non-degeneracy and key polynomials.}

Now that we monomialized every non-degenerate element with respect
to the generators of the maximal ideal of our local ring, we are going
to show that every element is non-degenerate with respect to a particular
sequence of immediate successors. We denote by $\Lambda$ the set
of key polynomials and 
\[
M_{\alpha}:=\left\{ Q\in\Lambda\text{ such that \ensuremath{\deg}(Q)=\ensuremath{\alpha}}\right\} .
\]

\begin{prop}
\label{prop: suite infinie bornee}We consider $\nu$ an archimedean
valuation centered in a noetherian local domain $\left(R,\mathfrak{m},k\right)$.
We denote by $\Gamma$ the value group of $\nu$ and we set $\Phi:=\nu\left(R\setminus\left(0\right)\right)$.

The set $\Phi$ does not contain an infinite bounded strictly increasing
sequence.
\end{prop}

\begin{proof}
Assume, aiming for contradiction, that we have an infinite sequence
\[
\alpha_{1}<\alpha_{2}<\dots
\]
 of elements of $\Phi$ bounded by an element $\beta\in\Phi$.

Then we have an infinite decreasing sequence $\cdots\subseteq P_{\alpha_{2}}\subseteq P_{\alpha_{1}}$
such that for every index $i$, we have $P_{\beta}\subseteq P_{\alpha_{i}}$.
And so we have an infinite decreasing sequence of ideals of $\frac{R}{P_{\beta}}$.

We set
\[
\delta=\nu\left(\mathfrak{m}\right)=\min\limits _{x\in\Phi\setminus\left\{ 0\right\} }\left\{ \nu\left(x\right)\right\} .
\]
Since $\nu$ is archimedean, we know that there exists a non-zero
integer $n$ such that $\beta\leq n\delta$, and so such that $\mathfrak{m}^{n}\subseteq P_{\beta}$.
This way, we construct an epimorphism of rings $\frac{R}{\mathfrak{m}^{n}}\twoheadrightarrow\frac{R}{P_{\beta}}$.
Since the ring $R$ is noetherian, $\frac{R}{\mathfrak{m}^{n}}$ is
artinian, and so is $\frac{R}{P_{\beta}}$. This contradicts the existence
of the infinite decreasing sequence of ideals of $\frac{R}{P_{\beta}}$.
\end{proof}
\begin{defn}
Assume that the set $M_{\alpha}$ is non-empty and does not have an
maximal element. Assume also that there exists a key polynomial $Q\in\Lambda$
such that $\epsilon(Q)>\epsilon(M_{\alpha})$. We call a limit key
polynomial\emph{\index{Limit key polynomial}} every polynomial of
minimal degree which has this property.
\end{defn}

~
\begin{defn}
Let $\left(Q_{i}\right)_{i\in\mathbb{N}}$ be a sequence of key polynomials.
We say that it is a sequence of immediate successors if for every
integer $i$, we have $Q_{i}<Q_{i+1}$.
\end{defn}

\begin{prop}
If there are no limit key polynomials then there exists a finite or
infinite sequence of immediate successors $Q_{1}<\ldots<Q_{i}<\ldots$
such that the sequence $\left\{ \epsilon(Q_{i})\right\} $ is cofinal
in $\epsilon(\Lambda)$. Equivalently, such that 
\[
\forall Q\in\Lambda\text{ }\exists i\text{ such that }\epsilon(Q_{i})\geq\epsilon(Q).
\]
\end{prop}

\begin{proof}
We do the proof by contrapositive.

Assume that for every finite or infinite sequence of immediate successors
key polynomials $(Q_{i})$, the sequence $\left\{ \epsilon(Q_{i})\right\} $
is not cofinal in $\epsilon(\Lambda)$. Let us show that there exists
a limit key polynomial.

First let assume that for every $\alpha\in\Omega=\left\{ \beta\text{ such that }M_{\beta}\neq\emptyset\right\} $,
$M_{\alpha}$ has a maximal element. It means that 
\[
\forall\alpha\in\Omega\text{ }\exists R_{\alpha}\in M_{\alpha}\text{ such that }\forall Q\in M_{\alpha}\text{, }\epsilon(R_{\alpha})\geq\epsilon(Q).
\]

We set $M:=\left\{ R_{\alpha}\right\} _{\alpha\in\Omega}$. All elements
in $M$ are of distinct degree, so they are strictly ordered by their
degrees. So if $\alpha<\alpha'$, then $\deg(R_{\alpha})<\deg(R_{\alpha'})$.
Since $R_{\alpha'}$ is a key polynomial, by definition, we have $\epsilon(R_{\alpha})<\epsilon(R_{\alpha'})$
as soon as $\alpha<\alpha'$. Then in $M$ the elements are strictly
ordered by their values of$\epsilon$.

Let us show that they are immediate successors. Let $R_{\alpha}$
and $R_{\alpha'}$ be two consecutive elements of $M$. We know that
\[
\alpha=\deg(R_{\alpha})<\deg(R_{\alpha'})=\alpha'
\]
 and $\epsilon(R_{\alpha})<\epsilon(R_{\alpha'})$. We want to show
that $R_{\alpha'}$ is of minimal degree for the property. So let
us set $R\in\Lambda$ such that $\epsilon(R_{\alpha})<\epsilon(R)$
and $\deg(R)\leq\deg(R_{\alpha'})$. Let us show that $\deg(R)=\deg(R_{\alpha'})=\alpha'$.
Since $\epsilon(R_{\alpha})<\epsilon(R)$ and since $R_{\alpha}$
is a key polynomial, by definition, 
\[
\deg(R_{\alpha})=\alpha\leq\deg(R)\leq\alpha'.
\]
 Since $R$ is a key polynomial, if we had $\deg(R)=\deg(R_{\alpha})$,
then we should have $\epsilon(R_{\alpha})\geq\epsilon(R)$, which
is a contradiction. Let us set $\lambda:=\deg(R),$ so we have $\alpha<\lambda\leq\alpha'$,
$R\in M_{\lambda}$ and $R_{\lambda}\in M$. Since the polynomials
in $M$ are strictly ordered by their degrees and that $R_{\alpha}$
and $R_{\alpha'}$ are consecutive, then we have $\lambda=\alpha'$,
and so $R_{\alpha}<R_{\alpha'}$.

So the set $M$ is a sequence of immediate successors. By hypothesis,
the sequence $\epsilon(M)$ is not cofinal, so there exists $R\in\Lambda$
such that $\epsilon(R)>\epsilon(M)$. But then there exists $\alpha$
such that $R\in M_{\alpha}$ and then $\epsilon(R_{\alpha})\geq\epsilon(R)>\epsilon(R_{\alpha})$.
It is a contradiction.

So there exists $\alpha\in\Omega$ such that $M_{\alpha}$ does not
have any maximal ideal. Then we have a sequence: 
\[
\epsilon(Q_{1})<\epsilon(Q_{2})<\ldots<\epsilon(Q_{i})<\ldots
\]
where $Q_{i}$ is an element of $M_{\alpha}$ for every integer $i$.

Let us show that the $Q_{i}$ are immediate successors. Let $R\in\Lambda$
such that $\epsilon(Q_{i})<\epsilon(R)$ and $\deg(R)\leq\deg(Q_{i+1})=\alpha$.
Since $Q_{i}$ is a key polynomial, by definition, $\deg(R)\geq\deg(Q_{i})=\alpha$.
So $\deg(R)=\deg(Q_{i+1})=\alpha$, and $Q_{i+1}$ is of minimal degree
for the property. Then for every integer $i$, we have $Q_{i}<Q_{i+1}$.

By hypothesis, the sequence of the $Q_{i}$ is a sequence of immediate
successors, so the sequence $\left(\epsilon(Q_{i})\right)_{i}$ is
not cofinal. So there exists a key polynomial $Q\in\Lambda$ such
that $\epsilon(Q)>\epsilon(Q_{i})$ for every integer $i$. Let $R\in M_{\alpha}$,
since $M_{\alpha}$ does not have a maximal element, there exists
$i$ such that $\epsilon(R)<\epsilon(Q_{i})<\epsilon(Q)$. So there
exists a key polynomial $Q\in\Lambda$ such that $\epsilon(Q)>\epsilon(M_{\alpha})$.
Then the polynomial $Q$ is a limit key polynomial.
\end{proof}
\begin{thm}
\label{thm:il existe une suite cofinale}There exists a finite or
infinite sequence $(Q_{i})_{i\ge1}$ of key polynomials such that
for each $i$ the polynomial $Q_{i+1}$ is either an optimal or a
limit immediate successor of $Q_{i}$ and such that the sequence $\left\{ \epsilon(Q_{i})\right\} $
is cofinal in $\epsilon(\Lambda)$ where $\Lambda$ is the set of
key polynomials.
\end{thm}

\begin{proof}
We know that $x$ is a key polynomial. If for every key polynomial
$Q\in\Lambda$, we have $\epsilon\left(x\right)\geq\epsilon\left(Q\right)$,
then the sequence $\left\{ \epsilon(x)\right\} $ is cofinal in $\epsilon(\Lambda)$
and it is done. Otherwise, it exists a key polynomial $Q\in\Lambda$
such that $\epsilon\left(x\right)<\epsilon\left(Q\right)$. If it
exists a maximal element among the key polynomials of same degree
than $Q$, then we exchange $Q$ by this element. By Proposition \ref{prop: l'un ou l'autre optimal},
it exists a finite sequence $Q_{1}=x<\dots<Q_{p}=Q$ of optimal (possibly
limit) immediate successors which begins at $x$ and ends at $Q$.

If for every key polynomial $Q'\in\Lambda$, there exists a key polynomial
of this sequence $Q_{i}$ such that $\epsilon\left(Q_{i}\right)\geq\epsilon\left(Q'\right)$,
then the sequence $\left\{ \epsilon(Q_{i})\right\} $ is cofinal in
$\epsilon(\Lambda)$ and it is over.

Otherwise there exists a polynomial $Q'\in\Lambda$ such that for
every integer $i\in\left\{ 1,\dots,p\right\} $, we have $\epsilon\left(Q_{i}\right)<\epsilon\left(Q'\right)$.
So $\epsilon\left(Q_{p}\right)<\epsilon\left(Q'\right)$ and we use
Proposition \ref{prop: l'un ou l'autre optimal} again to construct
a sequence of optimal (possibly limit) immediate successors which
begins at $Q_{p}$ and ends at $Q'$. So we have a sequence $Q_{1}=x,\dots,Q_{r}=Q'$
of optimal (possibly limit) immediate successors which begins at $x$
and ends at $Q'$.

We iterate the process until the sequence $\left\{ \epsilon(Q_{i})\right\} $
is cofinal in $\epsilon(\Lambda)$. If $Q_{i}$ is maximal among the
set of key polynomials of degree $\deg_{X}\left(Q_{i}\right)$, then
$\deg_{X}\left(Q_{i}\right)<\deg_{X}\left(Q_{i+1}\right)$. If $Q_{i}<_{\lim}Q_{i+1}$,
we have again $\deg_{X}\left(Q_{i}\right)<\deg_{X}\left(Q_{i+1}\right)$.
In fact, the degree of the polynomials of the sequence stricly increase
at least each two steps, so the process stops.
\end{proof}
\begin{prop}
\label{prop: binoooomme}Assume that $k=k_{\nu}$. There exists a
finite or infinite sequence $(Q_{i})_{i\ge1}$ of key polynomials
such that for each $i$ the polynomial $Q_{i+1}$ is either an optimal
or a limit immediate successor of $Q_{i}$ and such that the sequence
$\left\{ \epsilon(Q_{i})\right\} $ is cofinal in $\epsilon(\Lambda)$
where $\Lambda$ is the set of key polynomials. 

And this sequence is such that: if $Q_{i}<Q_{i+1}$, then the $Q_{i}$-expansion
of $Q_{i+1}$ has exactly two terms.
\end{prop}

\begin{proof}
We have $Q_{1}=x$, and we assume that $Q_{1},Q_{2},\dots,Q_{i}$
have been constructed. We note $a:=\deg_{x}\left(Q_{i}\right)$ and
recall that 
\[
G_{<a}=\sum\limits _{\deg_{x}\left(P\right)<a}\mathrm{in}_{\nu_{Q_{i}}}\left(P\right)G_{\nu}.
\]

If $Q_{i}$ is maximal in $\Lambda$, we stop. Otherwise, $Q_{i}$
is not maximal and so it has an immediate successor.

We set $\alpha:=\min\left\{ h\in\mathbb{N}^{\ast}\text{ such that }h\nu\left(Q_{i}\right)\in\Delta_{<a}\right\} $
where $\Delta_{<a}$ is the subgroup of $\Gamma$ generated by the
values of the elements of $G_{<a}$.

In fact, there exists a polynomial $f$ of degree strictly less than
$a$ such that $\alpha\nu\left(Q_{i}\right)=\nu\left(Q_{i}^{\alpha}\right)=\nu\left(f\right)\neq0$.

Then, since $k_{\nu}=k$, there exists $c\in k^{\ast}$ such that
$\mathrm{in}_{\nu}\left(Q_{i}^{\alpha}\right)=\mathrm{in}_{\nu}\left(cf\right)$.

We set $Q=Q_{i}^{\alpha}-cf$. By the proof of Proposition \ref{prop: successeurs immediats et somme des formes initiales},
we have $\epsilon\left(Q_{i}\right)<\epsilon\left(Q\right)$.

Let us show that $Q_{i}<Q$. We only have to show that $Q$ is of
minimal degree. So let us set $P$ a key polynomial such that $\epsilon\left(Q_{i}\right)<\epsilon\left(P\right)$.

Assume by contradiction that $\deg\left(P\right)<a\alpha$. We set
$P=\sum\limits _{j=0}^{\alpha-1}p_{j}Q_{i}^{j}$ the $Q_{i}$-expansion
of $P$. Then by the proof of Proposition \ref{prop: successeurs immediats et somme des formes initiales},
we have $\sum\limits _{j=0}^{\alpha-1}\mathrm{in}_{\nu}\left(p_{j}\right)\mathrm{in}_{\nu}\left(Q_{i}\right)^{j}=0$,
which contradicts the minimality of $\alpha$.

Then $Q$ is of minimal degree and $Q_{i}<Q$. Since it has just two
terms in his $Q_{i}$-expansion, it is an optimal immediate successor
of $Q_{i}$.

First case: $\alpha>1$. Then we set $Q_{i+1}:=Q$ and we iterate.

Second case: $\alpha=1$. Then all the elements of $M_{Q_{i}}$ have
same degree than $Q_{i}$. If $M_{Q_{i}}$ does not have a maximal
element, then we do the same thing than in the proof of Proposition
\ref{prop: l'un ou l'autre optimal} and we set $Q_{i+1}$ a limit
immediate successor of $Q_{i}$.

Otherwise, $M_{Q_{i}}$ has a maximal element $Q_{i+1}$. This element
has same degree as $Q_{i}$, so we have $Q_{i+1}=Q_{i}-h$ with $h$
of degree strictly less than the degree of $Q_{i}$. Then it is an
immediate successor of $Q_{i}$ which $Q_{i}$-expansion admits uniquely
two terms. So it is optimal, and this completes the proof.
\end{proof}
We now assume $k=k_{\nu}$ and consider $\mathcal{Q}:=\left(Q_{i}\right)_{i}$
a sequence of optimal (possibly limit) immediate successors such that
$\left(\epsilon\left(Q_{i}\right)\right)_{i}$ is cofinal in $\epsilon\left(\Lambda\right)$
and such that if $Q_{i}<Q_{i+1}$, then the $Q_{i}$-expansion of
$Q_{i+1}$ admits exactly two terms.
\begin{rem}
We keep the same hypothesis as in Example \ref{exa:pc}. Then $\mathcal{Q}=\left\{ z,Q\right\} $.
\end{rem}

\begin{cor}
For every polynomial $f$, there exists an index $i$ such that $\nu_{Q_{i}}(f)=\nu(f)$.
\end{cor}

\begin{proof}
By Proposition \ref{prop:non degeneresence de polyn=0000F4mes par rapport au m=0000EAme polyn=0000F4me clef},
there exists a key polynomial$Q$ such that $\nu_{Q}(f)=\nu(f)$.

The sequence $\left\{ \epsilon(Q_{i})\right\} $ being cofinal, there
exists an index $i$ such that
\[
\epsilon(Q_{i})\geq\epsilon(Q).
\]
 By Proposition \ref{prop: inegalite des valuations tronqu=0000E9es},
$\nu_{Q}(f)\leq\nu_{Q_{i}}(f)$ and since $\nu_{Q}(f)=\nu(f)$, we
have $\nu_{Q_{i}}(f)=\nu(f)$.
\end{proof}
\begin{rem}
\label{rem: tout element est non degenere par rapport a un poly clef}So,
for every polynomial $f$, there exists a key polynomial $Q_{i}$
of the sequence $\mathcal{Q}$ such that $f$ is non-degenerate with
respect to $Q_{i}$.
\end{rem}

~
\begin{rem}
\label{rem: le r peut pas monter trop}Let $Q_{i}\in\mathcal{Q}$.
We don't assume here $k=k_{\nu}$.

We set $a_{i}:=\deg_{x}\left(Q_{i}\right)$ and $\Gamma_{<a_{i}}$
the group $\nu\left(G_{<a_{i}}\setminus\left\{ 0\right\} \right)$.

If $\nu\left(Q_{i}\right)\notin\Gamma_{<a_{i}}\otimes_{\mathbb{Z}}\mathbb{Q}$,
then $\epsilon\left(Q_{i}\right)$ is maximal in $\epsilon\left(\Lambda\right)$
and the sequence $\mathcal{Q}$ stops at $Q_{i}$.
\end{rem}

\newpage{}

\section{\label{sec:Monomialisation-des-polyn=0000F4mes}Monomialization of
the key polynomials.}

We set $K:=k\left(u_{1},\dots,u_{n-1}\right)$ and we consider the
extension $K(u_{n})$. We consider also a sequence of key polynomials
$\mathcal{Q}$ as in the section \ref{sectionsuitecofinale}.

In other words, $\mathcal{Q}=\left(Q_{i}\right)_{i}$ is a sequence
of optimal (possibly limit) immediate successors such that $\left(\epsilon\left(Q_{i}\right)\right)_{i}$
is cofinal in $\epsilon\left(\Lambda\right)$.

Let $f$ be an element of $R$. We know that this element is non-degenerate
with respect to a key polynomial of the sequence $\mathcal{Q}$. We
also know that every element non-degenerate with respect to a regular
system of parameters is monomializable. 

Then, to monomialize $f$, it is enough to monomialize the set of
key polynomials of this sequence. We assume in this part that the
residue field is $k$.

\subsection{Generalities.}

Let $r:=r\left(R,u,\nu\right)$ be the dimension of 
\[
\sum\limits _{i=1}^{n}\nu(u_{i})\mathbb{Q}
\]
 in $\Gamma\otimes_{\mathbb{Z}}\mathbb{Q}$. Renumbering, if necessary,
we can assume that $\nu\left(u_{1}\right),\dots,\nu\left(u_{r}\right)$
are rationally independent and we consider $\Delta$ the subgroup
of $\Gamma$ generated by $\nu(u_{1}),\ldots,\nu(u_{r})$. 
\begin{rem}
Let $\left(R,u\right)\to\left(R_{1},u^{\left(1\right)}\right)$ be
a framed blow-up. Then $r\leq r_{1}:=r\left(R_{1},u^{\left(1\right)},\nu\right)$. 
\end{rem}

~
\begin{rem}
We will consider the framed local blow-ups
\[
\left(R,u\right)\to\dots\to\left(R_{i},u^{\left(i\right)}\right)\to\dots
\]
 Then we write $r_{i}:=r\left(R_{i},u^{\left(i\right)},\nu\right)$.
\end{rem}

We set $E:=\left\{ 1,\dots,r,n\right\} $ and $\overline{\alpha^{(0)}}:=\min\limits _{h\in\mathbb{N}^{\ast}}\left\{ h\text{ such that }h\nu(u_{n})\in\Delta\right\} .$ 

So $\overline{\alpha^{(0)}}\nu(u_{n})=\sum\limits _{j=1}^{r}\alpha_{j}^{(0)}\nu(u_{j})$
with, renumbering the $\alpha_{i}^{(0)}$ if necessary, 
\[
\alpha_{1}^{(0)},\ldots,\alpha_{s}^{(0)}\geq0
\]
and 
\[
\alpha_{s+1}^{(0)},\ldots,\alpha_{r}^{(0)}<0.
\]

We set 
\[
w=(w_{1},\ldots,w_{r},w_{n})=(u_{1},\ldots,u_{r},u_{n})
\]
 and 
\[
v=(v_{1},\ldots,v_{t})=(u_{r+1},\ldots,u_{n-1}),
\]
 with $t=n-r-1$.

We set $x_{i}=\mathrm{in}_{\nu}u_{i}$ , and we have that $x_{1},\ldots,x_{r}$
are algebraically independent over $k$ in $G_{\nu}$. Let $\lambda_{0}$
be the minimal polynomial of $x_{n}$ over $k\left(x_{1},\ldots,x_{r}\right)$,
of degree $\alpha$.

We set:

\[
y=\prod\limits _{j=1}^{r}x_{j}^{\alpha_{j}^{(0)}},
\]
 
\[
\overline{y}=\prod\limits _{j=1}^{r}w_{j}^{\alpha_{j}^{(0)}},
\]
 
\[
z=\frac{x_{n}^{\overline{\alpha^{(0)}}}}{y}
\]
 and 
\[
\overline{z}=\frac{w_{n}^{\overline{\alpha^{(0)}}}}{\overline{y}}.
\]

We have 
\[
\lambda_{0}=X^{\alpha}+c_{0}y
\]
 where $c_{0}\in k$ , and $Z+c_{0}$ is the minimal polynomial $\lambda_{z}$
of $z$ over $\mathrm{gr}_{\nu}k\left(x_{1},\dots,x_{r}\right)$. 

Indeed, $k_{\nu}\simeq k\simeq\frac{k\left[Z\right]}{\left(\lambda_{z}\right)}$
so $\lambda_{z}$ is of degree $1$ in $Z$. Then $\lambda_{0}$ is
of degree $\overline{\alpha^{(0)}}$, and so $\alpha=\overline{\alpha^{\left(0\right)}}$.
\begin{defn}
\index{Monomializable}We say that $Q_{i}$ is \emph{monomializable}
if there exists a sequence of blow-ups $\left(R,u\right)\to\left(R_{l},u^{\left(l\right)}\right)$
such that in $R_{l}$, $Q_{i}$ can be written as $u_{n}^{\left(l\right)}$
multiplied by a monomial in $\left(u_{1}^{\left(l\right)},\dots,u_{r_{l}}^{\left(l\right)}\right)$
up to a unit of $R_{l}$, where $r_{l}:=r\left(R_{l},u^{\left(l\right)},\nu\right)$.
\end{defn}

We are going to show that there exists a local framed sequence that
monomializes all the $Q_{i}$.

We have $Q_{1}=u_{n}$, it is a monomial. By the blow-ups, $Q_{1}$
stays a monomial. So we have to begin monomializing $Q_{2}$.

Since we want to monomialize the key polynomials $Q_{i}$ of the sequence
$\mathcal{Q}$ constructed earlier by induction on $i$, we are going
to do something more general here: we consider an immediate successors
(possibly limit) key element $Q_{2}$ of $Q_{1}$ instead of immediate
successor (possibly limit) key polynomial of $Q_{1}$.

First, let us consider
\[
Q=w_{n}^{\alpha}+a_{0}b_{0}\overline{y}
\]
 where $b_{0}\in R$ such that $b_{\text{0}}\equiv c_{0}$ modulo
$\mathfrak{m}$ and $a_{0}\in R^{\times}$. 

A priori, $Q$ is not a key polynomial but we are going to prove that
we can reduce this case to the case $Q_{2}=Q$ by a local framed sequence
independent of $u_{n}$.

\subsection{Puiseux packages.}

~\\
Let 
\[
\gamma=(\gamma_{1},\ldots,\gamma_{r},\gamma_{n})=(\alpha_{1}^{(0)},\ldots,\alpha_{s}^{(0)},0,\ldots,0)
\]
 and 
\[
\delta=(\delta_{1},\ldots,\delta_{r},\delta_{n})=(0,\ldots,0,-\alpha_{s+1}^{(0)},\ldots,-\alpha_{r}^{(0)},\alpha).
\]

We have 
\[
w^{\delta}=w_{n}^{\delta_{n}}\prod\limits _{j=1}^{r}w_{j}^{\delta_{j}}=\frac{w_{n}^{\alpha}}{\prod\limits _{j=s+1}^{r}w_{j}^{\alpha_{j}^{(0)}}}
\]
 and 
\[
w^{\gamma}=\prod\limits _{j=1}^{s}w_{j}^{\alpha_{j}^{(0)}}.
\]
 So $\frac{w^{\delta}}{w^{\gamma}}=\frac{w_{n}^{\alpha}}{\prod\limits _{j=1}^{r}w_{j}^{\alpha_{j}^{(0)}}}=\overline{z}$.

Let us compute the value of $w^{\delta}$.

\[
\begin{array}{ccc}
\nu(w^{\delta}) & = & \alpha\nu(w_{n})-\sum\limits _{j=s+1}^{r}\alpha_{j}^{(0)}\nu(w_{j})\\
 & = & \alpha\nu(u_{n})-\sum\limits _{j=s+1}^{r}\alpha_{j}^{(0)}\nu(u_{j})\\
 & = & \sum\limits _{j=1}^{r}\alpha_{j}^{(0)}\nu(u_{j})-\sum\limits _{j=s+1}^{r}\alpha_{j}^{(0)}\nu(u_{j})\\
 & = & \sum\limits _{j=1}^{s}\alpha_{j}^{(0)}\nu(w_{j})\\
 & = & \nu(\prod\limits _{j=1}^{s}w_{j}^{\alpha_{j}^{(0)}})\\
 & = & \nu(w^{\gamma}).
\end{array}
\]

\begin{thm}
\label{thm:monomialise}There exists a local framed sequence 

\begin{equation}
\left(R,u\right)\overset{\pi_{0}}{\to}\left(R_{1},u^{(1)}\right)\overset{\pi_{1}}{\to}\cdots\overset{\pi_{l-1}}{\to}\left(R_{l},u^{(l)}\right)\label{eq:suitelocale}
\end{equation}

with respect to $\nu$, independent of $v$, and that has the next
properties:

For every integer $i\in\{1,\dots,l\}$, we write $u^{(i)}:=\left(u_{1}^{(i)},\dots,u_{n}^{(i)}\right)$
and we recall that $k$ is the residue field of $R_{i}$.
\begin{enumerate}
\item The blow-ups $\pi_{0},\dots,\pi_{l-2}$ are monomials.
\item We have $\overline{z}\in R_{l}^{\times}$.
\item We set $u^{\left(l\right)}:=\left(w_{1}^{(l)},\dots,w_{r}^{(l)},v,w_{n}^{(l)}\right)$.
So for every integer $j\in\left\{ 1,\dots,r,n\right\} $, $w_{j}$
is a monomial in $w_{1}^{(l)},\dots,w_{r}^{(l)}$ multiplied by an
element of $R_{l}^{\times}$. And for every integer $j\in\left\{ 1,\dots,r\right\} $,
$w_{j}^{(l)}=w^{\eta}$ where $\eta\in\mathbb{Z}^{r+1}$.
\item We have $Q=w_{n}^{(l)}\times\overline{y}$.
\end{enumerate}
\end{thm}

\begin{proof}
We apply Proposition \ref{prop:existencesuitedivise} to $(w^{\delta},w^{\gamma})$
and so we obtain a local framed sequence for $\nu$, independent of
$v$ and such that $w^{\gamma}\mid w^{\delta}$ in $R_{l}$. 

By Proposition \ref{prop:divaleur} and the fact that $w^{\delta}$
and $w^{\gamma}$ have same value, we have that $w^{\delta}\mid w^{\gamma}$
in $R_{l}$. In fact $\overline{z},\overline{z}^{-1}\in R_{l}^{\times}.$
So we have $(2)$.

We choose the local sequence to be minimal, in other words the sequence
made by $\pi_{0},\ldots,\pi_{l-2}$ does not satisfy the conclusion
of the Proposition \ref{prop:existencesuitedivise} for $(w^{\delta},w^{\gamma})$.
Now we are going to prove that this sequence satisfies the five properties
of Theorem \ref{thm:monomialise}. Let $i\in\left\{ 0,\dots,l\right\} $.
We write $w^{(i)}=\left(w_{1}^{(i)},\dots,w_{r}^{(i)},w_{n}^{(i)}\right)$,
with $r=n-t-1$ and define $J_{i}$, $A_{i}$, $B_{i}$, $j_{i}$
and $D_{i}$ similarly that we defined $J$, $A$, $B$, $j$ and
$D_{1}$, considering the $i$-th blow-up. 

Since $D_{i}\subset\left\{ 1,\dots,n\right\} $, we have $\sharp D_{i}\leq n$.
Hence $\sharp(A_{i}\cup(B_{i}\cup\left\{ j_{i}\right\} ))\leq n$,
so $\sharp A_{i}+\sharp B_{i}+1\leq n$. As the sequence is independent
of $v$, this implies that $T\subset A_{i}$, and so $\sharp T\leq\sharp A_{i}$.
Then $\sharp T+1+\sharp B_{i}\leq n$, so $t+1\leq n$, and so $r\geq0$.
By the minimality of the sequence, we know that if $i<l$, $w^{\delta}\nmid w^{^{\gamma}}$
in $R_{i}$, and so $\sharp B_{i}\neq0$, hence $r>0$.

For every integers $i\in\left\{ 1,\dots,l\right\} $ and $j\in\left\{ 1,\dots,n\right\} $,
we set $\beta_{j}^{(i)}=\nu\left(u_{j}^{(i)}\right)$. For each $i<l$,
$\pi_{i}$ is a blow-up along an ideal of the form $\left(u_{J_{i}}^{(i)}\right)$.
Renumbering if necessary, we may assume that $1\in J_{i}$ and that
$R_{i+1}$ is a localisation of $R_{i}\left[\frac{u_{J_{i}}^{(i)}}{u_{1}^{(i)}}\right]$.
So we have $\beta_{1}^{(i)}=\min\limits _{j\in J_{i}}\left\{ \beta_{j}^{(i)}\right\} $.
\begin{fact}
\label{fact:premierentreeux}Let $X=\left(x_{1},\dots,x_{n}\right)\in\mathbb{Z}^{n}$
be a vector whose elements are relatively prime. Then there exists
a matrix $A\in\mathrm{SL}_{n}(\mathbb{Z})$ of determinant $1$ such
that $X$ is the first line of $A$.
\end{fact}

\begin{proof}
This proof is made by induction on $n$ and using Bezout theorem.
\end{proof}
\begin{lem}
\label{lem:pgcdtotal}Let $i\in\left\{ 0,\dots,l-1\right\} $. We
assume that the sequence $\pi_{0},\dots,\pi_{i-1}$ of \ref{eq:suitelocale}
is monomial. 

We set $w^{\gamma}=\left(w^{(i)}\right)^{\gamma^{(i)}}$ and $w^{\delta}=\left(w^{(i)}\right)^{\delta^{(i)}}$.
Then:
\end{lem}

\begin{enumerate}
\item 
\begin{equation}
\sum\limits _{q\in E}\left(\gamma_{q}^{(i)}-\delta_{q}^{(i)}\right)\beta_{q}^{(i)}=0,\label{eq:relation}
\end{equation}
\item $\mathrm{pgcd}\left(\gamma_{1}^{(i)}-\delta_{1}^{(i)},\dots,\gamma_{r}^{(i)}-\delta_{r}^{(i)},\gamma_{n}^{(i)}-\delta_{n}^{(i)}\right)=1,$
\item Every $\mathbb{Z}$-linear dependence relation between $\beta_{1}^{(i)},\dots,\beta_{r}^{(i)},\beta_{n}^{(i)}$
is an integer multiple of (\ref{eq:relation}).
\end{enumerate}
\begin{proof}
~
\begin{enumerate}
\item We have $\nu\left(w^{\gamma}\right)=\nu\left(w^{\delta}\right)$,
hence $\nu\left(\left(w^{(i)}\right)^{\gamma^{(i)}}\right)=\nu\left(\left(w^{(i)}\right)^{\delta^{(i)}}\right)$.
So, since $w^{(i)}=\left(w_{1}^{(i)},\dots,w_{r}^{(i)},w_{n}^{(i)}\right)$:
\[
\nu\left(\prod\limits _{j=1}^{r}\left(w_{j}^{(i)}\right)^{\gamma_{j}^{(i)}}\times\left(w_{n}^{(i)}\right)^{\gamma_{n}^{(i)}}\right)=\nu\left(\prod\limits _{j=1}^{r}\left(w_{j}^{(i)}\right)^{\delta_{j}^{(i)}}\times\left(w_{n}^{(i)}\right)^{\delta_{n}^{(i)}}\right)
\]
in other words 
\[
\sum\limits _{j=1}^{r}\gamma_{j}^{(i)}\nu\left(w_{j}^{(i)}\right)+\gamma_{n}^{(i)}\nu\left(w_{n}^{(i)}\right)=\sum\limits _{j=1}^{r}\delta_{j}^{(i)}\nu\left(w_{j}^{(i)}\right)+\delta_{n}^{(i)}\nu\left(w_{n}^{(i)}\right).
\]
By definition of $w^{(i)}$, for every integer $j\in\left\{ 1,\dots,r,n\right\} $,
we have $w_{j}^{(i)}=u_{j}^{(i)}$, so $\nu(w_{j}^{(i)})=\beta_{j}^{(i)}$.
Then:
\[
\sum\limits _{j=1}^{r}\gamma_{j}^{(i)}\beta_{j}^{(i)}+\gamma_{n}^{(i)}\beta_{n}^{(i)}=\sum\limits _{j=1}^{r}\delta_{j}^{(i)}\beta_{j}^{(i)}+\delta_{n}^{(i)}\beta_{n}^{(i)}.
\]
Then $\sum\limits _{j\in\left\{ 1,\dots,r,n\right\} }\left(\gamma_{j}^{(i)}-\delta_{j}^{(i)}\right)\beta_{j}^{(i)}=0$. 
\item We do an induction. Case $i=0$. \\
We have 
\[
\begin{array}{c}
\mathrm{pgcd}\left(\gamma_{1}^{(i)}-\delta_{1}^{(i)},\dots,\gamma_{r}^{(i)}-\delta_{r}^{(i)},\gamma_{n}^{(i)}-\delta_{n}^{(i)}\right)\\
=\mathrm{pgcd}\left(\gamma_{1}^{(0)}-\delta_{1}^{(0)},\dots,\gamma_{r}^{(0)}-\delta_{r}^{(0)},\gamma_{n}^{(0)}-\delta_{n}^{(0)}\right)\\
=\mathrm{pgcd}\left(\alpha_{1}^{(0)},\dots,\alpha_{s}^{(0)},\alpha_{s+1}^{(0)},\dots,\alpha_{r}^{(0)},-\overline{\alpha^{(0)}}\right).
\end{array}
\]
By definition 
\[
\alpha=\overline{\alpha^{(0)}}=\min\limits _{h\in\mathbb{N}^{\ast}}\left\{ h\text{ such that }h\beta_{n}\in\Delta\right\} 
\]
 and 
\[
\alpha\beta_{n}=\sum\limits _{j=1}^{r}\alpha_{j}^{(0)}\beta_{j}.
\]
 So $\mathrm{pgcd}\left(\alpha_{1}^{(0)},\dots,\alpha_{s}^{(0)},\alpha_{s+1}^{(0)},\dots,\alpha_{r}^{(0)},-\alpha\right)=1$.\\
Case $i>0$. We assume the result shown at the previous rank. We have
$\gamma^{(i)}=\gamma^{(i-1)}G^{(i)}$, $\delta^{(i)}=\delta^{(i-1)}G^{(i)}$
and $\beta^{(i)}=\beta^{(i-1)}F^{(i)}$ where $F^{(i)}=\left(G^{(i)}\right)^{-1}$
and $G^{(i)}\in\mathrm{SL}_{r+1}(\mathbb{Z})$ such that 
\[
G_{sq}^{(i)}=\begin{cases}
1 & \text{if }s=q\\
1 & \text{if }q=j\text{ and }s\in J\\
0 & \text{otherwise.}
\end{cases}
\]
So $\left(\gamma^{(i)}-\delta^{(i)}\right)=\left(\gamma^{(i-1)}-\delta^{(i-1)}\right)G^{(i)}=\left(\gamma-\delta\right)G$
where $G$ is a product of unimodular matrixes, and so $G$ is unimodular.\\
By the case $i=0$, $\left(\gamma-\delta\right)$ is a vector whose
elements are relatively prime. By \ref{fact:premierentreeux} this
vecteur can be complete as a base of $\mathbb{Z}^{r+1}$, which, by
a unimodular matrix, stay a base of $\mathbb{Z}^{r+1}$. The vector
$\left(\gamma^{(i)}-\delta^{(i)}\right)$ is then a vector of this
base, so its elements are relatively prime.
\item Case $i=0$ is the fact that $\beta_{1},\dots,\beta_{r},\beta_{n}$
generate a vector space of dimension $r$. \\
Let 
\[
Z:=\left\{ (x_{1},\dots,x_{r+1})\in\mathbb{Z}^{r+1}\text{ such that }\sum\limits _{j=1}^{^{r}}x_{j}\beta_{j}+x_{r+1}\beta_{n}=0\right\} .
\]
But $\alpha\beta_{n}=\sum\limits _{j=1}^{r}\alpha_{j}^{(0)}\beta_{j}$,
so: 
\[
Z=\left\{ (x_{1},\dots,x_{r+1})\in\mathbb{Z}^{r+1}\text{ such that }\sum\limits _{j=1}^{^{r}}\left(\alpha x_{j}+x_{r+1}\alpha_{j}^{(0)}\right)\beta_{j}=0\right\} .
\]
 Since $\beta_{1},\dots,\beta_{r}$ are $\mathbb{Q}$-linearly independent
elements, we have that $Z$ is a free $\mathbb{Z}$-module of rank
$1$, so it is generated by a unique vector. By point $\left(1\right)$,
the vector $(\gamma-\delta)$ is in $Z$, and by point $(2)$, it
is composed of relatively prime elements. This vector generates the
free $\mathbb{Z}$-module of rank $1$.\\
Let $i>0$. We already know that $\beta^{(i)}=\beta^{(i-1)}F^{(i)}=\beta F$
where $F$ is a unimodular matrix, so an automorphism of $\mathbb{Z}^{r}$.
\\
Let 
\[
Z^{(i)}:=\left\{ (x_{1},\dots,x_{r+1})\in\mathbb{Z}^{r+1}\text{ such that }\sum\limits _{j=1}^{r}x_{j}\beta_{j}^{(i)}+x_{r+1}\beta_{n}^{(i)}=0\right\} .
\]
So 
\[
Z^{(i)}=\left\{ (x_{1},\dots,x_{r+1})\in\mathbb{Z}^{r+1}\text{ such that }\sum\limits _{j=1}^{r}x_{j}\beta_{j}F+x_{r+1}\beta_{n}F=0\right\} ,
\]
 then 
\[
Z^{(i)}=\left\{ (x_{1},\dots,x_{r+1})\in\mathbb{Z}^{r+1}\text{ such that }\sum\limits _{j=1}^{r}x_{j}\beta_{j}+x_{r+1}\beta_{n}=0\right\} .
\]
 Then the set $Z^{(i)}$ is a free $\mathbb{Z}$-module of rank $1$
by the case $i=0$. And we know by (3) that the vector $\left(\gamma^{(i)}-\delta^{(i)}\right)$
is a vector of $Z^{(i)}$ composed of relatively prime elements, so
it generates $Z^{(i)}$. This completes the proof.
\end{enumerate}
\end{proof}
\begin{lem}
\label{lem:suitenonmonomiale}The sequence (\ref{eq:suitelocale})
is not monomial.
\end{lem}

\begin{proof}
Assume, aiming for contradiction, that it is. By induction on $i$,
we have $r_{i}=r$ for every $i\in\left\{ 0,\dots,l\right\} $. We
know that $w^{(l)}$ is a regular system of parameters of $R_{l}$
and that $w^{\delta}$ and $w^{\gamma}$ divide each other in $R_{l}$.

We saw that 
\[
\begin{array}{ccc}
\gamma^{(l)} & = & \gamma^{(l-1)}G^{(l)}\\
 & = & \gamma\prod\limits _{j\in\left\{ 1,\dots,l\right\} }G^{(j)}
\end{array}
\]
 and 
\[
\begin{array}{ccc}
\delta^{(l)} & = & \delta^{(l-1)}G^{(l)}\\
 & = & \delta\prod\limits _{j\in\left\{ 1,\dots,l\right\} }G^{(j)}.
\end{array}
\]
 So $\delta^{(l)}=\gamma^{(l)}$ .

But $\left(\gamma^{(l)}-\delta^{(l)}\right)=\left(\gamma-\delta\right)G$
where $G$ is a unimodular matrix, hence $\gamma=\delta$, which is
a contradiction.
\end{proof}
\begin{lem}
\label{lem:chaqueeclatementestmonomial}Let $i\in\{0,\dots,l-1\}$
and assume $\pi_{0},\dots,\pi_{i-1}$ are all monomials. Then the
following assertions are equivalent:
\begin{enumerate}
\item The blow-up $\pi_{i}$ is not monomial.
\item There exists a unique index $q\in J_{i}\setminus\left\{ 1\right\} $
such that $\beta_{q}^{(i)}=\beta_{1}^{(i)}$.
\item We have $i=l-1$.
\end{enumerate}
\end{lem}

\begin{proof}
$(3)\Rightarrow(1)$ by Lemma \ref{lem:suitenonmonomiale}.

$(1)\Rightarrow(2)$ First, we prove the existence. We have$\beta_{1}^{(i)}=\min\limits _{j\in J_{i}}\left\{ \beta_{j}^{(i)}\right\} $.
So $\pi_{i}$ monomial $\Leftrightarrow B_{i}=J_{i}\setminus\left\{ 1\right\} \Leftrightarrow\beta_{q}^{(i)}>\beta_{1}^{(i)}$
for every $q\in J_{i}\setminus\left\{ 1\right\} $.

Since the blow-up is not monomial by hypothesis, there exists $q\in J_{i}\setminus\left\{ 1\right\} $
such that $\beta_{q}^{(i)}=\beta_{1}^{(i)}$.

Now let us show the unicity. Assume, aiming for contradiction, that
there exist two different indexes $q$ and $q'$ in $J_{i}\setminus\left\{ 1\right\} $
such that $\beta_{q}^{(i)}-\beta_{1}^{(i)}=0$ and $\beta_{q'}^{(i)}-\beta_{1}^{(i)}=0$
.

Then we have two linear dependence relations between $\beta_{1}^{(i)},\dots,\beta_{r}^{(i)}$
and the element $\beta_{n}^{(i)}$, which are not linearly dependent.
It is a contradiction by point $(4)$ of Lemma \ref{lem:pgcdtotal}.

$(2)\Rightarrow(3)$

By Remark \ref{rem:wenfonctiondew'}, we write $w_{1}^{(i)}=w^{\epsilon}$
and $w_{q}^{(i)}=w^{\mu}$ where $\epsilon$ and $\mu$ are two colons
of an unimodular matrix. Then $\epsilon-\mu$ is unimodular, so its
total pgcd is one.

So 
\[
\nu(w^{\mu})=\sum\limits _{s\in E}\mu_{s}\beta_{s}=\nu(w_{q}^{(i)})=\beta_{q}^{(i)}
\]
 and 
\[
\nu(w^{\epsilon})=\sum\limits _{s\in E}\epsilon_{s}\beta_{s}=\nu(w_{1}^{(i)})=\beta_{1}^{(i)}.
\]

By hypothesis, $\beta_{q}^{(i)}=\beta_{1}^{(i)}$. Then $\sum\limits _{s\in E}\left(\mu_{s}-\epsilon_{s}\right)\beta_{s}=0$
and by points $(3)$ and $(4)$ of Lemma \ref{lem:pgcdtotal}, and
the fact that the total pgcd of $\mu-\epsilon$ is one, we have $\mu-\epsilon=\pm(\gamma-\delta)$.

So $\frac{w_{q}^{(i)}}{w_{1}^{(i)}}=w^{\epsilon-\mu}=w^{\pm(\gamma-\delta)}=\overline{z}^{\pm1}$,
then either $\overline{z}\in R_{i+1}$ or $\overline{z}^{-1}\in R_{i+1}$. 

To show that $i=l-1$, we are going to show that $i+1=l$ . And to
do this, we are going to use the fact that $l$ has been chosen minimal
such that $\overline{z}\in R_{l}^{\times}$. So let us show that $\overline{z}\in R_{i+1}^{\times}$.

Since $\overline{z}\in R_{i+1}$ or $\overline{z}^{-1}\in R_{i+1}$,
we know that $w^{\delta}\mid w^{\gamma}$ in $R_{i+1}$ or the converse.
By Proposition \ref{prop:divaleur} and the fact that $w^{\delta}$
and $w^{\gamma}$ have same value, we have $w^{\delta}\mid w^{\gamma}$
in $R_{i+1}$ if and only if the converse is true. So $\overline{z}\in R_{i+1}^{\times}$,
and the proof is complete.
\end{proof}
Doing an induction on $i$ and using Lemma \ref{lem:chaqueeclatementestmonomial},
we conclude that $\pi_{0},\dots,\pi_{l-2}$ are monomials. So we have
the first point of Theorem \ref{thm:monomialise}.

It remains to show the points $(3)$ and $(4)$.

By Lemma \ref{lem:chaqueeclatementestmonomial} there exists a unique
element $q\in J_{l-1}\setminus\left\{ j_{l-1}\right\} $ such that
$\beta_{q}^{(l-1)}=\beta_{1}^{(l-1)}$, so we are in the case $\sharp B_{l-1}+1=\sharp J_{l-1}-1$.
Now we have to see if we are in the case $t_{k_{l-1}}=0$ or in the
case $t_{k_{l-1}}=1$.

We recall that $w_{1}^{(l-1)}=w^{\epsilon}$ and $w_{q}^{(l-1)}=w^{\mu}$
where $\epsilon$ and $\mu$ are two colons of a unimodular matrix
such that $\mu-\epsilon=\pm(\gamma-\delta)$. So we have $x_{1}^{(l-1)}=x^{\epsilon}$
and $x_{q}^{(l-1)}=x^{\mu}$, then 
\[
\frac{x_{q}^{(l-1)}}{x_{1}^{(l-1)}}=x^{\mu-\epsilon}=x^{\pm\left(\gamma-\delta\right)}=x^{\pm\left(\alpha_{1}^{(0)},\dots,\alpha_{r}^{(0)},-\alpha\right)}.
\]

In other words
\[
\frac{x_{q}^{(l-1)}}{x_{1}^{(l-1)}}=\left(\frac{\prod\limits _{j=1}^{r}x_{j}^{\alpha_{j}^{(0)}}}{x_{n}^{\alpha}}\right)^{\pm1}=\left(z^{-1}\right)^{\pm1}=z^{\pm1}.
\]
 Replacing $x_{1}^{(l-1)}$ and $x_{q}^{(l-1)}$ if necessary, we
may assume $\frac{x_{q}^{(l-1)}}{x_{1}^{(l-1)}}=z$ .

Since $\beta_{1}^{(l-1)},\dots,\beta_{r}^{(l-1)}$ are linearly independent,
we have $q=n$.

We recall that $\lambda_{0}=X^{\alpha}+c_{0}y$ where $c_{0}\in k$
, and $Z+c_{0}$ is the minimal polynomial $\lambda_{z}$ of $z$
on $\mathrm{gr}_{\nu}k\left(x_{1},\dots,x_{r}\right)$. By \ref{fact:nombredegene},
we have 
\[
w_{n}^{(l)}=u_{n}^{(l)}=\overline{\lambda_{0}}(u'_{n})=\overline{\lambda_{0}}\left(\frac{u_{n}^{(l-1)}}{u_{1}^{(l-1)}}\right)=\overline{\lambda_{0}}\left(\frac{w_{n}^{(l-1)}}{w_{1}^{(l-1)}}\right)=\overline{\lambda_{0}}\left(\overline{z}\right)=\overline{z}+a_{0}b_{0}.
\]
\begin{rem}
We know that $\overline{\lambda_{0}}\left(\overline{z}\right)=\overline{z}+b_{0}g_{0}$
where $g_{0}$ is a unit and$b_{0}\in R$ such that $b_{\text{0}}\equiv c_{0}$
modulo $\mathfrak{m}$. Then we choose $g_{0}=a_{0}$.
\end{rem}

But $\overline{z}=\frac{w_{n}^{\alpha}}{\overline{y}}$, so 
\[
w_{n}^{(l)}=\frac{w_{n}^{\alpha}}{\overline{y}}+a_{0}b_{0}=\frac{w_{n}^{\alpha}+a_{0}b_{0}\overline{y}}{\overline{y}}=\frac{Q}{\overline{y}}
\]
 as desired in point $(4)$. 

Let us show the point $(3)$. We apply Proposition \ref{prop:chgmtvari}
at $i=0$ and $i'=l$. By the monomiality of $\pi_{0},\dots,\pi_{l-2}$,
we know that $D_{i}=\{1,\dots,n\}$ for each $i\in\{1,\dots,l-1\}$,
and we know that $D_{l}=\{1,\dots,n\}$. We set $u_{T}=v$.

For every $j\in\left\{ 1,\dots,r,n\right\} $, the fact that $w_{j}=u_{j}$
is a monomial in $w_{1}^{(l)},\dots,w_{r}^{(l)}$, in other words
in $u_{1}^{(l)},\dots,u_{r}^{(l)}$, multiplied by an element of $R_{l}^{\times}$
is a consequence of Proposition \ref{prop:chgmtvari}.

The fact that for every integer $j\in\{1,\dots,r\}$, we have $w_{j}^{(l)}=w^{\eta}$
is a consequence of the same Proposition. This completes the proof.
\end{proof}
\begin{rem}
In the case $Q_{2}=Q$, we monomialized $Q_{2}$ as desired.
\end{rem}

\begin{defn}
\cite{CRS} \index{Puiseux package}\label{def:Puiseux}A local framed
sequence that satisfies Theorem \ref{thm:monomialise} is called a\emph{
$n$-}Puiseux package.

Let $j\in\left\{ r+1,\dots,n\right\} $. A \emph{$j$-Puiseux }package
is a $n$-Puiseux package replacing $n$ by $j$ in Theorem \ref{thm:monomialise}.
\end{defn}

\begin{lem}
\label{lem: on se ramene au bon cas}Let $P=u_{n}^{\alpha}+c_{0}$
be the $u_{n}$-expansion of an immediate successor key element of
$u_{n}$. 

There exists a local framed sequence $\left(R,u\right)\to\left(R_{l},u^{(l)}\right)$,
independant of $u_{n}$, that transforms $c_{0}$ in a monomial in
$\left(u_{1}^{(l)},\dots,u_{r}^{(l)}\right)$, multiplied by a unit
of $R_{l}$. 

In particular, after this local framed sequence, the element $P$
is of the form $w_{n}^{\alpha}+a_{0}b_{0}\overline{y}$.
\end{lem}

\begin{proof}
We will prove this Lemma in a more general version in Lemma \ref{lem: on se ramene au bon cas-1}.
\end{proof}
\begin{cor}
\label{cor: tout element qui suit est monomialaisable}Let $P$ be
an immediate successor key element of $u_{n}$. Then $P$ is monomializable.
\end{cor}

\begin{proof}
If $u_{n}\ll P$, we use Lemma \ref{lem: on se ramene au bon cas}
to reduce to the case $P=w_{n}^{\alpha}+a_{0}b_{0}\overline{y}.$
By Theorem \ref{thm:monomialise}, we can monomialize $P$. 
\end{proof}
Let $G$ be a local ring essentially of finite type over $k$ of dimension
strictly less than $n$ that is equipped with a valuation centered
on $G$.
\begin{thm}
\label{thm: monomialiase pc lim de u_n} Assume that for every ring
$G$ as above, every element of $G$ is monomializable.

We recall that $\mathrm{car}\left(k_{\nu}\right)=0$. If $u_{n}\ll_{\lim}P$,
then $P$ is monomializable.
\end{thm}

\begin{proof}
We write $P=\sum\limits _{j=0}^{N}b_{j}a_{j}u_{n}^{j}$ the $u_{n}$-expansion
of $P$, with $a_{j}\in R^{\times}$ and $Q=\sum\limits _{j=0}^{N}b_{j}u_{n}^{j}$
a limite immediate successor of $u_{n}$.

By Theorem \ref{thm: delta et pc limite}, we have $\delta_{u_{n}}\left(Q\right)=1$.
Then: 
\[
\nu\left(b_{0}\right)=\nu\left(b_{1}u_{n}\right)<\nu\left(b_{j}u_{n}^{j}\right),
\]
for every $j>1$. 

The elements $a_{i}$ are units of $R$, so for every $j>1$ we have:

\[
\nu\left(a_{0}b_{0}\right)=\nu\left(a_{1}b_{1}u_{n}\right)<\nu\left(a_{j}b_{j}u_{n}^{j}\right).
\]
In fact, $\nu\left(a_{1}b_{1}\right)<\nu\left(a_{0}b_{0}\right)$
and by hypothesis, after a sequence of blow-ups independent of $u_{n}$,
we can monomialize $a_{j}b_{j}$ for every index $j$, and assume
that $a_{1}b_{1}\mid a_{0}b_{0}$ by Proposition \ref{prop:divaleur}.

Then 
\[
\nu\left(b_{0}\right)=\nu\left(b_{1}u_{n}\right)<\nu\left(b_{j}\right)+j\nu\left(u_{n}\right)=\nu\left(b_{j}\right)+j\left(\nu\left(b_{0}\right)-\nu\left(b_{1}\right)\right).
\]

So $\nu\left(b_{0}\right)<\left(b_{j}\right)+j\left(\nu\left(b_{0}\right)-\nu\left(b_{1}\right)\right).$

In fact, $\nu\left(b_{1}^{j}\right)<\nu\left(b_{j}b_{0}^{j-1}\right)$.
So after a sequence of blow-ups independent of $u_{n}$, we have $b_{1}^{j}\mid b_{j}b_{0}^{j-1}$.
After a $n$-Puiseux package $\left(\ast\right)$ $\left(R,u\right)\to\dots\to\left(R',u'\right)$
in the special case $\alpha=1$, we obtain $P=\sum\limits _{j=0}^{N}b'_{j}\left(u'_{n}\right)^{j}$
with $b'_{1}\mid b'_{j}$ for every index $j$ with $u'_{n}=\frac{b_{1}u_{n}}{b_{0}}+1$.

In fact, $\frac{P}{b'_{1}}=u'_{n}+\varphi$ with $\varphi\in\left(u'_{1},\dots,u'_{n-1}\right)$.
So $u'':=\left(u'_{1},\dots,u'_{n-1},\frac{P}{b'_{1}}\right)$ is
a regular system of parameters of $R'$. Then, the sequence $\left(R,u\right)\to\dots\to\left(R',u"\right)$
given by $\left(\ast\right)$ changing uniquely the last parameter
$u'_{n}$ after the last blow-up is still a local framed sequence.
So $P$ is monomializable.
\end{proof}
\begin{rem}
Since $Q_{2}$ is an immediate successor (possibly limit) of $u_{n}$,
this is in particular an immediate successor (possibly limit) key
element of $u_{n}$. By Corollary \ref{cor: tout element qui suit est monomialaisable},
or Theorem \ref{thm: monomialiase pc lim de u_n}, it is monomializable
modulo Lemma \ref{lem: on se ramene au bon cas}.
\end{rem}

\subsection{Generalization.}

Now we monomialized $Q_{2}$, but we want to monomialize every key
polynomial of the sequence $\mathcal{Q}$. Here the key elements will
be useful. Indeed, modified by the blow-ups which monomialized $Q_{2}$,
we cannot know if $Q_{3}$ is still a key polynomial.

To be more general, we will show that if $Q_{i}\in\mathcal{Q}$ is
monomializable, then $Q_{i+1}$ is monomializable.

Assume that the polynomial $Q_{i}$ is monomializable after a sequence
of blow-ups $\left(R,u\right)\to\left(R_{l},u^{(l)}\right)$.

Let $\Delta_{l}$ be the group $\nu\left(k\left(u_{1}^{\left(l\right)},\dots,u_{n-1}^{\left(l\right)}\right)\setminus\left\{ 0\right\} \right)$.
We set 
\[
\alpha_{l}:=\min\left\{ h\text{ such that }h\beta_{n}^{(l)}\in\Delta_{l}\right\} .
\]

We set $X_{j}=\mathrm{in}_{\nu}\left(u_{j}^{(l)}\right)$, $W_{j}=w_{j}^{(l)}$
and $\lambda_{l}$ the minimal polynomial of $X_{n}$ over $\mathrm{gr}_{\nu}k\left(u_{1}^{\left(l\right)},\dots,u_{n-1}^{\left(l\right)}\right)$
of degree $\alpha_{l}$.

Since $k=k_{\nu}$, there exists $c_{0}\in\mathrm{gr}_{\nu}k\left(u_{1}^{\left(l\right)},\dots,u_{n-1}^{\left(l\right)}\right)$
such that 
\[
\lambda_{l}\left(X\right)=X^{\alpha_{l}}+c_{0}.
\]

Furthermore, we have $Q_{i}=\overline{\omega}w_{n}^{(l)}$ with $\overline{\omega}$
a monomial in $W_{1},\dots,W_{r_{l}}$ multiplied by a unit. We set
$\omega:=\mathrm{in}_{\nu}\left(\overline{\omega}\right)$.

We know that $Q_{i+1}$ is an optimal immediate successor of $Q_{i}$,
so we denote by 
\[
Q_{i+1}=Q_{i}^{\alpha_{l}}+b_{0}
\]
 the $Q_{i}$-expansion of $Q_{i+1}$ in $k\left(u_{1},\dots,u_{n-1}\right)\left[u_{n}\right]$
by Proposition \ref{prop: binoooomme} with $c_{0}=\mathrm{in}_{\nu}\left(b_{0}\right)$.

Since $Q_{i}=\overline{\omega}W_{n}$ and $Q_{i+1}=Q_{i}^{\alpha_{l}}+b_{0}$,
we have 
\[
\frac{Q_{i+1}}{\overline{\omega}^{\alpha_{l}}}=\left(u_{n}^{(l)}\right)^{\alpha_{l}}+\frac{b_{0}}{\overline{\omega}^{\alpha_{l}}}.
\]

We know that both terms of the $Q_{i}$-expansion of $Q_{i+1}$ have
same value. So these two terms are divisible by the same power of
$\overline{\omega}$ after a suitable sequence of blow-ups $(\ast_{i})$
independent of $u_{n}^{(l)}$.

We denote by $\widetilde{Q}_{i+1}$ the strict transform of $Q_{i+1}$
by the composition of $\left(\ast_{i}\right)$ with the sequence of
blow-ups $\left(\ast'_{i}\right)$ that monomialize $Q_{i}$. We denote
this composition by $\left(c_{i}\right)$. We write $\left(R,u\right)\overset{\left(c_{i}\right)}{\to}\left(R_{l},u^{(l)}\right)$.

We know that $\widetilde{Q}_{i}$, the strict transform of $Q_{i}$
by $\left(c_{i}\right)$, is a regular parameter of $R_{l}$. Indeed,
by Proposition \ref{prop:chgmtvari}, we know that every $u_{j}$
of $R$ can be written as a monomial in $w_{1}^{(l)},\dots,w_{r_{l}}^{(l)}$.
In fact, the reduced exceptional divisor of this sequence of blow-ups
is exactly $V\left(\overline{\omega}\right)_{\mathrm{red}}$. Then,
since $Q_{i}=W_{n}\overline{\omega}$, we have that the strict transform
of $Q_{i}$ is $\widetilde{Q}_{i}=W_{n}=w_{n}^{(l)}=u_{n}^{(l)}$.
So it is a key polynomial in the extension $k\left(u_{1}^{(l)},\dots,u_{n-1}^{(l)}\right)\left(u_{n}^{(l)}\right)$.

Let us show that $\widetilde{Q}_{i+1}=\frac{Q_{i+1}}{\overline{\omega}^{\alpha_{l}}}$. 

We have 
\[
Q_{i}^{\alpha_{l}}=\overline{\omega}^{\alpha_{l}}\left(u_{n}^{(l)}\right)^{\alpha_{l}}
\]
 and also $u_{n}^{(l)}\nmid\overline{\omega}$. Thus $\overline{\omega}^{\alpha_{l}}$
divides $Q_{i}^{\alpha_{l}}$ and all the non-zero terms of the $Q_{i}$-expansion
of $Q_{i+1}$. Furthermore, it is the greatest power of $\overline{\omega}$
that divides all the terms, so $\frac{Q_{i+1}}{\overline{\omega}^{\alpha_{l}}}$
is $\widetilde{Q}_{i+1}$, the strict transform of $Q_{i+1}$ by the
sequence of blow-ups.

Let $G$ be a local ring essentially of finite type over $k$ of dimension
strictly less than $n$ equipped with a valuation centered in $G$
whose residue field is $k$.
\begin{lem}
\label{lem: on se ramene au bon cas-1} Assume that for every ring
$G$ as above, every element of $G$ is monomializable.

Assume that $Q_{i}<Q_{i+1}$ in $\mathcal{Q}$.

There exists a local framed sequence $\left(R_{l},u^{(l)}\right)\to\left(R_{e},u^{(e)}\right)$
such that in $R_{e}$, the strict transform of $Q_{i+1}$ is of the
form $\left(u_{n}^{\left(e\right)}\right)^{\alpha_{l}}+\tau_{0}\eta$,
where $\tau_{0}\in R_{e}^{\times}$ and $\eta$ is a monomial in $u_{1}^{\left(e\right)},\dots,u_{r_{e}}^{\left(e\right)}$. 
\end{lem}

\begin{proof}
By hypothesis, after a sequence of blow-ups independent of $u_{n}^{(l)}$,
we can monomialize $b_{0}$ and assume that it is a monomial in $\left(u_{1}^{(l)},\dots,u_{n-1}^{(l)}\right)$
multiplied by a unit of $R_{l}$.

For every $g\in\left\{ r_{l}+1,\dots,n-1\right\} $, we do a $g$-Puiseux
package, and then we have a sequence 
\[
\left(R_{l},u^{(l)}\right)\to\left(R_{t},u^{(t)}\right)
\]
 such that every $u_{g}^{(l)}$ is a monomial in $\left(u_{1}^{(t)},\dots,u_{r_{t}}^{(t)}\right)$.

In fact, we can assume that $b_{0}$ is a monomial in $\left(u_{1}^{(l)},\dots,u_{r_{l}}^{(l)}\right)$
multiplied by a unit of $R_{l}$.

Since the strict transform $\widetilde{Q}_{i+1}=\left(u_{n}^{(l)}\right)^{\alpha_{l}}+\frac{b_{0}}{\overline{\omega}^{\alpha_{l}}}$
is an immediate successor key element of $\widetilde{Q}_{i}$. This
completes the proof.
\end{proof}
\begin{rem}
Lemma \ref{lem: on se ramene au bon cas} is a special case of Lemma
\ref{lem: on se ramene au bon cas-1}.
\end{rem}

Let $G$ be a local ring essentially of finite type over $k$ of dimension
strictly less than $n$ equipped with a valuation centered in $G$
whose residue field is $k$.
\begin{thm}
\label{thm:monomialisation de Q3} Assume that for every ring $G$
as above, every element of $G$ is monomializable.

We recall that $\mathrm{car}\left(k_{\nu}\right)=0$. If $Q_{i}$
is monomializable, there exists a local framed sequence 
\begin{equation}
\left(R,u\right)\overset{\pi_{0}}{\to}\left(R_{1},u^{(1)}\right)\overset{\pi_{1}}{\to}\cdots\overset{\pi_{l-1}}{\to}\left(R_{l},u^{(l)}\right)\overset{\pi_{l}}{\to}\cdots\overset{\pi_{m-1}}{\to}\left(R_{m},u^{(m)}\right)\label{eq:suitelocale2}
\end{equation}
that monomializes $Q_{i+1}$.
\end{thm}

\begin{proof}
There are two cases.

First: $Q_{i}<Q_{i+1}$. Then we just saw that the strict transform
$\widetilde{Q}_{i+1}$ of $Q_{i+1}$ by the sequence $\left(R,u\right)\to\left(R_{l},u^{(l)}\right)$
that monomializes $Q_{i}$ is an immediate successor key element of
$\widetilde{Q}_{i}=u_{n}^{(l)}$, and that we can reduce the problem
to the hypotheses of Theorem \ref{thm:monomialise} by Lemma \ref{lem: on se ramene au bon cas-1}.
So we use Theorem \ref{thm:monomialise} replacing $Q_{1}$ by $\widetilde{Q}_{i}$
and $Q_{2}$ by $\widetilde{Q}_{i+1}$.

Then we have constructed a local framed sequence (\ref{eq:suitelocale2})
that monomializes $\widetilde{Q}_{i+1}$.

Second case: $Q_{i}<_{\lim}Q_{i+1}$. 

Then we saw that the strict transform $\widetilde{Q}_{i+1}$ of $Q_{i+1}$
by the sequence $\left(R,u\right)\to\left(R_{l},u^{(l)}\right)$ that
monomialize $Q_{i}$ is a limit immediate successor key element of
$\widetilde{Q}_{i}=u_{n}^{(l)}$. Then we apply Theorem \ref{thm: monomialiase pc lim de u_n}
replacing $Q_{1}$ by $\widetilde{Q}_{i}$ and $Q_{2}$ by $\widetilde{Q}_{i+1}$.

We have constructed a local framed sequence (\ref{eq:suitelocale2})
that monomializes $\widetilde{Q}_{i+1}$. 
\end{proof}
\begin{thm}
\label{thm:generalisation de la monomialisation}There exists a local
sequence 
\begin{equation}
\left(R,u\right)\overset{\pi_{0}}{\to}\cdots\overset{\pi_{s-1}}{\to}\left(R_{s},u^{(s)}\right)\overset{\pi_{s}}{\to}\cdots\label{eq:suitelocaleterminator}
\end{equation}

that monomializes all the key polynomials of $\mathcal{Q}$.

More precisely, for every index $i$, there exists an index $s_{i}$
such that in $R_{s_{i}}$, $Q_{i}$ is a monomial in $u^{\left(s_{i}\right)}$
multiplied by a unit of $R_{s_{i}}$.
\end{thm}

\begin{proof}
Induction on the dimension $n$ and on the index $i$ and we iterate
the previous process.
\end{proof}

\subsection{Divisibility.}

We consider, for every integer $j$, the countable sets
\[
\mathscr{S}_{j}:=\left\{ \prod\limits _{i=1}^{n}\left(u_{i}^{(j)}\right)^{\alpha_{i}^{(j)}}\text{, with }\alpha_{i}^{(j)}\in\mathbb{Z}\right\} 
\]
 and 
\[
\widetilde{\mathscr{S}}_{j}:=\left\{ \left(s_{1},s_{2}\right)\in\mathscr{S}_{j}\times\mathscr{S}_{j}\text{, with }\nu\left(s_{1}\right)\leq\nu\left(s_{2}\right)\right\} 
\]
with the convention that for every $i\in\left\{ 1,\dots,n\right\} $,
$u_{i}^{(0)}=u_{i}$.

The set $\widetilde{\mathscr{S}}_{j}$ being countable for every integer
$j$, we can number its elements, and then we write $\widetilde{\mathscr{S}}_{j}:=\left\{ s_{m}^{(j)}\right\} _{m\in\mathbb{N}}$.
We consider now the finite set 
\[
\mathscr{S}'_{j}:=\left\{ s_{m}^{(j)},\text{ }m\leq j\right\} \cup\left\{ s_{j}^{(m)},\text{ }m\leq j\right\} .
\]

Then $\bigcup\limits _{j\in\mathbb{N}}\left(\mathscr{S}_{j}\times\mathscr{S}_{j}\right)=\bigcup\limits _{j\in\mathbb{N}}\widetilde{\mathscr{S}}_{j}=\bigcup\limits _{j\in\mathbb{N}}\mathscr{S}'_{j}$
is a countable union of finite sets.

Now we fix a local framed sequence
\[
\left(R,u\right)\to\cdots\to\left(R_{i},u^{(i)}\right).
\]

\begin{thm}
\label{thm:existence suite les ensembles finis se divisent}There
exists a finite local framed sequence
\[
p_{i}\colon\left(R_{i},u^{(i)}\right)\to\cdots\to\left(R_{i+q_{i}},u^{(i+q_{i})}\right)
\]
 such that for every integer $j\leq i$ and for every element $s$
of $\mathscr{S}'_{j}$ , the first coordinate of $s$ divides its
second coordinate in $R_{i+q_{i}}$.
\end{thm}

\begin{proof}
Consider an integer $j\leq i$ and an element $s=\left(s_{1},s_{2}\right)\in\mathscr{S}'_{j}$
. We want to construct a sequence of blow-ups such that at the end
we have $s_{1}\mid s_{2}$.

We know that $s\in\widetilde{\mathscr{S}}_{m}$ with $m\leq j$. All
cases being similar, we may assume $s\in\widetilde{\mathscr{S}}_{j}$
and then we have
\[
s_{1}=\prod\limits _{i=1}^{n}\left(u_{i}^{(j)}\right)^{\alpha_{i,1}^{(j)}}
\]
 and 
\[
s_{2}=\prod\limits _{i=1}^{n}\left(u_{i}^{(j)}\right)^{\alpha_{i,2}^{(j)}}.
\]
 By Proposition \ref{prop:existencesuitedivise} applied to $R_{i}$
instead of $R$, there exists a sequence $\left(R_{i},u^{(i)}\right)\to\cdots\to\left(R_{i+l},u^{(i+l)}\right)$
such that in $R_{i+l}$, $s_{1}\mid s_{2}$ or $s_{2}\mid s_{1}$.
By definition $\nu(s_{1})\leq\nu(s_{2})$, so we have $s_{1}\mid s_{2}$
by Proposition \ref{prop:divaleur}.

By point $4$ of Theorem \ref{thm:monomialise}, we know that $\mathscr{S}_{j}\subseteq R_{i+l}^{\times}\mathscr{S}_{i+l}$.
In other words every element of $\mathscr{S}_{j}$ can be written
$z_{i+l}s_{i+l}$ with $z_{i+l}\in R_{i+l}^{\times}$ and $s_{i+l}\in\mathscr{S}_{i+l}$.

Let $\left(s_{3},s_{4}\right)\in\mathscr{S}'_{j}$, be another pair
of $\mathscr{S}'_{j}$, let us say that it is still in $\widetilde{\mathscr{S}}_{j}$.
We just saw that $s_{3},s_{4}\in R_{i+l}^{\times}\mathscr{S}_{i+l}$.
Units don't have an effect on divisibility, so we can only consider
the part of $s_{3}$ and $s_{4}$ which is in $\mathscr{S}_{i+l}$.
Hence we can iterate the Proposition \ref{prop:existencesuitedivise}
applying it to $\left(R_{i+l},u^{(i+l)}\right)$. So we constructed
an other sequence of blow-ups
\[
\left(R_{i+l},u^{(i+l)}\right)\overset{}{\to}\cdots\overset{}{\to}\left(R_{i+h},u^{(i+h)}\right)
\]
 such that $R_{i+h}$ we have $s_{3}\mid s_{4}$ or $s_{4}\mid s_{3}$.
Since $\nu(s_{3})\leq\nu(s_{4})$, we know that $s_{3}$ divides $s_{4}$.

We iterate the process for all the pairs of $\mathscr{S}'_{j}$, and
for every $j\leq i$ . This is a finite number of times since $\mathscr{S}'_{j}$
has a finite number of elements for every $j$ and since we consider
a finite number of such sets. Then we obtain a finite sequence of
blow-ups 
\[
\left(R_{i},u^{(i)}\right)\overset{}{\to}\cdots\overset{}{\to}\left(R_{i+q_{i}},u^{(i+q_{i})}\right)
\]
such that for every integer $j\leq i$ and every $s$ in $\mathscr{S}'_{j}$
, the first coordinate of $s$ divides the second coordinate in $R_{i+q_{i}}$.
\end{proof}
The goal of the next theorem is to construct an infinite local framed
sequence \begin{equation} (R,u)\rightarrow\dots\rightarrow\left(R_i,u^{(i)}\right)\dots\label{locframedsequence} \end{equation}
that monomializes all the key elements, as well as other elements
specified below, and to ensure countably many divisibility conditions,
also specified below. We will use the notation $$ B_i:=k\left[u^{(i)}_1,\dots,u^{(i)}_{n-1}\right]. $$
\begin{thm}
\label{thm: La suite de la mort qui tue} We recall that $\mathrm{car}\left(k_{\nu}\right)=0$.
There exists an infinite sequence of blow-ups 
\begin{equation}
\left(R,u\right)\overset{}{\to}\cdots\overset{}{\to}\left(R_{m},u^{(m)}\right)\overset{}{\to}\cdots\label{eq: la mort qui tue}
\end{equation}
that monomializes all the key polynomials, all the elements of $B_{i}$
for every index $i$ and that has the following property:

\[
\forall j\in\mathbb{N}\text{ }\forall s=\left(s_{1},s_{2}\right)\in\mathscr{S}'_{j}\text{ }\exists i\in\mathbb{N}_{\geq j}\text{ such that in }R_{i}\text{ we have }s_{1}\mid s_{2}.
\]
\end{thm}

\begin{proof}
The first key polynomial is a monomial, so for it we do not need to
do anything. For $j=0$, the elements of $\mathscr{S}'_{j}=\mathscr{S}'_{0}$
are just pairs of monomials in $u$. Let us consider $s=(s_{1},s_{2})\in\mathscr{S}'_{0}$
and apply Proposition \ref{prop:existencesuitedivise}. We construct
a sequence $p_{0}\colon R\to R_{q_{0}}$ such that in $R_{q_{0}}$,
we have $s_{1}\mid s_{2}$ or $s_{2}\mid s_{1}$. Since $\nu\left(s_{1}\right)\leq\nu\left(s_{2}\right)$,
we have $s_{1}\mid s_{2}$. We do the same for all the elements of
$\mathscr{S}'_{0}$ (recall that the set $\mathscr{S}'_{0}$ is finite),
and by abuse of notation we still denote by $p_{0}\colon R\to R_{q_{0}}$
the sequence obtained at the end. Now we have a sequence of blow-ups
$p_{0}\colon R\to R_{q_{0}}$ such that the first key polynomial is
a monomial and such that for every $s=(s_{1},s_{2})\in\mathscr{S}'_{0}$,
we have $s_{1}\mid s_{2}$ in $R_{q_{0}}$. 

We denote by $\left(P_{j}^{(i)}\right)_{j\in\mathbb{N}}$ the sequence
of the generators of the $\nu$-ideals of the $B_{i}$. For the moment
we only monomialize $P_{0}^{(0)}$ and still denote by $p_{0}\colon R\to R_{q_{0}}$
the sequence of blow-ups that monomializes the first key polynomial
$P_{0}^{(0)}$ and such that for every $s=(s_{1},s_{2})\in\mathscr{S}'_{0}$,
we have $s_{1}\mid s_{2}$ in $R_{q_{0}}$. 

Arguing exactly as in the proof of Theorem \ref{thm:generalisation de la monomialisation},
we show that there exists a sequence $\pi^{(2)}\colon R_{q_{0}}\to...\to R_{1}$
that monomializes the second key polynomial.

We have a sequence $\pi^{(2)}\circ p_{0}\colon R\to R_{q_{0}}\to R_{1}$
that monomializes the first two key polynomials, the element $P_{0}^{(0)}$,
and such that for every $s=(s_{1},s_{2})\in\mathscr{S}'_{0}$, we
have $s_{1}\mid s_{2}$ in $R_{q_{0}}$. Now, again by Proposition
\ref{prop:existencesuitedivise}, we construct a sequence $p_{1}\colon R_{1}\to R_{q_{1}}$
such that for every $s=(s_{1},s_{2})\in\mathscr{S}'_{1}$, we have
$s_{1}\mid s_{2}$ in $R_{q_{1}}$. 

Now we monomialize all the $P_{j}^{(i)}$ for $i,j\leq1$ and still
denote, by abuse of notation, by $p_{1}\colon R_{1}\to R_{q_{1}}$
the sequence of blow-ups that monomializes these $P_{j}^{(i)}$ and
such that for every $s=(s_{1},s_{2})\in\mathscr{S}'_{1}$, we have
$s_{1}\mid s_{2}$ in $R_{q_{1}}$. 

Arguing exactly as in the proof of Theorem \ref{thm:generalisation de la monomialisation},
we show that there exists a sequence of blow-ups $\pi^{(3)}\colon R_{q_{1}}\to...\to R_{2}$
that monomializes the third key polynomial.

So we have a sequence $\pi^{(3)}\circ p_{1}\circ\pi^{(2)}\circ p_{0}\colon R\to R_{q_{0}}\to R_{q_{1}}\to R_{2}$
that monomializes the first three key polynomials, the elements $P_{j}^{(i)}$
for $i,j\leq1$, and such that for every $s=(s_{1},s_{2})\in\mathscr{S}'_{0}$
or $\mathscr{S}'_{1}$, we have $s_{1}\mid s_{2}$ in $R_{q_{0}}$
or in $R_{q_{1}}$. Now, again by Proposition \ref{prop:existencesuitedivise},
we construct a sequence $p_{2}\colon R_{2}\to R_{q_{2}}$ such that
for every $s=(s_{1},s_{2})\in\mathscr{S}'_{2}$, we have $s_{1}\mid s_{2}$
in $R_{q_{2}}$. 

Now we monomialize all the $P_{j}^{(i)}$ for $i,j\leq2$ and still
denote, by abuse of notation, by $p_{2}\colon R_{2}\to R_{q_{2}}$
the sequence of blow-ups that monomializes these $P_{j}^{(i)}$ and
such that for every $s=(s_{1},s_{2})\in\mathscr{S}'_{2}$, we have
$s_{1}\mid s_{2}$ in $R_{q_{2}}$. 

Then we have a sequence $p_{2}\circ\pi^{(3)}\circ p_{1}\circ\pi^{(2)}\circ p_{0}$
that monomializes the first three key polynomials, the elements $P_{j}^{(i)}$
for $i,j\leq2$, and such that for every $s=(s_{1},s_{2})\in\mathscr{S}'_{i}$
for $i\in\left\{ 0,1,2\right\} $ we have $s_{1}\mid s_{2}$ in $R_{q_{i}}$.
We iterate this process an infinite number of times. Hence we construct
a sequence of blow-ups $(R,u)\to\cdots\to(R_{m},u^{(m)})\to\cdots$
that monomializes all the key polynomials, all the generators $P_{j}^{(i)}$
(and so all the elements of the $B_{i}$) and that has the last property
of the statement of the Theorem.
\end{proof}
\newpage{}

\section{Conclusion.}

Now we can prove the main result of this chapter, namely, simultaneaous
embedded local uniformization for the local rings essentially of finite
type over a field of characteristic zero.

A local algebra $K$ essentially of finite type over a field $k$
that has $k$ as residue field is an étale extension of 
\[
K'=k\left[u_{1},\dots,u_{n}\right]_{\left(u_{1},\dots,u_{n}\right)}.
\]
 Let $f\in K$ be an irreducible element over $k$ and 
\[
I:=\left(f\right)\bigcap k\left[u_{1},\dots,u_{n}\right].
\]
 The ideal $I$ is a prime ideal of height $1$, so $I$ principal.
We consider a generator $\widetilde{f}$ of $I$. Then $\frac{K'}{\left(\widetilde{f}\right)}\hookrightarrow\frac{K}{\left(f\right)}$
and each local sequence in $\frac{K'}{\left(\widetilde{f}\right)}$
induced a local sequence in $\frac{K}{\left(f\right)}$. 

So it is enough to prove local uniformization in the case of the rings
$k\left[u_{1},\dots,u_{n}\right]_{\left(u_{1},\dots,u_{n}\right)}$
to prove it in the general case of algebras essentially of finite
type over a field $k$.
\begin{thm}
\label{thm:Monomialisons}Let us consider the sequence 
\[
\left(R,u\right)\overset{}{\to}\cdots\overset{}{\to}\left(R_{m},u^{(m)}\right)\overset{}{\to}\cdots
\]
 of Theorem \ref{thm: La suite de la mort qui tue}. 

Then for every element $f$ of $R$, there exists $i$ such that in
$R_{i}$, $f$ is a monomial multiplied by a unit. 
\end{thm}

\begin{proof}
Let $f\in R$. By Theorem \ref{thm:il existe une suite cofinale},
there exists a finite or infinite sequence $\left(Q_{i}\right)_{i}$
of key polynomials of the extension $K\left(u_{n}\right)$, optimal
(possibly limit) immediate successors, such that $\left(\epsilon(Q_{i})\right)_{i}$
is cofinal in $\epsilon(\Lambda)$ where $\Lambda$ is the set of
key polynomials. 

Then by Remark \ref{rem: tout element est non degenere par rapport a un poly clef},
$f$ is non-degenerate with respect to one of these polynomials $Q_{i}$.
But we saw in Theorem \ref{thm: La suite de la mort qui tue} that
there exists an index $l$ such that in $R_{l}$, all the $Q_{j}$
with $j\leq i$ are monomials, hence $f$ is non-degenerate with respect
to a regular system of parameters of $R_{l}$. 

Let $N=\left(w_{1},\dots,w_{s}\right)$ be a monomial ideal in $u^{\left(l\right)}$
such that $\nu\left(N\right)=\nu\left(f\right)$ with $w_{j}$ monomials
in $u^{\left(l\right)}$ such that $\nu\left(w_{1}\right)=\min\left\{ \nu\left(w_{j}\right)\right\} $.
By construction of the local framed sequence, there exists $l'\geq l$
such that in $R_{l'}$, $w_{1}\mid w_{j}$ for all $j$. So in $R_{l'}$,
$f$ is equal to $w_{1}$ multiplied by a unit of $R_{l'}$.
\end{proof}
\begin{thm}[Embedded local uniformization]
\label{thm: pour le citer} Let $k$ be a zero characteristic field
and $f=\left(f_{1},\dots,f_{l}\right)\in k\left[u_{1},\dots,u_{n}\right]^{l}$
be a set of $l$ polynomials in $n$ variables, that are irreducible
over $k$. We set $R:=k\left[u_{1},\dots,u_{n}\right]_{\left(u_{1},\dots,u_{n}\right)}$
and $\nu$ a valuation centered in $R$ such that $k=k_{\nu}$.

We consider the sequence $\left(R,u\right)\overset{}{\to}\cdots\overset{}{\to}\left(R_{m},u^{(m)}\right)\overset{}{\to}\cdots$
of Theorem \ref{thm: La suite de la mort qui tue}.

Then there exists an index $j$ such that the subscheme of $\mathrm{Spec}\left(R_{j}\right)$
defined by the ideal $\left(f_{1},\dots,f_{l}\right)$ is a normal
crossing divisor.
\end{thm}

\begin{proof}
Renumbering, if necessary, we may assume 
\[
\nu\left(f_{1}\right)=\min\left\{ \nu\left(f_{j}\right)\right\} .
\]

By Theorem \ref{thm: La suite de la mort qui tue} there exists an
index $j_{1}$ such that in $R_{j_{1}}$, the total transform of $f_{1}$
is a monomial in $u^{(j_{1})}$, and so defines a normal crossing
divisor.

Now we look at the equation $f_{2}$ in $R_{j_{1}}$. By Theorem \ref{thm: La suite de la mort qui tue},
there exists an index $j_{2}$ such that in $R_{j_{2}}$, the total
transform of $f_{2}$ defines a normal crossing divisor.

In $R_{2}$, the total transforms of $f_{1}$ and $f_{2}$ define
normal crossing divisors.

We iterate the process until the total transforms of $f_{1},\dots,f_{l}$
define normal crossing divisors in $R_{j_{l}}$.

By construction of the local framed sequence $\left(R,u\right)\overset{}{\to}\cdots\overset{}{\to}\left(R_{m},u^{(m)}\right)\overset{}{\to}\cdots$,
there exists $j\geq j_{l}$ such that in $R_{j}$, we have $f_{1}\mid f_{i}$
for every index $i$. 
\end{proof}
\begin{cor}
We keep the same notation and hypotheses as in the previous Theorem.

Then $R_{\nu}=\lim\limits _{\rightarrow}R_{i}$.
\end{cor}

\newpage{}

\part{Simultaneous local uniformization in the case of quasi-excellent
rings for valuations of rank less than or equal to $2$.}

\section{Preliminaries.}

Let $R$ be a local noetherian domain of equicharacteristic zero and
$\nu$ a valuation of $\mathrm{Frac}\left(R\right)$ of rank $1$,
centered in $R$ and of value group $\Gamma_{1}$. We are going to
define the implicit prime ideal $H$ of $R$ for the valuation $\nu$,
which is a key object in local uniformization. Indeed, this ideal
will be the ideal we have to desingularize. We are going to prove
in this part that to regularize $R$, hence to construct a local uniformization,
we only have to regularize $\widehat{R}_{H}$ and $\frac{\widehat{R}}{H}$.
At this point, the hypothesis of quasi excellence is very important:
if $R$ is quasi excellent, the ring $\widehat{R}_{H}$ is regular.
So we will only have to monomialize the elements of $\frac{\widehat{R}}{H}$.

\subsection{Quasi-excellent rings and the implicit prime ideal.}
\begin{defn}
Let $R$ be a domain. We say that $R$ is a \emph{G-ring \index{G-ring}
}if for every prime ideal $\mathfrak{p}$ of $R$, the completion
morphism $R_{\mathfrak{p}}\to\widehat{R}_{\mathfrak{p}}$ is a regular
homomorphism.
\end{defn}

~
\begin{defn}
Let $R$ be a local ring. Then $R$ is quasi-excellent if $R$ is
a G-ring.\emph{ \index{Quasi excellent ring}}

More generally, if $A$ is a ring, then $A$ is quasi-excellent if
$A$ is a local G-ring whose regular locus is open for all $A$-algebra
of finite type.
\end{defn}

\begin{prop}
\cite{M} A local noetherian ring $R$ is quasi-excellent if the completion
morphism $R\to\widehat{R}$ is regular.
\end{prop}

\begin{rem}
Let $R$ be a local ring. If $R$ is a G-ring, then its regular locus
is open. Since the class of $G$-rings is stable under passing to
algebras of finite type, for every $R$-algebra $A$ of finite type,
the set $\mathrm{Reg}(A)$ is open.
\end{rem}

\begin{defn}
We call the implicit prime ideal $H$ of $R$ the ideal $H=\bigcap\limits _{\beta\in\nu\left(R\setminus\left\{ 0\right\} \right)}P_{\beta}\widehat{R}$.
The ideal $H$ is composed of the elements of $\widehat{R}$ whose
value is greater than every element of $\Gamma_{1}$. 

Furthermore, the valuation $\nu$ extends uniquely to a valuation
$\widetilde{\nu}$ centered in $\frac{\widehat{R}}{H}$ (\cite{V}).
\end{defn}

\begin{prop}
\label{prop:localise regulier}Let $R$ be a quasi-excellent local
ring. Then $\widehat{R}_{H}$ is regular.
\end{prop}

\begin{proof}
The ring $R$ is a G-ring. Then for every prime ideal $\mathfrak{p}$
of $R$, we have the injective map $\kappa\left(\mathfrak{p}\right)\hookrightarrow\kappa\left(\mathfrak{p}\right)\otimes_{R}\widehat{R}$
such that the fiber $\kappa\left(\mathfrak{p}\right)\otimes_{R}\widehat{R}$
is geometrically regular over $\kappa\left(\mathfrak{p}\right)$,
where $\kappa\left(\mathfrak{p}\right):=\frac{R_{\mathfrak{p}}}{\mathfrak{p}R_{\mathfrak{p}}}$.
Since $R$ is a domain, $\left(0\right)$ is a prime ideal of $R$.

We write $K:=\mathrm{Frac}\left(R\right)$, then we have the injective
map $K\hookrightarrow K\otimes_{R}\widehat{R}$ such that the fiber
$K\otimes_{R}\widehat{R}$ is geometrically regular over $\mathrm{K}$.
In other words the morphism $K\hookrightarrow K\otimes_{R}\widehat{R}$
is regular.

But $R\setminus\left\{ 0\right\} $ and $\widehat{R}\setminus H$
are two multiplicative subsets of $\widehat{R}$ such that $R\setminus\left\{ 0\right\} \subseteq\widehat{R}\setminus H$,
since $R\cap H=\left\{ 0\right\} $. Then, $\widehat{R}_{H}$ is a
localisation of $\widehat{R}_{R\setminus\left\{ 0\right\} }$. If
we show that $\widehat{R}_{R\setminus\left\{ 0\right\} }$ is regular,
then $\widehat{R}_{H}$ will be also regular as a localization of
a regular ring. By the universal property of tensor product, the ring
$\widehat{R}_{R\setminus\left\{ 0\right\} }$ is isomorphic to $K\otimes_{R}\widehat{R}$,
which is regular by hypothesis. This completes the proof.
\end{proof}

\subsection{Numerical characters associated to a singular local noetherian ring.}

Let $\left(S,\mathfrak{q},L\right)$ be a local noetherian ring and
$\mu$ a valuation centered in $S$. We write $\mu=\mu_{2}\circ\mu_{1}$
with $\mu_{1}$ of rank $1$. The valuation $\mu_{2}$ is trivial
if and only if $\mu$ is also of rank $1$. We denote by $G$ the
value group of $\mu$ and by $G_{1}$ the value group of $\mu_{1}$.
In fact $G_{1}$ is the smallest isolated subgroup non-trivial of
$G$. We set $I:=\left\{ x\in S\text{ such that }\mu(x)\notin G_{1}\right\} $,
and then $\mu_{1}$ induces a valuation of rank $1$ over $\frac{S}{I}$.
Let $\overline{J}$ be the implicit prime ideal of $\frac{\hat{S}}{I\hat{S}}$
for the valuation $\mu_{1}$ and $J$ its preimage in $\hat{S}$. 
\begin{defn}
We set 
\[
e\left(S,\mu\right):=\mathrm{emb.dim}\left(\frac{\hat{S}}{J}\right).
\]
\end{defn}

We assume that $I\subseteq\mathfrak{q}^{2}$. Let $v=\left(v_{1},\dots,v_{n}\right)$
be a minimal set of generators of $\mathfrak{q}$. We have $\mu\left(v_{j}\right)\in G_{1}$
for every index $j$.
\begin{defn}
We have $\sum\limits _{j=1}^{n}\mathbb{Q}\mu\left(v_{j}\right)\subseteq G_{1}\otimes\mathbb{Q}$
and we set 
\[
r\left(S,v,\mu\right):=\dim_{\mathbb{Q}}\left(\sum\limits _{j=1}^{n}\mathbb{Q}\mu\left(v_{j}\right)\right).
\]
\end{defn}

\begin{rem}
We have $r\left(S,v,\mu\right)\leq e\left(S,\mu\right)$.
\end{rem}

Now we consider $M\subset\left\{ 1,\dots,n\right\} $ and 
\[
\left(S,v\right)\to\left(S_{1},v^{(1)}=\left(v_{1}^{(1)},\dots,v_{n_{1}}^{(1)}\right)\right)
\]
 a framed blow-up along $\left(v_{M}\right)$. We set $C'=\left\{ 1,\dots,n_{1}\right\} \setminus D_{1}$,
where $D_{1}$ is as in \ref{def:=0000E9clatement encadr=0000E9}.

If the elements of $v_{M}$ are $L$-linearly independent in $\frac{\mathfrak{q}\hat{S}}{J+\mathfrak{q}^{2}\hat{S}}$,
then there exists a partition of $A$ that we denote by $A'\sqcup A"$.
This partition is such that $v_{M}\cup v_{A'}$ are $L$-linearly
independent modulo $J+\mathfrak{q}^{2}\hat{S}$ and $v_{A"}$ is in
the space generated by $v_{J}\cup v_{A'}$ over $L$ modulo $J+\mathfrak{q}^{2}\hat{S}$.
As we know that $v'_{A\cup B\cup\left\{ j\right\} }=v_{D_{1}}^{(1)}$,
we can identify $A'\cup B\cup\left\{ j\right\} $ with a subset of
$D_{1}$.

Now we set $I_{1}:=\left\{ x\in S_{1}\text{ such that }\mu(x)\notin G_{1}\right\} $
and we consider $\overline{J}_{1}$ the implicit prime ideal of $\frac{\hat{S}_{1}}{I_{1}\hat{S}_{1}}$
with respect to $\mu_{1}$ and $J_{1}$ its preimage in $\hat{S}_{1}$.
We call $\mathfrak{q}_{1}$ the maximal ideal of $S_{1}$ and $L_{1}$
its residue field.
\begin{rem}
\label{rem:e=00003Dn implique}We have $e\left(S,\mu\right)=n$ if
and only if the elements of $v$ are $L$-linarly independent in $\frac{\mathfrak{q}\hat{S}}{J+\mathfrak{q}^{2}\hat{S}}$.
\end{rem}

\begin{thm}
\label{thm:decroissance de la dimension plongee}If $e\left(S,\mu\right)=n$,
then:

\[
e\left(S_{1},\mu\right)\leq e\left(S,\mu\right).
\]

This inequality is strict once the elements of $v_{A'\cup B\cup\left\{ j\right\} \cup C'}^{(1)}$
are $L_{1}$-linearly dependent in $\frac{\mathfrak{q_{1}}\hat{S}_{1}}{J_{1}+\mathfrak{q}_{1}^{2}\hat{S}_{1}}$.
\end{thm}

\begin{proof}
By definition, $v^{(1)}$ generates the maximal ideal $\mathfrak{q}_{1}$
of $S_{1}$, and so induces a set of generators of $\mathfrak{q}_{1}\frac{\widehat{S_{1}}}{J_{1}}$.
Since $n_{1}\leq n$, by definition of a framed blow-up, we know that
$\sharp C'\leq\sharp C$.

Furthermore, we have $e\left(S,\mu\right)=\sharp M+\sharp A'$. We
also know that $v_{D_{1}\setminus\left(A'\cup B\cup\left\{ j\right\} \right)}^{(1)}$
is in the $L$-vector space of $v_{A'\cup B\cup\left\{ j\right\} \cup C'}^{(1)}$
modulo $J_{1}+\mathfrak{q}_{1}^{2}\hat{S}_{1}$.

So:
\[
\begin{array}{ccc}
e\left(S_{1},\mu\right) & \leq & \sharp A'+\sharp B+\sharp\left\{ j\right\} +\sharp C'\\
 & \leq & \sharp A'+\sharp B+1+\sharp C\\
 & = & \sharp A'+\sharp M\\
 & = & e\left(S,\mu\right).
\end{array}
\]

If in addition the elements of $v_{A'\cup B\cup\left\{ j\right\} \cup C'}^{(1)}$
are $L_{1}$-linearly dependents in $\frac{\mathfrak{q_{1}}\hat{S}_{1}}{J_{1}+\mathfrak{q}_{1}^{2}\hat{S}_{1}}$,
then we have $e\left(S_{1},\mu\right)<\sharp A'+\sharp B+\sharp\left\{ j\right\} +\sharp C'$
and so $e\left(S_{1},\mu\right)<e\left(S,\mu\right)$.
\end{proof}
\begin{thm}
\label{thm: r croit}We have $r\left(S_{1},v^{(1)},\mu\right)\geq r\left(S,v,\mu\right)$.
\end{thm}

\begin{proof}
This is induced by the two last points of Proposition \ref{prop:chgmtvari}.
\end{proof}
\begin{cor}
Once $e\left(S,\mu\right)=n$, we have 
\[
\left(e\left(S_{1},\mu\right),e\left(S_{1},\mu\right)-r\left(S_{1},v^{(1)},\mu\right)\right)\leq\left(e\left(S,\mu\right),e\left(S,\mu\right)-r\left(S,v,\mu\right)\right).
\]
 The inequality is strict if $e\left(S_{1},\mu\right)<n$.
\end{cor}

\begin{rem}
We are doing an induction on the dimension $n$. We saw that this
dimension decreases by the sequence of blow-ups.

If it decreases strictly, then it will happen a finite number of time
and the proof is finished.

Then, after now, we assume this dimension to be constant by blow-up.
In other words for all framed sequence $S\to S_{1}$, we assume that
$e\left(S,\mu\right)=e\left(S_{1},\mu\right)=n$.

Similarly, we may assume that $r\left(S,v,\mu\right)=r\left(S_{1},v^{\left(1\right)},\mu\right)$.
\end{rem}

\newpage{}

\section{\label{sec:Id=0000E9al-implicite.}Implicit ideal.}

Let $\left(R,\mathfrak{m},k\right)$ be a local quasi excellent ring
equicharacteristic and let $\nu$ be a valuation of rank $1$ of its
field of fractions, centered in $R$ and of value group $\Gamma_{1}$.
We denote by $H$ the implicit prime ideal of $R$ for the valuation
$\nu$.

By the Cohen structure Theorem, there exists an epimorphism $\Phi$
from a complete regular local ring $A\simeq k\left[\left[u_{1},\dots,u_{n}\right]\right]$
of field of fractions $K$ into $\frac{\widehat{R}}{H}$. Its kernel
$I$ is a prime ideal of $A$. 

We consider $\mu$ a monomial valuation with respect to a regular
system of parameters of $A_{I}$. It is a valuation on $A$ centered
in $I$ such that $k_{\mu}=\kappa\left(I\right)$ where $\kappa\left(I\right)$
is the residue field of $I$. Then we set $\widehat{\nu}:=\widetilde{\nu}\circ\mu$,
hence we define a valuation on $A$. Let $\Gamma$ be the value group
of $\widehat{\nu}$.

Then, $\Gamma_{1}$ is the smallest non-trivial isolated subgroup
of $\Gamma$ and we have:

\[
\begin{array}{ccc}
I & = & \left\{ f\in A\text{ such that }\widehat{\nu}\left(f\right)\notin\Gamma_{1}\right\} .\end{array}
\]
 
\begin{defn}
Let $\pi\colon\left(A,u\right)\to\left(A',u'\right)$ be a framed
blow-up and $\sigma\colon A'\to\widehat{A'}$ be the formal completion
of $A'$. The composition $\sigma\circ\pi$ is called \emph{formal
framed blow-up}.

A composition of such blow-ups is called a \emph{formal framed sequence}.
\end{defn}

Let $\left(A,u\right)\to\left(A_{1},u^{(1)}\right)\to\cdots\to\left(A_{l},u^{(l)}\right)$
a formal sequence, that we denote by $(\ast)$.
\begin{defn}
The formal sequence $\left(A,u\right)\to\left(A_{1},u^{(1)}\right)\to\cdots\to\left(A_{l},u^{(l)}\right)$
is said \emph{defined on $\Gamma_{1}$ }if for every integers $i\in\left\{ 0,\dots,l-1\right\} $
and $q\in J_{i}$, we have $\nu\left(u_{q}^{(i)}\right)\in\Gamma_{1}$.
\end{defn}

Now we consider $A_{i}\simeq k_{i}\left[\left[u_{1}^{(i)},\dots,u_{n}^{(i)}\right]\right]$
and we denote by $I_{i}^{\mathrm{strict}}$ the strict transform of
$I$ in $A_{i}$.
\begin{defn}
We call \emph{formal transformed} of $I$ in $A_{i}$, and we denote
it by $I_{i}$, the preimage in $A_{i}$ of the implicit ideal of
$\frac{A_{i}}{I_{i}^{\mathrm{strict}}}$.
\end{defn}

Let $v_{i}$ be the greatest integer of $\left\{ r,\dots,n\right\} $
such that 
\[
I_{i}\cap k_{i}\left[\left[u_{1}^{(i)},\dots,u_{v_{i}}^{(i)}\right]\right]=\left(0\right)
\]
 and we set 
\[
B_{i}:=k_{i}\left[\left[u_{1}^{(i)},\dots,u_{v_{i}}^{(i)}\right]\right].
\]

\begin{defn}
Let $P$ be a prime ideal of $A$. We call \emph{$\ell$-th symbolic
power of} $P$ the ideal $P^{(\ell)}:=\left(P^{\ell}A_{P}\right)\cap A$.

Equivalently, we have $P^{(\ell)}=\left\{ x\in A\text{ such that }\exists y\in A\setminus P\text{ such that }xy\in P^{\ell}\right\} $.

It is the set composed by the elements that vanish with order at least
$\ell$ in the generic point of $\mathrm{V}\left(P\right)$.

Let $G$ be a complete ring of dimension strictly less than $n$ and
let $\theta$ be a valuation centered in $G$, of value group $\widetilde{\Gamma}$.

We consider $\widetilde{\Gamma}_{1}$ the first non trivial isolated
subgroup of $\widetilde{\Gamma}$ and $\mathfrak{g}:=\left\{ g\in G\text{ such that }\theta\left(g\right)\notin\widetilde{\Gamma}_{1}\right\} $.

The next result will help us to prove the simultaneous local uniformization
by induction.
\end{defn}

\begin{prop}
\label{prop:monomialisation en dehors du carre}Assume that:
\begin{enumerate}
\item In the formal sequence $\left(A,u\right)\to\left(A_{1},u^{(1)}\right)\to\cdots\to\left(A_{l},u^{(l)}\right)$,
there exists a formal framed subsequence 
\[
\pi\colon\left(A,u\right)\to\left(A_{i},u^{(i)}\right)
\]
 such that $v_{i}<n-1$. 
\item For every ring $G$ as above, every element in $G\setminus\mathfrak{g}^{(2)}$
is monomializable by a formal framed sequence defined on $\widetilde{\Gamma}_{1}$.
\end{enumerate}
Then for every element $f$ of $A\setminus I^{(2)}$, there exists
a formal sequence 
\[
\left(A,u\right)\to\cdots\to\left(A_{l},u^{(l)}\right)
\]
 defined over $\Gamma_{1}$ such that $f$ can be written as a monomial
in $u_{1}^{(l)},\dots,u_{n}^{(l)}$ multiplied by an element of $A_{l}^{\times}$.
\end{prop}

\begin{proof}
We assume that there exists a formal framed sequence 
\[
\pi\colon\left(A,u\right)\to\left(A_{i},u^{(i)}\right)
\]
 such that $v_{i}<n-1$. It means that $v_{i}+1<n$. By definition
of $v_{i}$, we know that $\mathfrak{g}_{i}:=I_{i}\cap k_{i}\left[\left[u_{1}^{(i)},\dots,u_{v_{i}+1}^{(i)}\right]\right]\neq\left(0\right)$.
So we consider an element $g$ in $\mathfrak{g}_{i}\setminus\mathfrak{g}_{i}^{(2)}\subseteq C_{i}\setminus\mathfrak{g}_{i}^{(2)}$,
where $C_{i}:=k_{i}\left[\left[u_{1}^{(i)},\dots,u_{v_{i}+1}^{(i)}\right]\right]$.
Since $v_{i}+1<n$, the ring $C_{i}$ is of dimension strictly less
than $n$. So we can use the second hypothesis on the element $g$
in the ring $C_{i}$.

Hence there exists a formal sequence defined over $\Gamma_{1}$ 
\[
\left(C_{i},\left(u_{1}^{(i)},\dots,u_{v_{i}+1}^{(i)}\right)\right)\to\cdots\to\left(S',\left(u'_{1},\dots,u'_{v'}\right)\right)
\]
where $v'\leq v_{i}+1$, and such that $g$ can be written as a monomial
in $u'_{1},\dots,u'_{v'}$ multiplied by an element of $S'^{\times}$. 

Since $g\in\mathfrak{g}_{i}$, there exists a regular parameter of
$S'$, say $u'_{v'}$, such that $\nu\left(u'_{v'}\right)\notin\Gamma_{1}$.
Indeed, $g\in\mathfrak{g}_{i}=I_{i}\cap C_{i}$, so $g\in I_{i}$
hence it belongs to $I$. Equivalently, it satisfies $\widehat{\nu}(g)\notin\Gamma_{1}$.
Since $g$ can be written as a monomial in the generators of the maximal
ideal of $S'$, one of these generators which appears in the factorization
of $g$ must be in $I$. Hence $e\left(S',\widehat{\nu}_{|_{S'}}\right)<v_{i}+1$. 

Replacing every ring $O$ which appears in 
\[
\left(C_{i},\left(u_{1}^{(i)},\dots,u_{v_{i}+1}^{(i)}\right)\right)\to\cdots\to\left(S',\left(u'_{1},\dots,u'_{v'}\right)\right)
\]
 by $O\left[\left[u_{v_{i}+2}^{(i)},\dots,u_{n}^{(i)}\right]\right]$,
we obtain a formal sequence 
\[
\pi'\colon\left(A_{i},u^{(i)}\right)\to\cdots\to\left(A_{l},u^{(l)}\right)
\]
 independent of $u_{v_{i}+2}^{(i)},\dots,u_{n}^{(i)}$, with $A_{l}=S'\left[\left[u_{v_{i}+2}^{(i)},\dots,u_{n}^{(i)}\right]\right]$.
But we know that $e\left(S',\widehat{\nu}_{|_{S'}}\right)<v_{i}+1$,
and so $e\left(A_{l},\widehat{\nu}\right)<n$. 

Let $f$ be an element of $A\setminus I^{(2)}$. Its image under $\pi'\circ\pi$
is an element of $A_{l}$, whose dimension is strictly less than $n$.
Since all the $A_{i}$ are quasi-excellent, we have $f\notin A_{i}\setminus I_{i}^{\left(2\right)}$
and we can use again the second hypothesis. Hence we constructed a
formal sequence $\pi'\circ\pi$ such that $f$ can be written as a
monomial in the generators of the maximal ideal of $A_{l}$ multiplied
by a unit of $A_{l}$. This completes the proof.
\end{proof}
Now, we assume that for every formal sequence $\left(A,u\right)\to\left(A_{1},u^{(1)}\right)\to\cdots\to\left(A_{l},u^{(l)}\right)$
and for every integer $i$, we have $v_{i}\in\left\{ n-1,n\right\} $.

So for every integer $i$, we have $I_{i}\cap k_{i}\left[\left[u_{1}^{(i)},\dots,u_{n-1}^{(i)}\right]\right]=(0)$.

We consider a complete local ring $G$ of dimension strictly less
than $n$ and a valuation $\theta$ of rank $1$ centered in $G$.
\begin{lem}
\label{lem:L'id=0000E9al est de hauteur au plus 1} Assume that for
every ring $G$ as above, there exists a formal framed sequence that
monomializes every element of $G$.

Then $I$ is of height at most $1$.
\end{lem}

\begin{proof}
If $I=(0)$, the proof is finished. So we assume $I\neq(0)$ and we
consider $f\in I\setminus\left\{ 0\right\} $. We write
\[
f=\sum\limits _{j=0}^{\infty}a_{j}u_{n}^{j}
\]
 with $a_{j}\in k\left[\left[u_{1},\dots,u_{n-1}\right]\right]$.
We consider an integer $N$ big enough such that every $a_{j}$ with
$j>N$ is in the ideal generated by $\left(a_{0},\dots,a_{N}\right)$.
Now let us consider 
\[
\delta:=\min\left\{ j\in\left\{ 0,\dots,N\right\} \text{ such that }\nu\left(a_{j}\right)=\min\limits _{0\leq s\leq N}\left\{ \nu\left(a_{s}\right)\right\} \right\} .
\]

We set $\overline{u}:=\left(u_{1},\dots,u_{n-1}\right)$ and $B:=k\left[\left[\overline{u}\right]\right]$.
Since $B$ is a complete local ring of dimension strictly less than
$n$, by hypothesis we can construct a formal sequence $\left(B,\overline{u}\right)\to\left(B',\overline{u}'\right)$
such that for every $j\in\left\{ 0,\dots,N\right\} $, the element
$a_{j}$ is a monomial in $\overline{u}'$. By Propositions \ref{prop:existencesuitedivise}
and \ref{prop:divaleur}, we can construct a local framed sequence
$\left(B',\overline{u}'\right)\to\left(B",\overline{u}"\right)$ such
that $a_{\delta}\mid a_{j}$ for every $j\in\left\{ 0,\dots,N\right\} $
in $B"$, since $a_{\delta}$ has minimal value. So we have a sequence
\[
\left(B,\overline{u}\right)\to\left(B',\overline{u}'\right)\to\left(B",\overline{u}"\right).
\]
 We compose with the formal completion and obtain 
\[
\left(B,\overline{u}\right)\to\left(\widehat{B"},\overline{u}"\right)
\]
 in which we still have $a_{\delta}\mid a_{j}$ for every $j\in\left\{ 0,\dots,N\right\} $. 

We replace again all the rings $O$ of the sequence $\left(B,\overline{u}\right)\to\left(\widehat{B"},\overline{u}"\right)$
by $O\left[\left[u_{n}\right]\right]$, and obtain a sequence $\left(A,u\right)\to\left(A',u'\right)$
independent of $u_{n}$ and in which we still have $a_{\delta}\mid a_{j}$
for every $j\in\left\{ 0,\dots,N\right\} $.

We recall that for every index $i$, we have
\[
I_{i}\cap k_{i}\left[\left[u_{1}^{(i)},\dots,u_{n-1}^{(i)}\right]\right]=(0).
\]
 If we denote by $I'$ the formal transform of $I$ in $A'$, we obtain
$I'\cap\widehat{B"}=(0)$. We know that $\frac{f}{a_{\delta}}\in I'$,
and by Weierstrass preparation Theorem, $\frac{f}{a_{\delta}}=xy$
where $x$ is a unit of $A'$, and $y$ is a monic polynomial in $u_{n}$
of degree $\delta$. Then the morphism $\widehat{B"}\to\frac{A'}{I'}$
is injective and finite. 

Hence $\dim\left(\frac{A'}{I'}\right)=\dim\left(\widehat{B"}\right)=n-1$.
Since $\dim\left(A'\right)=n$, we have $\mathrm{ht}\left(I\right)\leq\mathrm{ht}\left(I'\right)=\dim\left(A'\right)-\dim\left(\frac{A'}{I'}\right)=n-(n-1)=1$.
This completes the proof.
\end{proof}
\begin{cor}[of Lemma \ref{lem:L'id=0000E9al est de hauteur au plus 1}.]
\label{cor: h est unitaire}We keep the same hypothesis as in Lemma
\ref{lem:L'id=0000E9al est de hauteur au plus 1}. Let $I=\left(h\right)$. 

There exists a formal framed sequence $\left(A,u\right)\to\left(A',u'\right)$
such that in $A'$, the strict transform of $h$ is a monic polynomial
of degree $\delta$.
\end{cor}

From now on, we assume that $h$ is a monic polynomial of degree $\delta$.
\begin{prop}
\label{prop:implicite et pc}We keep the same hypothesis as in Lemma
\ref{lem:L'id=0000E9al est de hauteur au plus 1}. Let $I=\left(h\right)$.
The polynomial $h$ is a key polynomial.
\end{prop}

\begin{proof}
By definition, $I=\left\{ f\in A\text{ such that }\widehat{\nu}\left(f\right)\notin\Gamma_{1}\right\} $,
so $\widehat{\nu}\left(h\right)\notin\Gamma_{1}$ . Furthermore, for
every non-zero integer $b$, we have $\widehat{\nu}\left(\partial_{b}h\right)\in\Gamma_{1}$
since $h$ is a generator of $I$, hence has the smallest degree among
all the elements of $I$ and so $\partial_{b}h\notin I$.

Then $\epsilon\left(h\right)\notin\Gamma_{1}$.

Let $P$ be a polynomial such that $\deg\left(P\right)<\deg\left(h\right)$.
To show that $h$ is a key polynomial, it remains to prove that $\epsilon\left(P\right)<\epsilon\left(h\right)$.

By the minimality of $\deg\left(h\right)$, we still have $P\notin I$
and so $\widehat{\nu}\left(P\right)\in\Gamma_{1}$. So for every non-zero
integer $b$, we also have $\widehat{\nu}\left(\partial_{b}P\right)\in\Gamma_{1}$.
Then $\epsilon\left(P\right)\in\Gamma_{1}$.

Assume, aiming for contradiction, that $\epsilon\left(P\right)\ge\epsilon\left(h\right)$. 

Then $-\epsilon\left(P\right)\leq\epsilon\left(h\right)\leq\epsilon\left(P\right)$
and since $\Gamma_{1}$ is an isolated subgroup, $\Gamma_{1}$ is
a segment and so $\epsilon\left(h\right)\in\Gamma_{1}$. Contradiction.

Hence, $\epsilon\left(P\right)<\epsilon\left(h\right)$ and $h$ is
a key polynomial.
\end{proof}
Now we are going to monomialize the key polynomial $h$.

As in the previous part, we construct a sequence $(Q_{i})_{i\ge1}$
of key polynomials such that for each $i$ the polynomial $Q_{i+1}$
is either an optimal or a limit immediate successor of $Q_{i}$ that
begins with $x$ and ends with $h$. So since $\epsilon\left(h\right)$
is maximal in $\epsilon\left(\Lambda\right)$, we stop. Then we have
a finite sequence $(Q_{i})_{i\ge1}$ of key polynomials such that
for each $i$ the polynomial $Q_{i+1}$ is either an optimal or a
limit immediate successor of $Q_{i}$ that begins with $x$ and ends
with $h$.

In the case $I=\left(0\right)$, we construct again a sequence $(Q_{i})_{i\ge1}$
of key polynomials such that for each $i$ the polynomial $Q_{i+1}$
is either an optimal or a limit immediate successor of $Q_{i}$ such
that $\epsilon\left(\mathcal{Q}\right)$ is cofinal in $\epsilon\left(\Lambda\right)$.

Since we don't assume $k=k_{\nu}$ in this part, we need a generalization
of the monomialization Theorems of the Part $3$, paragraph $7$.

\newpage{}

\section{Monomialization of key polynomials.}

Here we consider the ring $A\simeq k\left[\left[u_{1},\dots,u_{n}\right]\right]$
and a valuation $\nu$ centered in $A$ of value group $\Gamma$.
For more clarity, we recall some previous notation.

Let $r$ be the dimension of $\sum\limits _{i=1}^{n}\mathbb{Q}\nu(u_{i})$
in $\Gamma\otimes_{\mathbb{Z}}\mathbb{Q}$. Renumbering if necessary,
we may assume that $\nu\left(u_{1}\right),\dots,\nu\left(u_{r}\right)$
are rationaly independent and we consider $\Delta$ the subgroup of
$\Gamma$ generated by $\nu(u_{1}),\ldots,\nu(u_{r})$. 

We set $E:=\left\{ 1,\dots,r,n\right\} $ and
\[
\overline{\alpha^{(0)}}:=\min\limits _{\alpha\in\mathbb{N}^{\ast}}\left\{ \alpha\text{ such that }\alpha\nu(u_{n})\in\Delta\right\} .
\]
 So $\overline{\alpha^{(0)}}\nu(u_{n})=\sum\limits _{j=1}^{r}\alpha_{j}^{(0)}\nu(u_{j})$
with 
\[
\alpha_{1}^{(0)},\ldots,\alpha_{s}^{(0)}\geq0
\]
 and 
\[
\alpha_{s+1}^{(0)},\ldots,\alpha_{r}^{(0)}<0.
\]

We set
\[
w=(w_{1},\ldots,w_{r},w_{n})=(u_{1},\ldots,u_{r},u_{n})
\]
 and 
\[
v=(v_{1},\ldots,v_{t})=(u_{r+1},\ldots,u_{n-1}),
\]
 with $t=n-r-1$.

We write $x_{i}=\mathrm{in}_{\nu}u_{i}$, and so $x_{1},\ldots,x_{r}$
are algebraically independent over $k$ in $G_{\nu}$. Let $\lambda_{0}$
be the minimal polynomial of $x_{n}$ over $k[x_{1},\ldots,x_{r}]$,
of degree $\alpha$. If $x_{n}$ is transcental, we set $\lambda_{0}:=0$.

We consider

\[
y=\prod\limits _{j=1}^{r}x_{j}^{\alpha_{j}^{(0)}},
\]
 
\[
\overline{y}=\prod\limits _{j=1}^{r}w_{j}^{\alpha_{j}^{(0)}},
\]
 
\[
z=\frac{x_{n}^{\overline{\alpha^{(0)}}}}{y}
\]
 and 
\[
\overline{z}=\frac{w_{n}^{\overline{\alpha^{(0)}}}}{\overline{y}}.
\]

Let $d_{0}:=\frac{\alpha}{\overline{\alpha^{(0)}}}\in\mathbb{N}$. 

If $\lambda_{0}\neq0$, we have 
\[
\lambda_{0}=\sum\limits _{q=0}^{d_{0}}c_{q}y^{d_{0}-q}X^{q\overline{\alpha^{(0)}}}
\]
 where $c_{q}\in k$ , $c_{d}=1$ and $\sum\limits _{q=0}^{d_{0}}c_{q}Z^{q}$
is the minimal polynomial of $z$ over $G_{\nu}$.

We are going to show that there exists a formal framed sequence that
monomializes all the $Q_{i}$. We have $Q_{1}=u_{n}$ so we have to
begin by monomializing $Q_{2}$.

First, let us consider
\[
Q=\sum\limits _{q=0}^{d_{0}}a_{q}b_{q}\overline{y}^{d_{0}-q}w_{n}^{q\overline{\alpha^{(0)}}}
\]
 where $b_{q}\in R$ such that $b_{q}\equiv c_{q}$ modulo $\mathfrak{m}$
and $a_{q}\in A^{\times}$. 

Then we will show that we can reduce the problem to this special case.

Let 
\[
\gamma=(\gamma_{1},\ldots,\gamma_{r},\gamma_{n})=(\alpha_{1}^{(0)},\ldots,\alpha_{s}^{(0)},0,\ldots,0)
\]
 and 
\[
\delta=(\delta_{1},\ldots,\delta_{r},\delta_{n})=(0,\ldots,0,-\alpha_{s+1}^{(0)},\ldots,-\alpha_{r}^{(0)},\overline{\alpha^{(0)}}).
\]

We have 
\[
w^{\delta}=w_{n}^{\delta_{n}}\prod\limits _{j=1}^{r}w_{j}^{\delta_{j}}=\frac{w_{n}^{\overline{\alpha^{(0)}}}}{\prod\limits _{j=s+1}^{r}w_{j}^{\alpha_{j}^{(0)}}}
\]
 and 
\[
w^{\gamma}=\prod\limits _{j=1}^{s}w_{j}^{\alpha_{j}^{(0)}}.
\]
 So $\frac{w^{\delta}}{w^{\gamma}}=\frac{w_{n}^{\overline{\alpha^{(0)}}}}{\prod\limits _{j=1}^{r}w_{j}^{\alpha_{j}^{(0)}}}=\overline{z}$.

Let us compute the value of $w^{\delta}$.

\[
\begin{array}{ccc}
\nu(w^{\delta}) & = & \overline{\alpha^{(0)}}\nu(w_{n})-\sum\limits _{j=s+1}^{r}\alpha_{j}^{(0)}\nu(w_{j})\\
 & = & \overline{\alpha^{(0)}}\nu(u_{n})-\sum\limits _{j=s+1}^{r}\alpha_{j}^{(0)}\nu(u_{j})\\
 & = & \sum\limits _{j=1}^{r}\alpha_{j}^{(0)}\nu(u_{j})-\sum\limits _{j=s+1}^{r}\alpha_{j}^{(0)}\nu(u_{j})\\
 & = & \sum\limits _{j=1}^{s}\alpha_{j}^{(0)}\nu(u_{j})\\
 & = & \sum\limits _{j=1}^{s}\alpha_{j}^{(0)}\nu(w_{j})\\
 & = & \nu(w^{\gamma}).
\end{array}
\]

\begin{thm}
\label{thm:monomialise general}There exists a local framed sequence 

\begin{equation}
\left(A,u\right)\overset{\pi_{0}}{\to}\left(A_{1},u^{(1)}\right)\overset{\pi_{1}}{\to}\cdots\overset{\pi_{l-1}}{\to}\left(A_{l},u^{(l)}\right)\label{eq:suitelocale-1}
\end{equation}

with respect to $\nu$, independent of $v$, that has the following
properties:

For every integer $i\in\{1,\dots,l\}$, we write $u^{(i)}=\left(u_{1}^{(i)},\dots,u_{n_{i}}^{(i)}\right)$
and denote by $k_{i}$ the residue field of $A_{i}$.
\begin{enumerate}
\item The blow-ups $\pi_{0},\dots,\pi_{l-2}$ are monomial.
\item We have $\overline{z}\in A_{l}^{\times}$.
\item We have
\[
n_{l}=\begin{cases}
n & \text{if }\lambda_{0}\neq0\\
n-1 & \text{otherwise}.
\end{cases}
\]
\item We set
\[
u^{(l)}=\begin{cases}
\left(w_{1}^{(l)},\dots,w_{r}^{(l)},v,w_{n}^{(l)}\right) & \text{ if }\lambda_{0}\neq0\\
\left(w_{1}^{(l)},\dots,w_{r}^{(l)},v\right) & \text{otherwise}.
\end{cases}
\]
For every integer $j\in\left\{ 1,\dots,r,n\right\} $, $w_{j}$ is
a monomial in $w_{1}^{(l)},\dots,w_{r}^{(l)}$ multiplied by an element
of $A_{l}^{\times}$. And for every integer $j\in\left\{ 1,\dots,r\right\} $,
$w_{j}^{(l)}=w^{\eta}$ where $\eta\in\mathbb{Z}^{r+1}$.
\item If $\lambda_{0}\neq0$, then $Q=w_{n}^{(l)}\times\overline{y}^{d_{0}}$.
\end{enumerate}
\end{thm}

\begin{proof}
We apply Proposition \ref{prop:existencesuitedivise} to $(w^{\delta},w^{\gamma})$
and obtain a local framed sequence for $\nu$, independent of $v$,
such that $w^{\gamma}\mid w^{\delta}$ in $A_{l}$. 

By Proposition \ref{prop:divaleur} and the fact that $w^{\delta}$
and $w^{\gamma}$ have same value, we have $w^{\delta}\mid w^{\gamma}$
in $R_{l}$. In fact $\overline{z},\overline{z}^{-1}\in A_{l}^{\times}.$
So we have the point $(2)$.

We choose the sequence to be minimal, it means that the sequence composed
by $\pi_{0},\ldots,\pi_{l-2}$ does not satisfy the conclusion of
Proposition \ref{prop:existencesuitedivise} for $(w^{\delta},w^{\gamma})$.
We are now going to show that this sequence satisfies the conclusion
of Theorem \ref{thm:monomialise general}. Let $i\in\left\{ 0,\dots,l\right\} $.
We write $w^{(i)}=\left(w_{1}^{(i)},\dots,w_{r_{i}}^{(i)},w_{n_{i}}^{(i)}\right)$,
with $r_{i}=n_{i}-t-1>0$. For every integers $i\in\left\{ 1,\dots,l\right\} $
and $j\in\left\{ 1,\dots,n_{i}\right\} $, we write $\beta_{j}^{(i)}=\nu\left(u_{j}^{(i)}\right)$.
For all $i<l$, $\pi_{i}$ is a blow-up along an ideal of the form
$\left(u_{J_{i}}^{(i)}\right)$. Renumbering if necessary, we may
assume that $1\in J_{i}$ and that $A_{i+1}$ is a localization of
$A_{i}\left[\frac{u_{J_{i}}^{(i)}}{u_{1}^{(i)}}\right]$. Hence, $\beta_{1}^{(i)}=\min\limits _{j\in J_{i}}\left\{ \beta_{j}^{(i)}\right\} $.
\begin{lem}
Let $i\in\left\{ 0,\dots,l-1\right\} $. We assume that the sequence
$\pi_{0},\dots,\pi_{i-1}$ of \ref{eq:suitelocale-1} is monomial. 

We write $w^{\gamma}=\left(w^{(i)}\right)^{\gamma^{(i)}}$ and $w^{\delta}=\left(w^{(i)}\right)^{\delta^{(i)}}$.
Then:
\end{lem}

\begin{enumerate}
\item $r_{i}=r$,
\item 
\begin{equation}
\sum\limits _{q\in E}\left(\gamma_{q}^{(i)}-\delta_{q}^{(i)}\right)\beta_{q}^{(i)}=0,\label{eq:relation-1}
\end{equation}
\item $\mathrm{gcd}\left(\gamma_{1}^{(i)}-\delta_{1}^{(i)},\dots,\gamma_{r}^{(i)}-\delta_{r}^{(i)},\gamma_{n}^{(i)}-\delta_{n}^{(i)}\right)=1,$
\item Every $\mathbb{Z}$-linear dependence relation between $\beta_{1}^{(i)},\dots,\beta_{r}^{(i)},\beta_{n}^{(i)}$
is an integer multiple of (\ref{eq:relation-1}).
\end{enumerate}
\begin{proof}
~
\begin{enumerate}
\item It is enough to do an induction on $i$ and use Remark \ref{rem: =0000E9clatement monomial nombre de g=0000E9n=0000E9rateurs}.
\item We have $\nu\left(w^{\gamma}\right)=\nu\left(w^{\delta}\right)$,
in other words $\nu\left(\left(w^{(i)}\right)^{\gamma^{(i)}}\right)=\nu\left(\left(w^{(i)}\right)^{\delta^{(i)}}\right)$.
Since $w^{(i)}=\left(w_{1}^{(i)},\dots,w_{r_{i}}^{(i)},w_{n_{i}}^{(i)}\right)$,
we have: 
\[
\nu\left(\prod\limits _{j=1}^{r_{i}}\left(w_{j}^{(i)}\right)^{\gamma_{j}^{(i)}}\times\left(w_{n_{i}}^{(i)}\right)^{\gamma_{n_{i}}^{(i)}}\right)=\nu\left(\prod\limits _{j=1}^{r_{i}}\left(w_{j}^{(i)}\right)^{\delta_{j}^{(i)}}\times\left(w_{n_{i}}^{(i)}\right)^{\delta_{n_{i}}^{(i)}}\right).
\]
So we have
\[
\sum\limits _{j=1}^{r_{i}}\gamma_{j}^{(i)}\nu\left(w_{j}^{(i)}\right)+\gamma_{n_{i}}^{(i)}\nu\left(w_{n_{i}}^{(i)}\right)=\sum\limits _{j=1}^{r_{i}}\delta_{j}^{(i)}\nu\left(w_{j}^{(i)}\right)+\delta_{n_{i}}^{(i)}\nu\left(w_{n_{i}}^{(i)}\right).
\]
By definition of $w^{(i)}$, for every integer $j\in\left\{ 1,\dots,r_{i},n_{i}\right\} $,
we have $w_{j}^{(i)}=u_{j}^{(i)}$. So $\nu\left(w_{j}^{(i)}\right)=\beta_{j}^{(i)}$.
Then: 
\[
\sum\limits _{j=1}^{r_{i}}\gamma_{j}^{(i)}\beta_{j}^{(i)}+\gamma_{n_{i}}^{(i)}\beta_{n_{i}}^{(i)}=\sum\limits _{j=1}^{r_{i}}\delta_{j}^{(i)}\beta_{j}^{(i)}+\delta_{n_{i}}^{(i)}\beta_{n_{i}}^{(i)}.
\]
Hence $\sum\limits _{j\in\left\{ 1,\dots,r_{i},n_{i}\right\} }\left(\gamma_{j}^{(i)}-\delta_{j}^{(i)}\right)\beta_{j}^{(i)}=0$.
\\
But $r_{i}=n_{i}-t-1=r$, so $n_{i}=r+t+1=n$, and:
\[
\begin{array}{ccc}
\sum\limits _{j\in\left\{ 1,\dots,r_{i},n_{i}\right\} }\left(\gamma_{j}^{(i)}-\delta_{j}^{(i)}\right)\beta_{j}^{(i)} & = & \sum\limits _{j\in\left\{ 1,\dots,r,n\right\} }\left(\gamma_{j}^{(i)}-\delta_{j}^{(i)}\right)\beta_{j}^{(i)}\\
 & = & \sum\limits _{j\in E}\left(\gamma_{j}^{(i)}-\delta_{j}^{(i)}\right)\beta_{j}^{(i)}\\
 & = & 0.
\end{array}
\]
\item Same proof as in Theorem \ref{thm:monomialise}.
\item Same proof as in Theorem \ref{thm:monomialise}.
\end{enumerate}
\end{proof}
\begin{lem}
\label{lem:suitenonmonomiale-1}The sequence $\left(A,u\right)\overset{\pi_{0}}{\to}\left(A_{1},u^{(1)}\right)\overset{\pi_{1}}{\to}\cdots\overset{\pi_{l-1}}{\to}\left(A_{l},u^{(l)}\right)$
of Theorem \ref{thm:monomialise general} is not monomial.
\end{lem}

\begin{proof}
Same proof as Lemma \ref{lem:suitenonmonomiale}.
\end{proof}
\begin{lem}
\label{lem:chaqueeclatementestmonomial-1}Let $i\in\{0,\dots,l-1\}$
and we assume that $\pi_{0},\dots,\pi_{i-1}$ are all monomial. Then
following properties are equivalent:
\begin{enumerate}
\item The blow-up $\pi_{i}$ is not monomial.
\item There exists a unique index $q\in J_{i}\setminus\left\{ 1\right\} $
such that $\beta_{q}^{(i)}=\beta_{1}^{(i)}$.
\item We have $i=l-1$.
\end{enumerate}
\end{lem}

\begin{proof}
Same proof as Lemma \ref{lem:chaqueeclatementestmonomial}.
\end{proof}
Using induction on $i$ and Lemma \ref{lem:chaqueeclatementestmonomial-1},
we conclude that $\pi_{0},\dots,\pi_{l-2}$ are monomial. This proves
the first point of the Theorem.

It remains to prove the last three points.

By Lemma \ref{lem:chaqueeclatementestmonomial-1} we know that there
exists a unique element $q\in J_{l-1}\setminus\left\{ j_{l-1}\right\} $
such that $\beta_{q}^{(l-1)}=\beta_{1}^{(l-1)}$, hence we are in
the case $\sharp B_{l-1}+1=\sharp J_{l-1}-1$. We now have to see
if $t_{k_{l-1}}=0$ or $1$.

We recall that $w_{1}^{(l-1)}=w^{\epsilon}$ and $w_{q}^{(l-1)}=w^{\mu}$
where $\epsilon$ and $\mu$ are two columns of a unimodular matrix
such that $\mu-\epsilon=\pm(\gamma-\delta)$. So $x_{1}^{(l-1)}=x^{\epsilon}$
and $x_{q}^{(l-1)}=x^{\mu}$, then 
\[
\frac{x_{q}^{(l-1)}}{x_{1}^{(l-1)}}=x^{\mu-\epsilon}=x^{\pm\left(\gamma-\delta\right)}=x^{\pm\left(\alpha_{1}^{(0)},\dots,\alpha_{r}^{(0)},-\overline{\alpha^{(0)}}\right)}.
\]

In other words 
\[
\frac{x_{q}^{(l-1)}}{x_{1}^{(l-1)}}=\left(\frac{\prod\limits _{j=1}^{r}x_{j}^{\alpha_{j}^{(0)}}}{x_{n}^{\overline{\alpha^{(0)}}}}\right)^{\pm1}=\left(z^{-1}\right)^{\pm1}=z^{\pm1}.
\]
 So we can assume $\frac{x_{q}^{(l-1)}}{x_{1}^{(l-1)}}=z$ .

The case $t_{k_{l-1}}=1$ corresponds to the fact that $z$ is transcendantal
over $k$, in other words $\lambda_{0}=0$. The case $t_{k_{l-1}}=0$
corresponds to the fact that $z$ is algebraic over $k$, in other
words $\lambda_{0}\neq0$. The third point of the Theorem is then
a consequence of \ref{fact:nombredegene}.

Since $\beta_{1}^{(l-1)},\dots,\beta_{r}^{(l-1)}$ are linearly independent,
we have $q=n$. By \ref{fact:nombredegene}, if $\lambda_{0}\neq0$,
we have 
\[
w_{n}^{(l)}=u_{n}^{(l)}=\overline{\lambda_{0}}(u'_{n})=\overline{\lambda_{0}}\left(\frac{u_{n}^{(l-1)}}{u_{1}^{(l-1)}}\right)=\overline{\lambda_{0}}\left(\frac{w_{n}^{(l-1)}}{w_{1}^{(l-1)}}\right)=\overline{\lambda_{0}}\left(\overline{z}\right)=\sum\limits _{i=0}^{d}a_{i}b_{i}\overline{z}^{i}.
\]
\begin{rem}
We have $\overline{\lambda_{0}}\left(\overline{z}\right)=\sum\limits _{i=0}^{d}c_{i}b_{i}\overline{z}^{i}$
where $c_{i}$ are units. Then we choose to set $c_{i}=a_{i}$ for
every index $i$.
\end{rem}

But since $\overline{z}=\frac{w_{n}^{\overline{\alpha^{(0)}}}}{\overline{y}}$,
we have 
\[
w_{n}^{(l)}=\sum\limits _{i=0}^{d_{0}}a_{i}b_{i}\left(\frac{w_{n}^{\overline{\alpha^{(0)}}}}{\overline{y}}\right)^{i}=\frac{\sum\limits _{i=0}^{d_{0}}a_{i}b_{i}\overline{y}^{d_{0}-i}\left(w_{n}^{\overline{\alpha^{(0)}}}\right)^{i}}{\overline{y}^{d_{0}}}=\frac{Q}{\overline{y}^{d_{0}}}
\]
 and the point $(5)$ is proven.

So it remains to prove the point $\left(4\right)$.

We apply Proposition \ref{prop:chgmtvari} to $i=0$ and $i'=l$.
By the monomiality of $\pi_{0},\dots,\pi_{l-2}$, we know that $D_{i}=\{1,\dots,n\}$
for every $i\in\{1,\dots,l-1\}$.

We know that $D_{l}=\{1,\dots,n\}$ if $\lambda\neq0$ and $D_{l}=\{1,\dots,n-1\}$
otherwise. Here we set again $u_{T}=v$.

By Proposition \ref{prop:chgmtvari}, for every $j\in\left\{ 1,\dots,r,n\right\} $,
$w_{j}=u_{j}$ is a monomial in $w_{1}^{(l)},\dots,w_{r}^{(l)}$ (or
equivalently in $u_{1}^{(l)},\dots,u_{r}^{(l)}$) multiplied by an
element of $A_{l}^{\times}$.

Same thing for the fact that for every integer $j\in\{1,\dots,r\}$,
we have $w_{j}^{(l)}=w^{\eta}$. This completes the proof.
\end{proof}
\begin{rem}
In the case $Q_{2}=Q$, the we constructed a local framed sequence
such that the total transform of $Q_{2}$ is a monomial. We will bring
us to this case.
\end{rem}

\begin{defn}
\cite{CRS} \index{Generalized Puiseux package}A local framed sequence
that satisfies Theorem \ref{thm:monomialise general} is called a\emph{
$n$-generalized Puiseux package}.

Let $j\in\left\{ r+1,\dots,n\right\} $. A \emph{$j$-generalized
Puiseux package }is a $n$-generalized Puiseux package replacing $n$
by $j$ in Theorem \ref{thm:monomialise general}.
\end{defn}

\begin{rem}
We consider $\left(A,u\right)\to\dots\to\left(A_{i},u^{(i)}\right)\to\dots$
a $j$-generalized Puiseux package, with $j\in\left\{ r+1,\dots,n\right\} $.
We replace each ring of this sequence by its formal completion, hence
we obtain o formal framed sequence that we call a formal $j$-Puiseux
package. So Theorem \ref{thm:monomialise general} induces a formal
$n$-Puiseux package that satisfies the same conclusion as in Theorem
\ref{thm:monomialise general}.
\end{rem}

Since we want to to an induction, now we will assume until the end
of Theorem \ref{thm:monomialisation de Q3-1}, that we know how to
monomialize every complete local equicharacteristic quasi excellent
ring $G$ of dimension strictly less than $n$ equipped with a valuation
of rank $1$ centered in $G$ by a formal framed sequence. This hypothesis
is called $H_{n}$.
\begin{lem}
\label{lem: on se ramene au bon cas-2}Let $P=\sum\limits _{j\in S_{u_{n}}\left(P\right)}c_{j}u_{n}^{j}$
the $u_{n}$-expansion of an optimal immediat successor key element
of $u_{n}$. 

There exists a formal framed sequence $\left(A,u\right)\to\left(A_{l},u^{(l)}\right)$
that transforms each coefficient $c_{j}$ in a monomial in $\left(u_{1}^{(l)},\dots,u_{r}^{(l)}\right)$,
multiplied by a unit of $A_{l}$. 

Hence, after this sequence, $P$ can be written like $\sum\limits _{i=0}^{d_{0}}a_{i}b_{i}\overline{y}^{d_{0}-i}\left(w_{n}^{\overline{\alpha^{(0)}}}\right)^{i}$.
\end{lem}

\begin{proof}
We will prove a more general result in \ref{lem:on se ramene au bon cas le dernier}.
\end{proof}
\begin{thm}
\label{thm: monomialiase pc lim de u_n-1} If $u_{n}\ll_{\lim}P$,
then $P$ is monomializable.
\end{thm}

\begin{proof}
Same proof as Theorem \ref{thm: monomialiase pc lim de u_n}.
\end{proof}
\begin{lem}
There exists a formal framed sequence 
\[
\left(A,u\right)\to\left(A_{l},u^{(l)}\right)
\]
 such that in $A_{l}$, the strict transform of the polynomial $Q_{2}$
is a monomial.
\end{lem}

\begin{proof}
If $u_{n}<Q_{2}$, we use Lemma \ref{lem: on se ramene au bon cas-2}
and Theorem \ref{thm:monomialise general} to conclude. Otherwise,
$u_{n}<_{\lim}Q_{2}$ and so we use Theorem \ref{thm: monomialiase pc lim de u_n}.
\end{proof}
We constructed a formal framed sequence that monomializes $Q_{2}$.
But we want one that monomializes all the key polynomials of $\mathcal{Q}$. 

Now we are going to show that if we constructed a formal framed sequence
$\left(A,u\right)\to\left(A_{l},u^{\left(l\right)}\right)$ that monomializes
$Q_{i}$, then we can associate another $\left(A_{l},u^{\left(l\right)}\right)\to\left(A_{s},u^{\left(s\right)}\right)$
such that in $A_{s}$, the strict transform of $Q_{i+1}$ is also
a monomial.

Let $\Delta_{l}$ be the group $\nu\left(k_{l}\left(u_{1}^{\left(l\right)},\dots,u_{n-1}^{\left(l\right)}\right)\setminus\left\{ 0\right\} \right)$
and 
\[
\alpha_{l}:=\min\left\{ h\text{ such that }h\beta_{n}^{(l)}\in\Delta_{l}\right\} .
\]

We set $X_{j}=\mathrm{in}_{\nu}\left(u_{j}^{(l)}\right)$, $W_{j}=w_{j}^{(l)}$
and $\lambda_{l}$ the minimal polynomial of $X_{n}$ over $\mathrm{gr}_{\nu}k_{l}\left(u_{1}^{\left(l\right)},\dots,u_{n-1}^{\left(l\right)}\right)$
of degree $\alpha_{l}$.

We know that $Q_{i}=\overline{\omega}w_{n}^{(l)}$ with $\overline{\omega}$
a monomial in $W_{1},\dots,W_{r}$ multiplied by a unit. We set $\omega:=\mathrm{in}_{\nu}\left(\overline{\omega}\right)$.

If $Q_{i}<_{\lim}Q_{i+1}$, we use Theorem \ref{thm: monomialiase pc lim de u_n-1}
and the proof is finished. So we assume that $Q_{i+1}$ is an optimal
immediate successor of $Q_{i}$.

We write $Q_{i+1}=\sum\limits _{j\in S_{Q_{i}}(Q_{i+1})}a_{j}Q_{i}^{j}=\sum\limits _{j=0}^{s}a_{j}Q_{i}^{j}$
the $Q_{i}$-expansion of $Q_{i+1}$ in $k_{l}\left(u_{1}^{\left(l\right)},\dots,u_{n-1}^{\left(l\right)}\right)\left(u_{n}^{\left(l\right)}\right)$
.

We have $Q_{i+1}=Q_{i}^{s}+a_{s-1}Q_{i}^{s-1}+\cdots+a_{0}$ and since
$Q_{i}=\overline{\omega}w_{n}^{(l)}$, we have
\[
\frac{Q_{i+1}}{\overline{\omega}^{s}}=\left(u_{n}^{(l)}\right)^{s}+\frac{a_{s-1}}{\overline{\omega}}\left(u_{n}^{(l)}\right)^{s-1}+\cdots+\frac{a_{0}}{\overline{\omega}^{s}}.
\]

We know that for every index $j$ such that $a_{j}\neq0$, we have
\[
\nu\left(a_{j}Q_{i}^{j}\right)=\nu_{Q_{i}}\left(Q_{i+1}\right).
\]
 So all non-zero terms of the $Q_{i}$-expansion of $Q_{i+1}$ have
same value. Then, by hypothesis $H_{n}$, all these terms are divisible
by the same power of $\overline{\omega}$ after an appropriate sequence
of blow-ups $(\ast_{i})$ independent of $u_{n}^{(l)}$.

We denote by $\widetilde{Q}_{i+1}$ the strict transform of $Q_{i+1}$
by the composition of $\left(\ast_{i}\right)$ with the sequence $\left(\ast'_{i}\right)$
that monomializes $Q_{i}$. We denote this composition by $\left(c_{i}\right)$. 

We know that $\widetilde{Q}_{i}$, the strict transform of $Q_{i}$
by $\left(c_{i}\right)$, is a regular parameter of the maximal ideal
of $A_{l}$. Indeed, by Proposition \ref{prop:chgmtvari}, we know
that each $u_{j}$ of $A$ can be written as a monomial on $w_{1}^{(l)},\dots,w_{r}^{(l)}$.
In fact, the reduced exceptional divisor of this sequence is exactly
$\mathrm{V}\left(\overline{\omega}\right)_{\mathrm{red}}$. Hence,
as we know that $Q_{i}=w_{n}^{(l)}\overline{\omega}$, we do have
that the strict transform of $Q_{i}$ is $\widetilde{Q}_{i}=w_{n}^{(l)}=u_{n}^{(l)}$.
So it is a key polynomial in the extension $k_{l}\left(u_{1}^{(l)},\dots,u_{n-1}^{(l)}\right)\left(u_{n}^{(l)}\right)$.

Let us show that $\widetilde{Q}_{i+1}=\frac{Q_{i+1}}{\overline{\omega}^{s}}$. 

We have $a_{s}=1$ and $Q_{i}^{s}=\overline{\omega}^{s}\left(u_{n}^{(l)}\right)^{s}$
and also $u_{n}^{(l)}\nmid\overline{\omega}$, so $\overline{\omega}^{s}$
divides the term $a_{s}Q_{i}^{s}$ and so all the non-zero terms of
the $Q_{i}$-expansion of $Q_{i+1}$. Furthermore, it is the biggest
power of $\overline{\omega}$ that divides each term, hence $\frac{Q_{i+1}}{\overline{\omega}^{s}}\left(u_{n}^{(l)}\right)^{s}+\frac{a_{s-1}}{\overline{\omega}}\left(u_{n}^{(l)}\right)^{s-1}+\cdots+\frac{a_{0}}{\overline{\omega}^{s}}$
is $\widetilde{Q}_{i+1}$ the strict transform of $Q_{i+1}$ by the
sequence of blow-ups, that satisfies $\widetilde{Q}_{i}\ll\widetilde{Q}_{i+1}$
by hypothesis.

Let $G$ be a complete local equicharaceristic ring of dimension strictly
less than $n$ equipped with a valuation centered in $G$.
\begin{lem}
\label{lem:on se ramene au bon cas le dernier}We assume that for
every ring $G$ as above, every element of $G$ is monomializable.

Assume that$Q_{i}<Q_{i+1}$ in $\mathcal{Q}$.

Then there exists a local framed sequence $\left(A_{l},u^{(l)}\right)\to\left(A_{e},u^{(e)}\right)$
such that in $A_{e}$, the strict transform of $Q_{i+1}$ is of the
form $\sum\limits _{q=0}^{s}\tau_{q}\eta_{q}X_{n}^{q}$

where $\tau_{q}\in R_{e}^{\times}$ and $\eta_{q}$ are monomials
in $u_{1}^{\left(e\right)},\dots,u_{r}^{\left(e\right)}$.
\end{lem}

\begin{proof}
By hypothesis, after a sequence of blow-ups independent of $u_{n}^{(l)}$,
we can monomialize the $a_{j}$ and assume that they are monomials
in $\left(u_{1}^{(l)},\dots,u_{n-1}^{(l)}\right)$ multiplied by units
of $A_{l}$.

For every $g\in\left\{ r+1,\dots,n-1\right\} $, we do a generalized
$g$-Puiseux package as in Theorem \ref{thm:monomialise general},
hence we have a sequence 
\[
\left(A_{l},u^{(l)}\right)\to\left(A_{t},u^{(t)}\right)
\]
 such that each $u_{g}^{(l)}$ is a monomial in$\left(u_{1}^{(t)},\dots,u_{r}^{(t)}\right)$.

In fact we can assume that the $a_{j}$ are monomials in $\left(u_{1}^{(l)},\dots,u_{r}^{(l)}\right)$
multiplied by units of $A_{l}$.

Since the strict transform 
\[
\widetilde{Q}_{i+1}=\frac{Q_{i+1}}{\overline{\omega}^{s}}=\left(u_{n}^{(l)}\right)^{s}+\frac{a_{s-1}}{\overline{\omega}}\left(u_{n}^{(l)}\right)^{s-1}+\cdots+\frac{a_{0}}{\overline{\omega}^{s}}
\]
 is an immediate successor key element of $\widetilde{Q}_{i}$, this
completes the proof.
\end{proof}
\begin{rem}
Lemma \ref{lem: on se ramene au bon cas-2} is a particular case of
Lemma \ref{lem:on se ramene au bon cas le dernier}.
\end{rem}

\begin{thm}
\label{thm:monomialisation de Q3-1}We still assume $H_{n}$.

We recall that $\mathrm{car}\left(k_{\nu}\right)=0$. If $Q_{i}$
is monomializable, then there exists a formal framed sequence 
\begin{equation}
\left(A,u\right)\overset{\pi_{0}}{\to}\left(A_{1},u^{(1)}\right)\overset{\pi_{1}}{\to}\cdots\overset{\pi_{l-1}}{\to}\left(A_{l},u^{(l)}\right)\overset{\pi_{l}}{\to}\cdots\overset{\pi_{m-1}}{\to}\left(A_{m},u^{(m)}\right)\label{eq:suitelocale2-1}
\end{equation}
that monomializes $Q_{i+1}$.
\end{thm}

\begin{proof}
There are two cases.

The first one: $Q_{i}<Q_{i+1}$. 

Then the strict transform $\widetilde{Q}_{i+1}$ of $Q_{i+1}$ by
the sequence $\left(A,u\right)\to\left(A_{l},u^{(l)}\right)$ that
monomializes $Q_{i}$ is an immediate successor key element of $\widetilde{Q}_{i}=u_{n_{l}}^{(l)}$,
and by Lemma \ref{lem:on se ramene au bon cas le dernier} we just
saw that we can bring us to the hypothesis of Theorem \ref{thm:monomialise general}.
So we use Theorem \ref{thm:monomialise general} replacing $Q_{1}$
by $\widetilde{Q}_{i}$ and $Q_{2}$ by $\widetilde{Q}_{i+1}$.

The last one: $Q_{i}<_{\lim}Q_{i+1}$. 

We apply Theorem \ref{thm: monomialiase pc lim de u_n-1} replacing
$u_{n}$ by $\widetilde{Q}_{i}$ and $P$ by $\widetilde{Q}_{i+1}$.
\end{proof}
As in the previous part, we consider, for every integer $j$, the
countable sets 
\[
\mathscr{S}_{j}:=\left\{ \prod\limits _{i=1}^{n}\left(u_{i}^{(j)}\right)^{\alpha_{i}^{(j)}}\text{, with }\alpha_{i}^{(j)}\in\mathbb{Z}\right\} 
\]
 and 
\[
\widetilde{\mathscr{S}}_{j}:=\left\{ \left(s_{1},s_{2}\right)\in\mathscr{S}_{j}\times\mathscr{S}_{j}\text{, with }\nu\left(s_{1}\right)\leq\nu\left(s_{2}\right)\right\} 
\]
 assuming that for every $i\in\left\{ 1,\dots,n\right\} $, $u_{i}^{(0)}=u_{i}$.

The set $\widetilde{\mathscr{S}}_{j}$ is countable for every $j$,
so we can number its elements, and set $\widetilde{\mathscr{S}}_{j}:=\left\{ s_{m}^{(j)}\right\} _{m\in\mathbb{N}}$.
Now we consider the finite set 
\[
\mathscr{S}'_{j}:=\left\{ s_{m}^{(j)},\text{ }m\leq j\right\} \cup\left\{ s_{j}^{(m)},\text{ }m\leq j\right\} .
\]

Hence $\bigcup\limits _{j\in\mathbb{N}}\left(\mathscr{S}_{j}\times\mathscr{S}_{j}\right)=\bigcup\limits _{j\in\mathbb{N}}\widetilde{\mathscr{S}}_{j}=\bigcup\limits _{j\in\mathbb{N}}\mathscr{S}'_{j}$
is a countable union of finite sets.

Since we consider all the elements according uniquely to the variable
$u_{n}$, and more generally according to $u_{n}^{(i)}$, and since
we do an induction on the dimension, we have to know how to monomialize
the elements of $B_{i}:=k\left[u_{1}^{(i)},\dots,u_{n-1}^{(i)}\right]$. 
\begin{thm}
\label{thm:generalisation de la monomialisation-1}Let $A\simeq k\left[\left[u_{1},\dots,u_{n}\right]\right]$
equipped with a valuation $\nu$ centered in $A$.

We recall that $\mathrm{car}\left(k_{\nu}\right)=0$. There exists
a formal sequence 
\begin{equation}
\left(A,u\right)\overset{\pi_{0}}{\to}\cdots\overset{\pi_{s-1}}{\to}\left(A_{s},u^{(s)}\right)\overset{\pi_{s}}{\to}\cdots\label{eq:suitelocaleterminator-1}
\end{equation}

that monomializes all the key polynomials of $\mathcal{Q}$ and all
the elements of the $B_{i}$ for all $i$. Furthermore, the sequence
has the property:
\[
\forall j\in\mathbb{N}\text{ }\forall s=\left(s_{1},s_{2}\right)\in\mathscr{S}'_{j}\text{ }\exists i\in\mathbb{N}_{\geq j}\text{ such that }s_{1}\mid s_{2}\text{ in }A_{i}.
\]

In other words for every index $l$, there exists an index $p_{l}$
such that in $A_{p_{l}}$, $Q_{l}$ is a monomial in $u^{\left(p_{l}\right)}$
multiplied by a unit of $A_{p_{l}}$.
\end{thm}

\begin{proof}
To show that we can choose the sequence (\ref{eq:suitelocaleterminator-1})
such that
\[
\forall j\in\mathbb{N}\text{ }\forall s=\left(s_{1},s_{2}\right)\in\mathscr{S}'_{j}\text{ }\exists i\in\mathbb{N}_{\geq j}\text{ such that }s_{1}\mid s_{2}\text{ in }A_{i},
\]

and that all the elements of the $B_{i}$ are monomialized, we do
the same thing than in Theorem \ref{thm: La suite de la mort qui tue}.

Then we do an induction on the dimension $n$ and on the index $i$
and we iterate the above process.
\end{proof}
\begin{cor}
\label{cor: h est monomialisable}Let $A\simeq k\left[\left[u_{1},\dots,u_{n}\right]\right]$
equipped with a valuation $\widehat{\nu}$ centered in $A$, of value
group $\Gamma$. We assume 
\[
I=\left\{ a\in A\text{ such that }\widehat{\nu}\left(a\right)\notin\Gamma_{1}\right\} =\left(h\right)\neq\left(0\right),
\]
 where $\Gamma_{1}$ is the smallest isolated subgroup of $\Gamma$.
We recall that $\mathrm{car}\left(k_{\nu}\right)=0$. 

There exists a formal framed sequence 
\[
\left(A,u\right)\to\dots\to\left(A_{l},u^{\left(l\right)}\right)\to\dots
\]
 such that in $A_{l},$ the polynomial $h$ can be written as a monomial
multiplied by a unit.
\end{cor}

\begin{proof}
The sequence $\mathcal{Q}$ has been constructed to contain $h$,
so we just have to use Theorem \ref{thm:generalisation de la monomialisation-1}.
\end{proof}
\newpage{}

\section{Reduction.}

Let $\left(R,\mathfrak{m},k\right)$ be a local quasi excellent equicharacteristic
ring and let $\nu$ be a valuation of its field of fractions, of rank
$1$, centered in $R$ and of value group $\Gamma_{1}$. We denote
by $\overline{H}$ the implicit ideal of $R$.

We are going to see that in this case, we just have to regularise
$\frac{\widehat{R}}{\overline{H}}$.

We consider $\mathcal{F}:=\left\{ f_{1},\dots,f_{s}\right\} \subseteq\mathfrak{m}$,
and assume that $f_{1}$ has minimal value.
\begin{rem}
We consider $R\to\widehat{R}\to R_{1}\to\widehat{R}_{1}$ a formal
framed blow-up and we denote by $H'$ the strict transformed of $\overline{H}$
in $R_{1}$.

Then we define $\overline{H_{1}}$ as the preimage in $\widehat{R}_{1}$
of the implicit ideal of $\frac{\widehat{R}_{1}}{H'\widehat{R}_{1}}$.

We iterate this contruction for every formal framed sequence.
\end{rem}

\begin{thm}
\label{thm: quotient reg} We recall that $\mathrm{car}\left(k_{\nu}\right)=0$.
There exists a formal framed sequence 
\[
\left(R,u,k\right)=\left(R_{0},u^{(0)},k_{0}\right)\to\cdots\to\left(R_{i},u^{(i)}=\left(u_{1}^{(i)},\dots,u_{n}^{(i)}\right),k_{i}\right)
\]
 such that:
\begin{enumerate}
\item The ring $\frac{\widehat{R_{i}}}{\overline{H_{i}}}$ is regular, 
\item For every index $j$, we have that $f_{j}\mathrm{\mod}\left(\overline{H_{i}}\right)$
is a monomial in $u^{(i)}$ multiplied by a unit of $\frac{\widehat{R_{i}}}{\overline{H_{i}}}$,
\item For every index $j$, we have $f_{1}\mathrm{mod}\left(\overline{H_{i}}\right)\mid f_{j}\mathrm{mod}\left(\overline{H_{i}}\right)$
in $\frac{\widehat{R_{i}}}{\overline{H_{i}}}$.
\end{enumerate}
\end{thm}

\begin{proof}
Set $n:=e\left(R,\nu\right)$ and $u:=\left(y,x\right)$ with 
\[
y:=\left(y_{1},\dots,y_{\tilde{n}-n}\right)
\]
 and 
\[
x:=\left(x_{1},\dots,x_{n}\right)
\]
 such that the images of the $x_{j}$ in $\frac{\widehat{R}}{\overline{H}}$
induce a minimal set of generators of $\frac{\mathfrak{m}}{\overline{H}}$
and such that $y$ generates $\overline{H}$.

We do an induction on $\left(n_{i},n_{i}-r_{i},v_{i}\right)$.

We saw the existence of the surjection $\Phi$ from $A\simeq k\left[\left[x_{1},\dots,x_{n}\right]\right]$
to $\frac{\widehat{R}}{\overline{H}}$, of kernel $I=\left\{ f\in A\text{ such that }\widehat{\nu}\left(f\right)\notin\Gamma_{1}\right\} \in\mathrm{Spec}\left(A\right)$
where $\widehat{\nu}$ is defined as in section \ref{sec:Id=0000E9al-implicite.}.
We denote by $L$ the field of fractions of $A$. 

If $v_{0}<n-1$, then we do the same thing as in Proposition \ref{prop:monomialisation en dehors du carre}
and we strictly decrease $e\left(A,\widehat{\nu}\right)$.

The we can assume $v_{0}\in\left\{ n-1,n\right\} $. 

Assume $v_{0}=n-1$.

Then we know that $I=\left(h\right)$ and that there exists a formal
framed sequence $\left(A,x\right)\to\left(A_{\ell},x^{\left(\ell\right)}\right)$
that monomializes $h$ by Corollary \ref{cor: h est monomialisable}.
So one of the generators that appears in its decomposition must be
in $I_{\ell}$. Hence there exists $x_{p}^{\left(\ell\right)}$ such
that $\widehat{\nu}\left(x_{p}^{\left(\ell\right)}\right)\notin\Gamma_{1}$.
So by Theorems \ref{prop:monomialisation en dehors du carre} and
\ref{thm:decroissance de la dimension plongee}, there exists a local
framed sequence that decreases strictly $e\left(A,\widehat{\nu}\right)$,
so this case can happen a finite number of time, and we bring us at
the case $I=\left(0\right)$. It means the case where $\frac{\widehat{A}}{I}$
is regular.

Case $I=\left(0\right)$. For every $f_{j}$, we have $\widehat{\nu}\left(f_{j}\right)\in\Gamma_{1}$
. So the element $f_{j}$ is a non-zero formal series and by Weierstrass
preparation Theorem, we know that we can see it like a polynomial
in $x_{l}$ with coefficients in $k\left[\left[x_{1},\dots,x_{n-1}\right]\right]$.
We construct a sequence of key polynomials in the extension $k\left(\left(x_{1},\dots,x_{n-1}\right)\right)\left(x_{n}\right)$
as in previous section. In other words this sequence is a sequence
of optimal (possibly limit) immediate successors which is cofinal
in $\epsilon\left(\Lambda\right)$, where $\Lambda$ is the set of
key polynomials. So the element $f_{j}$ is non-degenerate with respect
of one of these polynomials that all are monomializable by the above
part. Hence there exists a local framed sequence $\left(A,x\right)\to\left(A_{i},x^{(i)}\right)$
such that in $A_{i}$, the strict transform of $f_{j}$ is a monomial
in $x^{(i)}$ multiplied by a unit of $A_{i}$.

If there exists a formal framed sequence such that $v_{i}<n-1$, then
by Proposition \ref{prop:monomialisation en dehors du carre}, we
can conclude by induction.

Iterating the case $I=\left(0\right)$, we assure the existence of
a local framed sequence such that all the strict transforms of the
$f_{j}$ are monomials multiplied by units. Doing another blow-up
if necessary, we assume that there exists of a local framed sequence
$\left(A,x\right)\to\left(A',x'\right)$ such that all the strict
transforms of the $f_{j}$ are monomials only in $x'_{1},\dots,x'_{r}$. 

By Proposition \ref{prop:existencesuitedivise}, we can assume that
for every $j$ and every $p$, we have either $f_{j}\mid f_{p}$ or
$f_{p}\mid f_{j}$.

So we have a local framed sequence 
\[
\left(A,x,k\right)\overset{\rho_{0}}{\to}\left(A_{1},x^{(1)},k_{1}\right)\overset{\rho_{1}}{\to}\cdots\overset{\rho_{i}}{\to}\left(A_{i},x^{(i)},k_{i}\right)
\]
 that monomializes the $f_{j}$ and such that for all $j$ and $q$,
we have $f_{j}\mid f_{q}$ or the converse.

By the minimality of $\nu\left(f_{1}\right)$, in $A_{i}$, we have
$f_{1}\mid f_{j}$ for every $j$.

We have also two maps
\[
\left(R,u,k\right)\to\left(\frac{\widehat{R}}{\overline{H}},x,k\right)\leftarrow\left(A,x,k\right),
\]
 and we know that $\frac{A}{I}\simeq\frac{\widehat{R}}{\overline{H}}$
since $I=\mathrm{Ker}\left(\Phi\right)$. Hence, looking at the strict
transform of $\frac{A}{I}$ at each step of the sequence $\left\{ \rho_{j}\right\} _{0\leq j\leq i}$,
we obtain a local framed sequence 
\[
\left(\frac{\widehat{R}}{\overline{H}},x,k\right)\overset{\tilde{\rho_{0}}}{\to}\left(\tilde{R}_{1},x^{(1)},k_{1}\right)\overset{\tilde{\rho_{1}}}{\to}\cdots\overset{\tilde{\rho_{i}}}{\to}\left(\tilde{R}_{i},x^{(i)},k_{i}\right).
\]

So we have the diagram:

\[
\begin{array}{ccccccc}
\left(\frac{\widehat{R}}{\overline{H}},x,k\right) & \overset{\tilde{\rho_{0}}}{\to} & \left(\tilde{R}_{1},x^{(1)},k_{1}\right) & \overset{\tilde{\rho_{1}}}{\to} & \cdots & \overset{\tilde{\rho_{i}}}{\to} & \left(\tilde{R}_{i},x^{(i)},k_{i}\right).\\
\uparrow &  & \uparrow &  & \uparrow &  & \uparrow\\
\left(A,x,k\right) & \overset{\rho_{0}}{\to} & \left(A_{1},x^{(1)},k_{1}\right) & \overset{\rho_{1}}{\to} & \cdots & \overset{\rho_{i}}{\to} & \left(A_{i},x^{(i)},k_{i}\right)
\end{array}
\]

Similarly, either $\frac{A}{I}$ is regular, or the sequence $\left\{ \rho_{j}\right\} $
can be chosen such that $e\left(R,\mu\right)$ strictly decreases.

So after a finite sequence of blow-ups, we bring us to the case where
$\frac{\widehat{R_{i}}}{\overline{H_{i}}}$ is regular. Hence we can
assume $\frac{\widehat{R_{i}}}{\overline{H_{i}}}$ regular and consider
$f_{1},\dots,f_{s}$ elements of $R\setminus\left\{ 0\right\} $ such
that $\nu\left(f_{1}\right)=\min\limits _{1\leq j\leq s}\left\{ \nu\left(f_{j}\right)\right\} $.
We know that the $f_{j}$ are all monomials in the $u^{(i)}$ and
that $f_{1}\mathrm{mod}\left(\overline{H_{i}}\right)\mid f_{j}\mathrm{mod}\left(\overline{H_{i}}\right)$.
This completes the proof.
\end{proof}
\begin{thm}
\label{thm: reduction quotient}Let $R$ be a local quasi excellent
domain and $H$ be his implicit prime ideal. We assume that $\frac{\widehat{R}}{H}$
is regular.

We recall that $\mathrm{car}\left(k_{\nu}\right)=0$. There exists
a sequence of blow-ups defined over $R$ that resolves the singularities
of R.
\end{thm}

\begin{proof}
The ring $\widehat{R}_{H}$ is regular by Proposition \ref{prop:localise regulier}.
So we know that there exist elements $\left(\widetilde{y}_{1},\dots,\widetilde{y}_{g}\right)$
of $H\widehat{R}_{H}$ that form a regular system of parameters of
$\widehat{R}_{H}$. 

By definition of $H\widehat{R}_{H}$, it means that there exist $y_{1},\dots,y_{g}$
elements of $H$ and $b_{1},\dots,b_{g}$ elements of $\widehat{R}\setminus H$
such that for every index $i$, we have $\widetilde{y}_{i}=\frac{y_{i}}{b_{i}}$.

The $b_{i}$ are elements of $\widehat{R}_{H}^{\times}$, so 
\[
\left(\widetilde{y}_{1},\dots,\widetilde{y}_{g}\right)\widehat{R}_{H}=\left(\frac{y_{1}}{b_{1}},\dots,\frac{y_{g}}{b_{g}}\right)\widehat{R}_{H}=\left(y_{1},\dots,y_{g}\right)\widehat{R}_{H}.
\]

Then we have some elements $\left(y_{1},\dots,y_{g}\right)$of $H$
that form a regular system of parameters of $\widehat{R}_{H}$.

Now we consider $\left(x_{1},\dots,x_{t}\right)$ some elements of
$\widehat{R}\setminus H$ whose images $\left(\overline{x}_{1},\dots,\overline{x}_{t}\right)$
modulo $H$ form a regular system of parameters of $\frac{\widehat{R}}{H}$.

If $\left(y_{1},\dots,y_{g}\right)$ generate $H$, then $\widehat{R}$
is regular. Indeed, in this case, $\left(y_{1},\dots,y_{g},x_{1},\dots,x_{t}\right)$
generate $\widehat{\mathfrak{m}}=\mathfrak{m}\otimes_{R}\widehat{R}$,
which is the maximal ideal of $\widehat{R}$. 

So 
\[
\dim\left(\widehat{R}\right)\leq g+t.
\]

We know that 
\[
g=\dim\left(\widehat{R}_{H}\right)=\mathrm{ht}\left(H\right)
\]
 and 
\[
t=\dim\left(\frac{\widehat{R}}{H}\right)=\mathrm{ht}\left(\frac{\widehat{\mathfrak{m}}}{H}\right).
\]

Then 
\[
\begin{array}{ccc}
\dim\left(\widehat{R}\right) & = & \mathrm{ht}\left(\widehat{\mathfrak{m}}\right)\\
 & \geq & \mathrm{ht}\left(H\right)+\mathrm{ht}\left(\frac{\widehat{\mathfrak{m}}}{H}\right)\\
 & = & g+t\\
 & \geq & \dim\left(\widehat{R}\right).
\end{array}
\]

Then $\dim\left(\widehat{R}\right)=g+t$ and $\left(y_{1},\dots,y_{g},x_{1},\dots,x_{t}\right)$
is a minimal set of generators of $\widehat{\mathfrak{m}}$, and so
$\widehat{R}$ is regular.

Now we assume that $\left(y_{1},\dots,y_{g}\right)$ do not generate
$H$ in $\widehat{R}$. So let us set $\left(y_{1},\dots,y_{g},y_{g+1},\dots,y_{g+s}\right)$
some elements that generate $H$ in $\widehat{R}$.

We consider $V:=\frac{H\widehat{R}_{H}}{H^{2}\widehat{R}_{H}}$ that
is a vector space of dimension $g=\mathrm{ht}\left(H\right)$ over
the residue field of $H$ since $\widehat{R}_{H}$ is regular.

We know that $y_{1},\dots,y_{g+s}$ generate $V$ and that 
\[
g+s>\mathrm{dim}\left(V\right)=g,
\]
 so there exist elements $a_{1},\dots,a_{g+s}$ of $\widehat{R}$
such that 
\[
a_{1}y_{1}+\dots+a_{g+s}y_{g+s}\in H^{2}\widehat{R}_{H}.
\]

In other words there exist $a_{1},\dots,a_{g+s}$ in $\widehat{R}$
and $\left(b_{i,j}\right)_{1\leq i,j\leq g+s}$ in $\widehat{R}_{H}$
such that 
\[
a_{1}y_{1}+\dots+a_{g+s}y_{g+s}=\sum\limits _{1\leq i,j\leq g+s}b_{i,j}y_{i}y_{j}.
\]

We may assume 
\[
\nu\left(a_{1}\right)=\min\limits _{1\leq i\leq s}\left\{ \nu\left(a_{i}\right)\right\} 
\]
 and also that for every $i$, the element $a_{i}$ is not in $H$
or is zero.

Since the $a_{i}$ are in $\widehat{R}$, we look at them modulo $H$.
By Theorem \ref{thm: La suite de la mort qui tue}, we know that the
classes $\overline{a_{i}}$ of $a_{i}$ modulo $H$ are monomialisable
in $\frac{\widehat{R}}{H}$ and that for every $i$, we have $\overline{a_{1}}\mid\overline{a_{i}}.$

Hence after a sequence of blow-ups, we have that $\overline{a_{1}}$
is a monomial $w=\prod\limits _{i=1}^{t}x_{i}^{c_{i}}$ in $x$ multiplied
by a unit.

If we can show that $a_{1}$ divides all the $b_{i,j}$, then we could
generate $H$ in $\widehat{R}$ by $\left(y_{2},\dots,y_{g+s}\right)$.

Iterating, we could generate $H$ in $\widehat{R}$ by $g$ elements,
and it would be over.

So let us show that we can do a sequence of blow-ups such that at
the end $a_{1}$ divides all the $b_{i,j}$.

For every index $i\in\left\{ 1,\dots,g+s\right\} $, there exists
$n_{i}\in\mathbb{N}_{>1}$ such that $y_{i}\in\widehat{\mathfrak{m}}^{n_{i}-1}\setminus\widehat{\mathfrak{m}}^{n_{i}}$.
We set $N:=\max\limits _{i\in\left\{ 1,\dots,g+s\right\} }\left\{ n_{i}\right\} $,
and then for every $i\in\left\{ 1,\dots,g+s\right\} $, $y_{i}\notin\widehat{\mathfrak{m}}^{N}$.

We have a map $R\to\widehat{R}$ and we know that for every integer
$c$, we have $\widehat{\mathfrak{m}}^{c}\cap R=\mathfrak{m}^{c}$.
Hence we have an isomorphism $\frac{R}{\mathfrak{m}^{c}}\to\frac{\widehat{R}}{\widehat{\mathfrak{m}}^{c}}$.

So for all $i\in\left\{ 1,\dots,g+s\right\} $, there exists $z_{i}\in R$
whose class modulo $\mathfrak{m}^{N+2}$ is sent on $y_{i}$ by this
map. Hence $z_{i}\mathrm{\text{ }mod}\left(\mathfrak{m}^{N+2}\right)=y_{i}$.
Increasing $N$ if necessary, we may assume $\nu\left(\widehat{\mathfrak{m}}^{N}\right)>\nu\left(a_{1}\right)$. 

More precisely $y_{i}=z_{i}+h_{i}+\zeta_{i}$ where $h_{i}\in\left(z_{1},\dots,z_{g+s}\right)^{2}$
and $\zeta_{i}\in\left(x_{1},\dots,x_{t}\right)^{N}$.

After a sequence of blow-ups independent of $\left(z_{1},\dots,z_{g+s}\right)$,
we may assume that $w$, and so $a_{1}$, divides all the $\zeta_{i}$.

We do $c_{1}$ blow-ups of $\left(z_{1},\dots,z_{g+s},x_{1}\right)$.
Each $z_{i}$ is transformed in a $z_{i}'$ which is of the form $\frac{z_{i}}{x_{1}^{c_{1}}}$.

We do $c_{2}$ blow-ups of $\left(z_{1}',\dots,z_{g+s}',x_{2}\right)$.
Each $z_{i}'$ is transformed in a $z_{i}"$ which is of the form
$\frac{z_{i}'}{x_{2}^{c_{2}}}=\frac{z_{i}}{x_{1}^{c_{1}}x_{2}^{c_{2}}}$.

We iterate until doing $c_{t}$ blow-ups of 
\[
\left(z_{1}^{(t-1)},\dots,z_{g+s}^{(t-1)},x_{t}\right).
\]
 So we transformed $z_{i}$ in $z_{i}^{(t)}$ which is of the form
$\frac{z_{i}}{a_{1}}$.

Then $a_{1}$ divides all the $z_{i}^{(t)}$, and so all the $h_{i}^{(t)}$
and the $y_{i}^{(t)}$. The $b_{i,j}$ are elements of $\widehat{R}_{H}$,
so after this sequence of blow-ups, since the strict transform of
$H$ is generated by the $y_{i}^{(t)}$, we have that $a_{1}$ divides
all the $b_{i,j}$, and the proof is finished.
\end{proof}
\newpage{}

\section{Conclusion.}

We know are going to give the principal results of this part. First
we recall a fundamental result of Novacoski and Spivakovsky (\cite{NS2}).
\begin{thm}
Let $S$ be a noetherian local ring. If the local uniformization Theorem
is true for every valuation of rank $1$ centered in $S$, then it
is true for any valuation centered in $S$.
\end{thm}

So we just have to consider valuations of rank $1$.
\begin{thm}
\label{thm:rang 1}Let $S$ be a noetherian equicharacteristic quasi
excellent singular local ring of characteristic zero. We consider
$\mu$ a valuation of rank $1$ centered in $S$. 

There exists a formal framed sequence 
\[
\left(S,u\right)\to\dots\to\left(S_{i},u^{\left(i\right)}\right)\to\dots
\]
 such that for $j$ big enough, $S_{j}$ is regular and for every
element $s$ of $S$, there exists $i$ such that in $S_{i}$, $s$
is a monomial.
\end{thm}

\begin{proof}
We consider $\widehat{S}$ the formal completion of $S$ and $H$
its implicit prime ideal. By Cohen structure Theorem, there exists
an epimorphism $\Phi$ from a complete regular local ring $R$ in
$\widehat{S}$. We consider $\overline{H}$ the preimage of $H$ in
$R$. We extend now $\mu$ to a valuation $\nu$ centered in $R$
by composition with a valuation centered in $\overline{H}$.

By Proposition \ref{prop:localise regulier} we know that $\widehat{S}_{H}$
is regular, and by Theorem \ref{thm: reduction quotient} it is enough
to show that $\frac{\widehat{S}}{H}$ is also regular.

We know that $\frac{\widehat{S}}{H}\simeq\frac{R}{\overline{H}}$,
so we just have to regularize $\frac{R}{\overline{H}}$. We conclude
with Theorem \ref{thm:generalisation de la monomialisation-1}.
\end{proof}
Now we prove the principal result of this part: the simultaneous embedded
local uniformization for local noetherian quasi excellent equicharacteristic
rings.
\begin{thm}
\label{thm:BAMDANSTAFACE}Let $R$ be a local noetherian quasi excellent
complete regular ring and $\nu$ be a valuation centered in $R$.

Assume that $\nu$ is of rank $1$ or $2$ but composed of a valuation
$\left(f\right)$-adic where $f$ is an irreducible element of $R$.
We assume $\mathrm{car}\left(k_{\nu}\right)=0$. 

There exists a formal framed sequence 
\[
\left(R,u\right)\to\dots\to\left(R_{l},u^{\left(l\right)}\right)\to\dots
\]
 such that for every element $g$ of $R$, there exists $i$ such
that in $R_{i}$, $g$ is a monomial.
\end{thm}

\begin{proof}
We consider the ring $A=\frac{R}{\left(f\right)}$. The valuation
$\nu$ is of rank $2$ composed of valuation $\left(f\right)$-adic,
so $\nu$ can be written $\mu\circ\theta$ where $\theta$ is the
valuation $\left(f\right)$-adic.

So we have a valuation $\mu$ centered in $A$ of rank $1$. By Theorem
\ref{thm:rang 1}, we can regularize $A$, and so there exists a local
framed sequence $\left(R,u\right)\to\dots\to\left(R_{i},u^{(i)}\right)$
such that in $R_{i}$, $f$ is a monomial. In $R_{i}$, we also have
that every element $g$ of $R$ can be written $g=\left(u_{n}^{\left(i\right)}\right)^{a}h$
where $u_{n}^{\left(i\right)}$ is the strict transform of $f$ and
$h$ is not divisible by $u_{n}^{\left(i\right)}$. We apply another
time Theorem \ref{thm:rang 1} to construct a local framed sequence
which monomialize $h$. This completes the proof.
\end{proof}
\begin{cor}
We keep the same notations and hypothesis as in the previous Theorem.

Then $\lim\limits _{\rightarrow}R_{i}$ is a valuation ring.
\end{cor}

\begin{rem}
The restriction on the rank of the valuation was setted to give an
autosufficient proof. Otherwise, there exists a countable sequence
of polynomials $\chi_{i}$ such that every $\nu$-ideal $P_{\beta}$
is generated by a subset of the $\chi_{i}$. Assume the embedded local
uniformization Theorem.

Then there exists a local (respectively formal) framed sequence $\left(R,u\right)\to\dots\to\left(R_{i},u^{\left(i\right)}\right)\to\dots$
that has following properties:
\begin{enumerate}
\item For $i$ big enough, $R_{i}$ is regular.
\item For every finite set $\left\{ f_{1},\dots,f_{s}\right\} \subseteq\mathfrak{m}$
there exists $i$ such that in $R_{i}$, every $f_{j}$ is a monomial
and $f_{1}\mid f_{j}$.
\end{enumerate}
Then for every element $g$ in $R$, there exists $i$ such that in
$R_{i}$, $g$ is a monomial.
\end{rem}

\printindex{}

\bibliographystyle{amsplain} 

\begin{thebibliography}{99}
\bibitem{A1} Shreeram Abhyankar, \emph{Local uniformization on algebraic surfaces over ground fields of characteristic $p\neq 0$} Annals of Mathematics, Second Series, 63 (3): 491-526
\bibitem{A2} Shreeram Abhyankar, \emph{Resolution of Singularities of Embedded Algebraic Surfaces} Academic Press, New York and London, 1966.
\bibitem{A3} S. Abhyankar, \emph{Reduction to multiplicity less than $p$ in a $p$-cyclic extension of a two dimensional regular local ring ($p =$ characteristic of the reside field)}, Math. Annalen 154 28 - 55 (1964)
\bibitem{A4} S. Abhyankar, \emph{An algorithm on polynomials in one indeterminate with coefficients in a two dimensional regular local domain,} Annali di Matematica Pura ed Applicata 71 25 - 60 (1966)
\bibitem{A5} S. Abhyankar, \emph{Uniformization in a $p$-cyclic extension of a two dimensional regular local domain of residue field of characteristic $p$,} Festschrift zur Ged\"achtnisfeier f\"ur Karl Weierstrass 1815 - 1965, Wissenschaftliche Abhandlungen des Landes Nordrhein-Westfalen 33 Westdeutscher Verlag, K\"oln und Opladen 243 - 317 (1966)
\bibitem{A6} S. Abhyankar, \emph{Nonsplitting of valuations in extensions of two dimensional regular local domains,} Math. Annalen 170 87 - 144 (1967)
\bibitem{BEV1} A. Benito, S. Encinas and O. Villamayor, \emph{Some natural properties of constructive resolution of singularities}, Asian J. Math 15, no 2 141-192 (2011)
\bibitem{B} B. Bennett, \emph{On the characteristic functions of a local ring} Ann. of Math. 91 25-87 (1970)
\bibitem{BDMV} E. Bierstone, S. Da Silva, P. Milman F. Vera Pacheco, \emph{Desingularization by blowings-up avoiding simple normal crossings}, Proc. Amer. Math. Soc. 142, no. 12 4099--4111 (2014)
\bibitem{BGMW} E. Bierstone, D. Grigoriev, P. Milman and J. W\l odarczyk, \emph{Effective Hironaka resolution and its complexity}, Asian J. Math. 15, no. 2 193--228 (2011) 
\bibitem{BM} Edward Bierstone et Pierre D. Milman, \emph{Local resolution of singularities.} In Real analytic and algebraic geometry (Trento, 1988), volume 1420 of Lecture Notes in Math., pages 42-64. Springer, Berlin, 1990.
\bibitem{BM3} E. Bierstone and P. Milman, \emph{A simple constructive proof of canonical resolution of singularities,} Effective Methods in Algebraic Geometry, Progress in Math. Birkh\"{a}user Boston  94 11--30 (1991)
\bibitem{BM4} E. Bierstone and P. Milman, \emph{Canonical desingularization in characteristic zero by blowing up the maximum strata of a local invariant}, Invent. Math. 128, 2 207--302 (1997)
\bibitem{BM8} E. Bierstone and P. Milman, \emph{Resolution except for minimal singularities I}, Adv. Math. 231, no. 5 3022-3053 (2012)
\bibitem{BMT} E. Bierstone, P. Milman and M. Temkin, \emph{$\mathbb Q$-universal desingularization}, Asian J. Math. 15, no. 2 229--249 (2011)
\bibitem{BOU} Nicolas Bourbaki, \emph{Eléments de mathématiques : Algèbre commutative}, Masson, 1985.
\bibitem{BrV1} A. Bravo and O.E. Villamayor, \emph{Strengthening the Theorem of Embedded desingularization}, Math. Res. Letters 8 79--90 (2001)
\bibitem{BrV2} A. Bravo and O. Villamayor, \emph{A strengthening of resolution of singularities in characteristic zero}, Proc. London. Math. Soc. (3), 86 (2) 327--357 (2003)
\bibitem{BV1} E. Bierstone and F. Vera Pacheco, \emph{Resolution of singularities of pairs preserving semi-simple normal crossings}, Rev. R. Acad. Cienc. Exactas F\'\i s. Nat. Ser. A Math. RACSAM 107, no. 1 159-188 (2013)
\bibitem{BV2} E. Bierstone and F. Vera Pacheco, \emph{Desingularization preserving stable simple normal crossings}, Israel J. Math. 206, no. 1 233--280 (2015) 
\bibitem{CP1} Vincent Cossart et Olivier Piltant, \emph{Resolution of singularities of threefolds in positive characteristic. I. Reduction to local uniformization on Artin-Schreier and purely inseparable coverings.} Journal of Algebra, 320 (3): 1051-1082
\bibitem{CP2} Vincent Cossart et Olivier Piltant, \emph{Resolution of singularities of threefolds in positive characteristic. II} Journal of Algebra, 321 (7): 1836-1976
\bibitem{CP3} Vincent Cossart et Olivier Piltant, \emph{Resolution of Singularities of Arithmetical Threefolds II.} arXiv:1412.0868.
\bibitem{CRS} Felipe Cano, Claude Roche et Mark Spivakovsky, \emph{Reduction of singularities of three-dimensional  line  foliations}  Revista de la Real Academia de Ciencias Ex-actas,  Fisicas  y  Naturales.  Serie  A.  Matematicas108,  Issue  1(2014),pages 221-258.
\bibitem{CSch} V. Cossart, B. Schober, \emph{A strictly decreasing invariant for resolution of singularities in dimension two,} arXiv:1411.4452v2 (2014)
\bibitem{Cu1} S. D. Cutkosky, {\em Resolution of singularities}, AMS, Providence, RI, Graduate Studies in Mathematics 63 (2004)
\bibitem{Cu2} S. D. Cutkosky, {\em Resolution of singularities for 3-folds in positive characteristic}, American Journal of Mathematics 131, no 1 59-127 (2009) 
\bibitem{DMS} Julie Decaup, Wael Mahboub et Mark Spivakovsky, \emph{Abstract key polynomials and comparison theorems with the key polynomials of Mac Lane-Vaquié} arXiv:611.06392.
\bibitem{EH}  Santiago Encinas et Herwig Hauser, \emph{Strong resolution of singularities in characteristic zero.} Comment. Math. Helv., 77(4) :821-845, 2002.
\bibitem{EIS} David Eisenbud, \emph{Commutative Algebra with a View Toward Algebraic Geometry}, Springer, 1995.
\bibitem{EV} Santiago Encinas et Orlando Villamayor, \emph{A new proof of desingularization over fields of characteristic zero.} In Proceedings of the International Conference on Algebraic Geometry and Singularities (Sevilla, 2001), volume 19, pages 339-353, 2003.
\bibitem{EV1} S. Encinas and O. Villamayor, \emph{Good points and constructive resolution of singularities,} Acta Mathematica 181 109-158 (1998)
\bibitem{H} Heisuke Hironaka, \emph{Resolution of singularities of an algebraic variety over a field of characteristic zero.} I, II. Ann. of Math. (2) 79 (1964), 109-203 ; ibid. (2), 79 :205-326, 1964
\bibitem{H3} Hironaka, H, \emph{Desingularization of excellent surfaces,} Notes by B. Bennett at the Conference on Algebraic Geometry, Bowdoin 1967. Reprinted in: Cossart, V., Giraud, J., Orbanz, U.: Resolution of surface singularities. Lecture Notes in Math. 1101, Springer (1984)
\bibitem{HMOS} Herrera Govantes, Mahboub, Olalla Acosta et Spivakovsky, \emph{Key polynomials for simple extensions of valued fields}, arXiv:1406.0657v3
\bibitem{HOST}  Herrera Govantes, Olalla Acosta, Spivakovsky et Teissier, \emph{Extending a valuation centered in a local domain to the formal completion} Proc. LMS105, Issue 3, 571-621, 2012.
\bibitem{L} J. Lipman, \emph{Desingularization of two-dimensional schemes,} Ann. Math.107 151--207 (1978)
\bibitem{M} Matsumura, \emph{Commutative Algebra} Benjamin/Cummings Publishing Co., Reading, Mass., 1970.
\bibitem{MCL} Saunders Mac Lane, \emph{A construction for absolute values in polynomial rings}, Transactions of the AMS, vol 40 (1936), pages 363-395.
\bibitem{MAH} Wael Mahboub, \emph{Une construction explicite de polynômes-clefs pour des valuations de rang fini}, Thèse, Université Toulouse 3, 2013.
\bibitem{NS}  Novacoski et Spivakovsky, \emph{Key polynomials and pseudo-convergent sequences} Journal of Algebra 495 pages 199-219, 2018.
\bibitem{NS2} Novacoski et Spivakovsky, \emph{Reduction of local uniformization to the case of rank one valuations for rings with zero divisors.} Valuation Theory in Interaction, EMS Series of Congress Reports, European Mathematical Society, 2014, pp. 404-431.
\bibitem{te} Teissier, \emph{Valuations,  deformations  and  toric  geometry},  Proceedings  ofthe Saskatoon Conference and Workshop on valuation theory, Vol II, F-V.Kuhlmann, S. Kuhlmann, M. Marshall, editors, Fields Institute Communi-cations,33(2003), 361-459.
\bibitem{Tem1} M. Temkin, \emph{Desingularization of quasi-excellent schemes of characteristic zero,} Adv. Math. 219 488--522 (2008)
\bibitem{Tem2}  M. Temkin, \emph{Absolute desingularization in characteristic zero,} Motivic integration and its interactions with model theory and non-archimedean geometry, Volume II London Math. Soc. Lecture Note Ser. 384 213--250 (2011)
\bibitem{Tem3} M. Temkin, \emph{Functorial desingularization of quasi-excellent schemes in characteristic zero: the non-embedded case,}  Duke Journal of Mathematics 161 2208--2254 (2012) 
\bibitem{V} Michel Vaquié, \emph{Valuations and local uniformization} Advanced Studies in Pure Mathematics 43, 2006. Singularity Theory and Its Applications, p477-527
\bibitem{VAQ1} Michel Vaquié, \emph{Extension d'une valuation}, Trans. Amer. Math. Soc. 359 (2007), no.7, 3439-3481.
\bibitem{VAQ2} Michel Vaquié, \emph{Valuations}, dans \emph{Resolution of Singularities}, Progr. in math. 181, 2000. MR1748635 (2001i:13005).
\bibitem{Vi} Orlando Villamayor, \emph{Constructiveness of Hironaka's resolution.} Ann. Sci. École Norm. Sup. (4), 22(1) :1-32, 1989.
\bibitem{Wa} Robert J. Walker, \emph{Reduction of the Singularities of an Algebraic Surface.} Annals of Mathematics, Second Series, 36 (2): 336-365
\bibitem{Wl} Jaroslaw Wlodarczyk, \emph{Simple Hironaka resolution in characteristic zero.} J. Amer. Math. Soc., 18(4) :779-822, 2005.
\bibitem{Z} Oscar Zariski, \emph{The reduction of the singularities of an algebraic surface.} Ann. of Math. (2), 40 :639-689, 1939.
\bibitem{Z2} Oscar Zariski, \emph{Local  uniformization  theorem  on  algebraic  varieties.} Ann. of Math. 41 :852-896, 1940.
\bibitem{Z3} Oscar Zariski, \emph{Reduction  of  singularities  of  algebraic  three  dimensional  varieties} Ann. of Math.45, 472-542, 1944.
\bibitem{ZS} Oscar Zariski et Pierre Samuel, \emph{Commutative Algebra, Volume 2}, Graduate Texts in Mathematics, 1960.
\end{thebibliography}

\end{document}